\documentclass[fontsize=12pt,a4paper,headings=normal,
twoside=false,leqno,parskip=half-,abstract=true]{scrartcl}
\usepackage[english]{babel}
\usepackage[utf8]{inputenc}
\usepackage{hyperref}
\hypersetup{
 pdftitle={Sturm 3-ball global attractors 3: Examples},
 pdfauthor={Bernold Fiedler, Carlos Rocha},
 colorlinks=true,
 linkcolor=blue,
 citecolor=blue,
 filecolor=blue,
 urlcolor=blue}

\usepackage{graphicx}
\usepackage[format=plain,labelfont=bf,font=small]{caption}
\usepackage{xcolor}
\usepackage[arrow, matrix, curve]{xy}
\usepackage{float}

\usepackage{caption}
\usepackage{tabulary}
\usepackage{array}
\newcolumntype{N}[1]{>{\centering\arraybackslash}m{#1}}

\usepackage{amsmath,amsthm}
\usepackage{amssymb} 
\newcommand{\transv}{\mathrel{\text{\tpitchfork}}}
\makeatletter
\newcommand{\tpitchfork}{%
  \vbox{
    \baselineskip\z@skip
    \lineskip-.52ex
    \lineskiplimit\maxdimen
    \m@th
    \ialign{##\crcr\hidewidth\smash{$-$}\hidewidth\crcr$\pitchfork$\crcr}
  }%
}
\makeatother
\usepackage{latexsym}

\usepackage[notref,notcite,color,final 
]{showkeys}

\definecolor{refkey}{rgb}{1,0,0}
\definecolor{labelkey}{rgb}{1,0,0}

\usepackage{tikz}

\usepackage[textwidth=2cm,textsize=small,backgroundcolor=none]{todonotes}

  \mathchardef\ordinarycolon\mathcode`\:
  \mathcode`\:=\string"8000
  \begingroup \catcode`\:=\active
    \gdef:{\mathrel{\mathop\ordinarycolon}}
  \endgroup

\theoremstyle{plain}
\newtheorem{thm}{Theorem}[section]
\newtheorem{lem}[thm]{Lemma}

\newtheorem{prop}[thm]{Proposition}
\newtheorem{cor}[thm]{Corollary}
\newtheorem{defi}[thm]{Definition}

\begin{document}

\title{{\LARGE{Sturm 3-ball global attractors 3:\\
Examples of Thom-Smale complexes}}}

\author{
 \\
{~}\\
Bernold Fiedler* and Carlos Rocha**\\
\vspace{2cm}}

\date{version of \today}
\maketitle
\thispagestyle{empty}

\vfill

*\\
Institut für Mathematik\\
Freie Universität Berlin\\
Arnimallee 3\\ 
14195 Berlin, Germany\\
\\
**\\
Center for Mathematical Analysis, Geometry and Dynamical Systems\\
Instituto Superior T\'ecnico\\
Universidade de Lisboa\\
Avenida Rovisco Pais\\
1049--001 Lisbon, Portugal\\


\newpage
\pagestyle{plain}
\pagenumbering{roman}
\setcounter{page}{1}

\begin{abstract}
Examples complete our trilogy on the geometric and combinatorial characterization of global Sturm attractors $\mathcal{A}$ which consist of a single closed 3-ball.
The underlying scalar PDE is parabolic,
$$ u_t = u_{xx} + f(x,u,u_x)\,, $$
on the unit interval $0 < x<1$ with Neumann boundary conditions.
Equilibria $v_t=0$ are assumed to be hyperbolic.\\
\newline
Geometrically, we study the resulting Thom-Smale dynamic complex with cells defined by the fast unstable manifolds of the equilibria. 
The Thom-Smale complex turns out to be a regular cell complex.
In the first two papers we characterized 3-ball Sturm attractors $\mathcal{A}$ as 3-cell templates $\mathcal{C}$. 
The characterization involves bipolar orientations and hemisphere decompositions which are closely related to the geometry of the fast unstable manifolds.\\
\newline
An equivalent combinatorial description was given in terms of the Sturm permutation, alias the meander properties of the shooting curve for the equilibrium ODE boundary value problem.
It involves the relative positioning of extreme 2-dimensionally unstable equilibria at the Neumann boundaries $x=0$ and $x=1$, respectively, and the overlapping reach of polar serpents in the shooting meander.\\
\newline
In the present paper we apply these descriptions to explicitly enumerate all 3-ball Sturm attractors $\mathcal{A}$ with at most 13 equilibria.
We also give complete lists of all possibilities to obtain solid tetrahedra, cubes, and octahedra as 3-ball Sturm attractors with 15 and 27 equilibria, respectively. For the remaining Platonic 3-balls, icosahedra and dodecahedra, we indicate a reduction to mere planar considerations as discussed in our previous trilogy on planar Sturm attractors.

\end{abstract}

\newpage
\tableofcontents


\newpage
\pagenumbering{arabic}
\setcounter{page}{1}

\section{Introduction}
\label{sec1}

\numberwithin{equation}{section}
\numberwithin{figure}{section}
\numberwithin{table}{section}

For our general introduction we first follow \cite{firo3d-1, firo3d-2} and the references there.
\emph{Sturm global attractors} $\mathcal{A}_f$ are the global attractors of scalar parabolic equations
	\begin{equation}
	u_t = u_{xx} + f(x,u,u_x)
	\label{eq:1.1}
	\end{equation}
on the unit interval $0<x<1$.
Just to be specific we consider Neumann boundary conditions $u_x=0$ at $x \in \{0,1\}$.
Standard semigroup theory provides local solutions $u(t,x)$ for $t \geq 0$ and given initial data at time $t=0$, in suitable Sobolev spaces $u(t, \cdot) \in X \subseteq C^1 ([0,1], \mathbb{R})$.
Under suitable dissipativeness assumptions on $f \in C^2$, any solution eventually enters a fixed large ball in $X$.
For large times $t$, in fact, that large ball of initial conditions itself limits onto the maximal compact and invariant subset $\mathcal{A}_f$ of $X$ which is called the global attractor. See \cite{he81, pa83, ta79} for a general PDE background, and \cite{bavi92, chvi02, edetal94, ha88, haetal02, la91, ra02, seyo02, te88} for global attractors in general.

Equilibria $v = v(x)$ are time-independent solutions, of course, and hence satisfy the ODE
	\begin{equation}
	0 = v_{xx} + f(x,v,v_x)
	\label{eq:1.3}
	\end{equation} 
for $0\leq x \leq 1$, again with Neumann~boundary.
Here and below we assume that all equilibria $v$ of \eqref{eq:1.1}, \eqref{eq:1.3} are \emph{hyperbolic}, i.e. without eigenvalues (of) zero (real part) of their linearization.
Let $\mathcal{E} = \mathcal{E}_f \subseteq \mathcal{A}_f$ denote the set of equilibria.
Our generic hyperbolicity assumption and dissipativeness of $f$ imply that $N$:= $|\mathcal{E}_f|$ is odd.

It is known that \eqref{eq:1.1} possesses a Lyapunov~function, alias a variational or gradient-like structure, under separated boundary conditions;  see \cite{ze68, ma78, mana97, hu11, fietal14}. In particular, the global attractor consists of equilibria and of solutions $u(t, \cdot )$, $t \in \mathbb{R}$, with forward and backward limits, i.e.
	\begin{equation}
	\underset{t \rightarrow -\infty}{\mathrm{lim}} u(t, \cdot ) = v\,,
	\qquad
	\underset{t \rightarrow +\infty}{\mathrm{lim}} u(t, \cdot ) = w\,.
	\label{eq:1.2}
	\end{equation}
In other words, the $\alpha$- and $\omega$-limit sets of $u(t,\cdot )$ are two distinct equilibria $v$ and $w$.
We call $u(t, \cdot )$ a \emph{heteroclinic} or \emph{connecting} orbit, or \emph{instanton},  and write $v \leadsto w$ for such heteroclinically connected equilibria. 
See fig.~\ref{fig:1.0}(a) for a simple 3-ball example with $N=9$ equilibria.

\begin{figure}[p!]
\centering \includegraphics[width=0.9\textwidth]{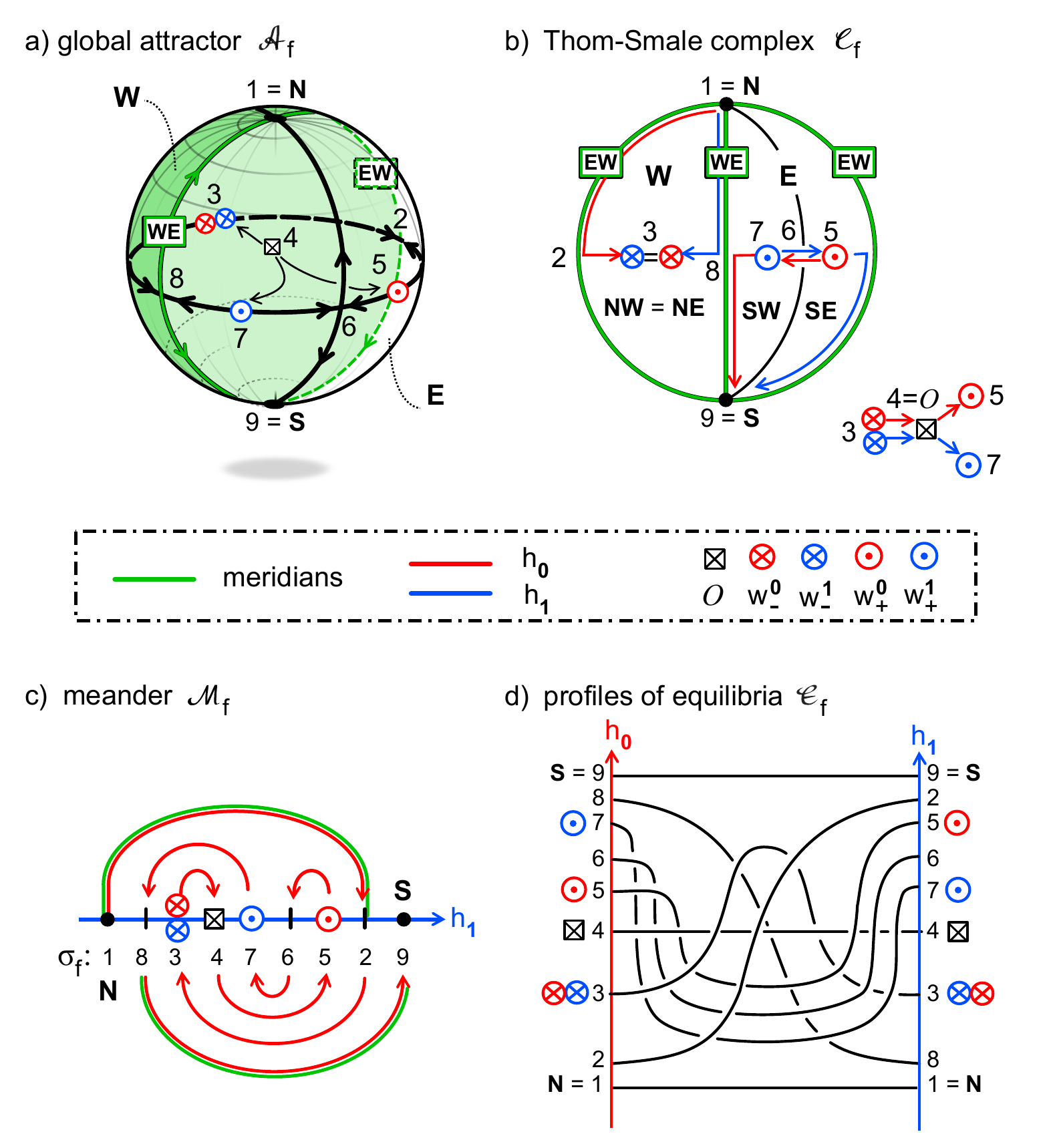}
\caption{\emph{
Example of a Sturm 3-ball global attractor $\mathcal{A}_f = clos W^u(\mathcal{O})$; see attractor $9.3^2$, alias case 2, of figs.~\ref{fig:6.3},~\ref{fig:6.4}, and case 2, $5.2|7.2^2$ of table \ref{tbl:6.5}. Equilibria are labeled as $\mathcal{E}_f=\{1,\ldots,9\}$. The previous papers \cite{firo3d-1,firo3d-2} established the equivalence of the viewpoints (a)--(d).
(a) The Sturm global attractor $\mathcal{A}_f$, 3d view, including the location of the poles $\textbf{N}$, $\textbf{S}$, the (green) meridians $\textbf{WE}$, $\textbf{EW}$, the central equilibrium $\mathcal{O}$ and the hemispheres $\textbf{W}$ (green), $\textbf{E}$.
(b) The dynamic Thom-Smale complex $\mathcal{C}_f$ of the boundary sphere $\Sigma^2 = \partial c_\mathcal{O}$, including the Hamiltonian SZS-pair of paths $(h_0,h_1)$, (red/blue), induced by the bipolar orientation of the 1-skeleton $\mathcal{C}_f^1$. 
The right and left boundaries denote the same $\textbf{EW}$ meridian and have to be identified. 
See fig.~\ref{fig:1.1} for the general case.
(c) The Sturm meander $\mathcal{M}_f$ of the global attractor $\mathcal{A}_f$. 
The meander $\mathcal{M}_f$ is the curve $a \mapsto (v,v_x)$, at $x=1$, which results from Neumann initial conditions $(v,v_x)=(a,0)$ at $x=0$ by shooting via the equilibrium ODE \eqref{eq:1.3}. 
Intersections of the meander with the horizontal $v$-axis indicate equilibria. 
See fig.~\ref{fig:1.2} for the general case.
(d) Spatial profiles $x\mapsto v(x)$ of the equilibria $v \in \mathcal{E}_f$. Note the different orderings of $v(x)$, by $h_0 = \mathrm{id}$ at the left boundary $x=0$, and by the Sturm permutation $\sigma_f = h_1 = (1\ 8\ 3\ 4\ 7\ 6\ 5\ 2\ 9)$ at the right boundary $x=1$. The same orderings define the meander in (c) and the Hamiltonian SZS-pair $(h_0,h_1)$ in the Thom-Smale complex (b).
}}
\label{fig:1.0}
\end{figure}

We attach the name of \emph{Sturm} to the PDE \eqref{eq:1.1}, and to its global attractor $\mathcal{A}_f$. This refers to a crucial nodal property of its solutions, which we express by the \emph{zero number} $z$.
Let $0 \leq z (\varphi) \leq \infty$ count the number of (strict) sign changes of $\varphi : [0,1] \rightarrow \mathbb{R}, \, \varphi \not\equiv 0$.
Then
	\begin{equation}
	t \quad \longmapsto \quad z(u^1(t, \cdot ) - u^2(t, \cdot ))\,
	\label{eq:1.4}
	\end{equation}
is finite and nonincreasing with time $t$, for $t>0$ and any two distinct solutions $u^1$, $u^2$ of \eqref{eq:1.1}.
Moreover $z$ drops strictly with increasing $t$, at any multiple zero of $x \longmapsto u^1(t_0 ,x) - u^2(t_0 ,x)$; see \cite{an88}.
See Sturm \cite{st1836} for a linear autonomous version. For a first introduction see also \cite{ma82, brfi88, fuol88, mp88, brfi89, ro91, fisc03, ga04} and the many references there.
As a convenient notational variant of the zero number $z$, we also write
	\begin{equation}
	z(\varphi) = j_{\pm}
	\label{eq:1.4+}
	\end{equation}
to indicate $j$ strict sign changes of $\varphi$, by $j$, and $\pm \varphi (0) >0$, by the index $\pm$.
For example $z(\pm \varphi_j) = j_{\pm}$, for the $j$-th Sturm-Liouville eigenfunction $\varphi_j$.

The dynamic consequences of the Sturm~structure are enormous.
In a series of papers, we have given a combinatorial description of Sturm global attractors $\mathcal{A}_f$; see \cite{firo96, firo99, firo00}.
Define the two \emph{boundary orders} $h^f_0, h^f_1$: $\lbrace 1, \ldots, N \rbrace \rightarrow \mathcal{E}_f$ of the equilibria such that
	\begin{equation}
	h^f_\iota (1) < h^f_\iota (2) < \ldots < h^f_\iota (N) \qquad \text{at}
	\qquad x=\iota \in \{0,1\}\,.
	\label{eq:1.5}
	\end{equation}
See fig.~\ref{fig:1.0}(d) for an illustration with $N=9$ equilibrium profiles, $\mathcal{E}_f = \{1,\ldots,9\}, \ h_0^f = \mathrm{id},\ h_1^f = (1\ 8\ 3\ 4\ 7\ 6\ 5\ 2\ 9)$.	

The combinatorial description is based on the \emph{Sturm~permutation} $\sigma_f \in S_N$ which was introduced by Fusco and Rocha in \cite{furo91} and is defined as
	\begin{equation}
	\sigma_f:= (h^f_0)^{-1} \circ h^f_1\,.
	\label{eq:1.6}
	\end{equation}
Using a shooting approach to the ODE boundary value problem \eqref{eq:1.3}, the Sturm~permutations $\sigma_f \in S_N$ have been characterized as \emph{dissipative Morse meanders} in \cite{firo99}; see also \eqref{eq:1.20}--\eqref{eq:1.22b} below.
In \cite{firo96} we have shown how to determine which equilibria $v, w$ possess a heteroclinic orbit connection \eqref{eq:1.2}, explicitly and purely combinatorially from $\sigma_f$.

More geometrically, global Sturm attractors $\mathcal{A}_f$ and $\mathcal{A}_g$ with the same Sturm permutation $\sigma_f = \sigma_g$ are $C^0$ orbit-equivalent \cite{firo00}.
For $C^1$-small perturbations, from $f$ to $g$, this global rigidity result is based on the $C^0$ structural stability of Morse-Smale systems; see e.g. \cite{pasm70} and \cite{pame82}.
It is the Sturm property \eqref{eq:1.4} which implies the Morse-Smale property, for hyperbolic equilibria.
Stable and unstable manifolds $W^u(v_-)$, $W^s(v_+)$, which intersect precisely along heteroclinic orbits $v_- \leadsto v_+$, are in fact automatically transverse:
$W^u(v_-) \transv W^s(v_+)$.
See \cite{he85, an86}.
In the Morse-Smale setting, Henry already observed, that a heteroclinic orbit $v_- \leadsto v_+$ is equivalent to $v_+$ belonging to the boundary $\partial W^u(v_-)$ of the unstable manifold $W^u(v_-)$; see \cite{he85}.

More recently, we have pursued a more explicitly geometric approach.
Let us consider \emph{finite~regular} CW-\emph{complexes}
	\begin{equation}
	\mathcal{C} = \bigcup\limits_{v\in \mathcal{E}} c_v\,,
	\label{eq:1.7}
	\end{equation}
i.e. finite disjoint unions of \emph{cell~interiors} $c_v$ with additional gluing properties.
We think of the labels $v\in \mathcal{E}$ as \emph{barycenter} elements of $c_v$.
For CW-complexes we require the closures $\overline{c}_v$ in $\mathcal{C}$ to be the continuous images of closed unit balls $\overline{B}_v$ under \emph{characteristic maps}.
We call $\mathrm{dim}\,\overline{B}_v$ the dimension of the (open) cell $c_v$. 
For positive dimensions of $\overline{B}_v$ we require $c_v$ to be the homeomorphic images of the interiors $B_v$. 
For dimension zero we write $B_v := \overline{B}_v$ so that any 0-cell $c_v= B_v$ is just a point.
The \emph{m-skeleton} $\mathcal{C}^m$ of $\mathcal{C}$ consists of all cells of dimension at most $m$.
We require $\partial c_v := \overline{c}_v \setminus c_v \subseteq \mathcal{C}^{m-1}$ for any $m$-cell $c_v$.
Thus, the boundary $(m-1)$-sphere $S_v := \partial B_v = \overline{B}_v \setminus B_v$ of any $m$-ball  $B_v$, $m>0$, maps into the $(m-1)$-skeleton,
	\begin{equation}
	\partial B_v \quad \longrightarrow \quad \partial c_v \subseteq \mathcal{C}^{m-1}\,,
	\label{eq:1.8}
	\end{equation}
for the $m$-cell $c_v$, by restriction of the characteristic map.
The continuous map \eqref{eq:1.8} is called the \emph{attaching} (or \emph{gluing}) \emph{map}.
For \emph{regular} CW-complexes, in contrast, the characteristic maps $ \overline{B}_v  \rightarrow  \overline{c}_v $ are required to be homeomorphisms, up to and including the \emph{attaching} (or \emph{gluing}) \emph{homeomorphism}. 
We require the $(m-1)$-sphere $\partial{c_v}$  to be a sub-complex of $\mathcal{C}^{m-1}$. 
See \cite{frpi90} for some further background on this terminology.

The disjoint dynamic decomposition
	\begin{equation}
	\mathcal{A}_f = \bigcup\limits_{v \in \mathcal{E}_f} W^u(v) =: \mathcal{C}_f
	\label{eq:1.9}
	\end{equation}
of the global attractor $\mathcal{A}_f$ into unstable manifolds $W^u$ of equilibria $v$ is called the \emph{Thom-Smale complex} or \emph{dynamic complex}; see for example \cite{fr79, bo88, bizh92}.
In our Sturm setting \eqref{eq:1.1} with hyperbolic equilibria $v \in \mathcal{E}_f$, the Thom-Smale complex is a finite regular CW-complex.
The open cells $c_v$ are the unstable manifolds $W^ u (v)$ of the equilibria $v \in \mathcal{E}_f$.
The proof is closely related to the Schoenflies result of \cite{firo13}; see \cite{firo14} for a summary.

We can therefore define the \emph{Sturm~complex} $\mathcal{C}_f$ to be the regular Thom-Smale complex $\mathcal{C}_f$
of the Sturm global attractor $\mathcal{A}_f$, provided all equilibria $v \in \mathcal{E}_f$ are hyperbolic.
Again we call the equilibrium $v \in \mathcal{E}_f$ the \emph{barycenter} of the cell $c_v=W^u(v)$.
The dimension $i(v)$ of $c_v$ is called the \emph{Morse index} of $v$.
A planar Sturm complex $\mathcal{C}_f$, for example, is the regular Thom-Smale complex of a planar $\mathcal{A}_f$, i.e. of a Sturm global attractor for which all equilibria $v \in \mathcal{E}_f$ have Morse indices $i(v) \leq 2$.
See fig.~\ref{fig:1.0}(b) for the Sturm complex $\mathcal{C}_f$ of the Sturm global attractor $\mathcal{A}_f$ sketched in fig.~\ref{fig:1.0}(a).

Our main result, in the first two parts \cite{firo3d-1, firo3d-2} of the present trilogy, was a geometric and combinatorial characterization of those global Sturm attractors, which are the closure
	\begin{equation}
	\mathcal{A}_f = \text{clos } W^u (\mathcal{O})
	\label{eq:1.10}
	\end{equation}
of the unstable manifold $W^u$ of a single equilibrium $v = \mathcal{O}$ with Morse index $i(\mathcal{O}) =3$.
We call such an $\mathcal{A}_f$ a 3-\emph{ball Sturm attractor}.
Recall that we assume all equilibria $v_1, \ldots, v_N$ to be hyperbolic:
\emph{sinks} have Morse index $i=0$, \emph{saddles} have $i=1$, and \emph{sources}  $i=2$.
This terminology also applies when viewed within the flow-invariant and attracting boundary 2-sphere
	\begin{equation}
	\Sigma^2 = \partial W^u(\mathcal{O}):= \left(
	\text{clos } W^u(\mathcal{O})\right) \smallsetminus
	W^u (\mathcal{O})\,.
	\label{eq:1.11}
	\end{equation}
Correspondingly we call the associated cells $c_v = W^u(v)$ of the Thom-Smale cell complex, or of any regular cell complex, \emph{vertices}, \emph{edges}, and {\emph{faces}.
The graph of vertices and edges, for example, defines the 1-skeleton $\mathcal{C}^1$ of the 3-ball cell complex $\mathcal{C} = \bigcup_v \, c_v$.

Any abstractly prescribed regular 3-ball complex $\mathcal{C}$ possesses a realization as the Sturm dynamic complex
	\begin{equation}
	\mathcal{C}_f = \mathcal{C}
	\label{eq:cfc}
	\end{equation}
of a suitably chosen nonlinearity $f$ with Sturm 3-ball $\mathcal{A}_f$; see \cite{firo14}.
However, there may be many meander permutations $\sigma_f \neq \sigma_g$ which realize the same complex,
	\begin{equation}
	\mathcal{C}_f = \mathcal{C} = \mathcal{C}_g\,,
	\label{eq:cfcg}
	\end{equation}
up to homeomorphisms which preserve the cell structure.
In section~\ref{sec2} we review \emph{trivial equivalences} as a (trivial) cause:
$f,\ g$, and hence $\sigma_f,\ \sigma_g$, may be related by transformations $x \mapsto 1-x$ or $u \mapsto -u$.
But there are much more subtle causes for the phenomenon \eqref{eq:cfcg}, where even the cycle lengths of the Sturm permutations $\sigma_f,\ \sigma_g$ disagree.
The examples of sections~\ref{sec5} and~\ref{sec6} will realize Sturm 3-ball attractors $\mathcal{A}_f= \mathcal{C}_f$ with prescribed 3-ball complex $\mathcal{C}$, as in \eqref{eq:cfc}, and will provide lists of all realizing permutations $\sigma_f$, in the sense of \eqref{eq:cfcg}.
The comparatively modest example of fig.~\ref{fig:1.0} possesses nine equilibria. Depending on their Morse index, they serve as the barycenters of two 0-cells, three 1-cells, three 2-cells, and one 3-cell.
The example will reappear as case 2, also labeled $9.3^2$, in figs.~\ref{fig:6.3}, \ref{fig:6.4} and in table \ref{tbl:6.5} below.

Our results are crucially based on the disjoint \emph{signed hemisphere decomposition}
	\begin{equation}
	\partial W^u(v) =
	\bigcup\limits_{0\leq j< i(v)}^\centerdot
	\Sigma_\pm^j(v)
	\label{eq:1.9a}
	\end{equation}
of the topological boundary $\partial W^u= \partial c_v = \overline{c}_v \smallsetminus c_v$ of the unstable manifold $W^u(v)=c_v$, for any equilibrium $v$.
As in \cite[(1.19)]{firo3d-2} we define the hemispheres by their Thom-Smale cell decompositions
	 \begin{equation}
	 \Sigma_ \pm^j(v) :=
	 \bigcup\limits_{w\in \mathcal{E}_\pm^j(v)}^\centerdot
	 W^u(w)
	 \label{eq:1.9b}
	 \end{equation}
with the equilibrium sets
	\begin{equation}
	\mathcal{E}_\pm^j(v) := \lbrace w\in \mathcal{E}_f\,|\,z(w-v)=j_\pm \quad \text{and}
	\quad v\leadsto w\rbrace\,,
	\label{eq:1.9c}
	\end{equation}
for $0\leq j<i(v)$.
Equivalently, we may define the hemisphere decompositions, inductively, via the topological boundary $j$-spheres
	\begin{equation}
	\Sigma^j(v):=
	\bigcup\limits_{0\leq k<j}^\centerdot \Sigma_\pm^k(v)
	\label{eq:1.9d}
	\end{equation}
of the fast unstable manifolds $W^{j+1}(v)$.
Here $W^{j+1}(v)$ is tangent to the eigenvectors $\varphi_0, \ldots ,\varphi_j$ of the first $j+1$ unstable eigenvalues $\lambda_0 > \ldots > \lambda_j >0$ of the linearization at the equilibrium $v$.
See \cite{firo3d-1} for details.

For 3-ball Sturm attractors, the signed hemisphere decomposition \eqref{eq:1.9a} reads
	\begin{equation}
	\Sigma^2 = \partial W^u (\mathcal{O})=
	\bigcup\limits_{j=0}^2 \Sigma_\pm^j\,.
	\label{eq:1.13}
	\end{equation}
at $v=\mathcal{O}$ with Morse index $i(\mathcal{O})=3$.
See \eqref{eq:1.10}, \eqref{eq:1.11}.
Here $\Sigma_\pm^0 = \lbrace \mathbf{N}, \mathbf{S}\rbrace$ is the boundary of the one-dimensional fastest unstable manifold $W^1 = W^1(\mathcal{O})$, tangent to the positive eigenfunction $\varphi_0$ of the largest eigenvalue $\lambda_0$ at $\mathcal{O}$.
Indeed, solutions $t \mapsto u(t,x)$ in $W^1$ are monotone in $t$, for any fixed $x$.
Accordingly
	\begin{equation}
	z(\mathbf{N}- \mathcal{O}) = 0_-\,, 
	\quad z(\mathbf{S}- \mathcal{O}) = 0_+\,,
	\label{eq:1.14}
	\end{equation}
i.e. $\mathbf{N} < \mathcal{O} < \mathbf{S}$ for all $0\leq x\leq 1$.
The \emph{poles} $\mathbf{N},\mathbf{S}$ split the circle boundary $\Sigma^1 = \partial W^2 (\mathcal{O})$ of the 2-dimensional fast unstable manifold into the two \emph{meridian} half-circles $\Sigma_\pm^1$.
The circle $\Sigma^1$, in turn, splits the boundary sphere $\Sigma^2 = \partial W^u(\mathcal{O})$ of the whole unstable manifold $W^u$ of $\mathcal{O}$ into the two hemispheres $\Sigma_\pm^2$.
See fig.~\ref{fig:1.0}(b), for example.

\begin{figure}[t!]
\centering \includegraphics[width=\textwidth]{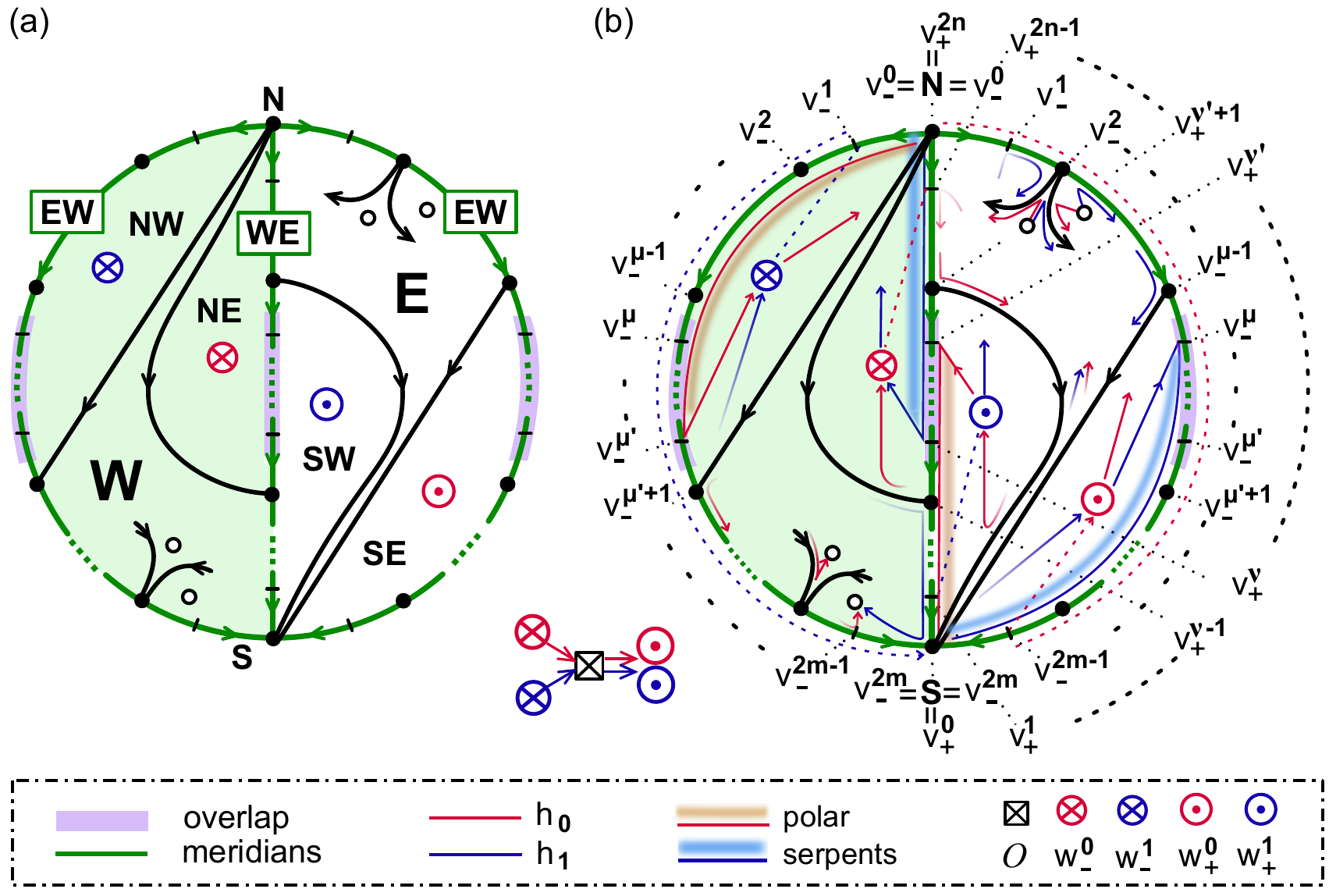}
\caption{\emph{
The 3-cell template, generalizing fig.~\ref{fig:1.0}(b). Shown is the 2-sphere boundary of the single 3-cell $c_\mathcal{O}$ with poles $\mathbf{N}$, $\mathbf{S}$, hemispheres $\mathbf{W}$ (green), $\mathbf{E}$, and separating meridians $\mathbf{EW}$, $\mathbf{WE}$ (both green).
The right and the left boundaries denote the same $\mathbf{EW}$ meridian and have to be identified.
Dots $\bullet$ are sinks, and small circles $\circ$ are sources.
(a) Note the hemisphere decomposition~(ii), the edge orientations~(iii) at meridian boundaries, and the meridian overlaps~(iv) of the $\mathbf{N}$-adjacent meridian faces $\otimes = w_-^\iota$ with their $\mathbf{S}$-adjacent counterparts $\odot =w_+^\iota$; see also \eqref{eq:1.24a}.
(b) The SZS-pair $(h_0,h_1)$ in a 3-cell template $\mathcal{C}$, with poles $\mathbf{N}, \mathbf{S}$, hemispheres $\mathbf{W}, \mathbf{E}$  and meridians $\mathbf{EW}, \mathbf{WE}$.
Dashed lines indicate the $h_\iota$-ordering of vertices in the closed hemisphere, when $\mathcal{O}$ and the other hemisphere are ignored, according to definition~\ref{def:2.4}(i).
The actual paths $h_\iota$ tunnel, from $w_ -^\iota \in  \mathbf{W}$ through the 3-cell barycenter $\mathcal{O}$, and re-emerge at $w_+^\iota \in  \mathbf{E}$, respectively.
Note the boundary overlap of the faces $\mathbf{NW}, \mathbf{SE}$ of $w_-^1, w_+^0$ from $v_-^{\mu-1}$ to $v_-^{\mu ' +1}$ on the $\mathbf{EW}$ meridian.
Similarly, the boundaries of the faces $\mathbf{NE}, \mathbf{SW}$ of $w_-^0, w_+^1$ overlap from $v_+^{\nu -1}$ to $v_+^{\nu ' +1}$ along $\mathbf{WE}$.
For many additional examples see sections~\ref{sec5} to \ref{sec7}. See also fig.~\ref{fig:2.3}.
}}
\label{fig:1.1}
\end{figure}

For the geometric characterization of 3-ball Sturm attractors $\mathcal{A}_f$ in \eqref{eq:1.10}, by their Thom-Smale dynamic complexes \eqref{eq:1.9}, we now drop all Sturmian PDE interpretations.
Instead we define 3-cell templates, abstractly, in the class of regular cell complexes $\mathcal{C}$ and without any reference to PDE or dynamics terminology.
See fig.~\ref{fig:1.1} for an illustration of the general case.

\begin{defi}\label{def:1.1}
A finite regular cell complex $\mathcal{C} = \bigcup_{v \in \mathcal{E}} c_v$ is called a \emph{3-cell template} if the following four conditions all hold.
\begin{itemize}
\item[(i)] $\mathcal{C} = \text{clos } c_{\mathcal{O}}= S^2 \,\dot{\cup}\, c_{\mathcal{O}}$ is the closure of a single 3-cell $c_{\mathcal{O}}$.
\item[(ii)] The 1-skeleton $\mathcal{C}^1$ of $\mathcal{C}$ possesses a \emph{bipolar orientation} from a pole vertex $\mathbf{N}$ (North) to a pole vertex $\mathbf{S}$ (South), with two disjoint directed \emph{meridian paths} $\mathbf{WE}$ and $\mathbf{EW}$ from $\mathbf{N}$ to $\mathbf{S}$.
The meridians decompose the boundary sphere $S^2$ into remaining \emph{hemisphere} components $\mathbf{W}$ (West) and $\mathbf{E}$ (East).
\item[(iii)] Edges are directed towards the meridians, in $\mathbf{W}$, and away from the meridians, in $\mathbf{E}$, at end points on the meridians other than the poles $\mathbf{N}$, $\mathbf{S}$.
\item[(iv)] Let $\mathbf{NE}$, $\mathbf{SW}$ denote the unique faces in $\mathbf{W}$, $\mathbf{E}$, respectively, which contain the first, last edge of the meridian $\mathbf{WE}$ in their boundary.
Then the boundaries of $\mathbf{NE}$ and $\mathbf{SW}$ overlap in at least one shared edge of the meridian $\mathbf{WE}$.

Similarly, let $\mathbf{NW}$, $\mathbf{SE}$ denote the unique faces in $\mathbf{W}$, $\mathbf{E}$, adjacent to the first, last edge of the other meridian $\mathbf{EW}$, respectively.
Then their boundaries overlap in at least one shared edge of $\mathbf{EW}$.
\end{itemize}
\end{defi}

We recall here that an edge orientation of the 1-skeleton $\mathcal{C}^1$ is called bipolar if it is without directed cycles, and with a single ``source'' vertex $\mathbf{N}$ and a single ``sink'' vertex $\mathbf{S}$ on the boundary of $\mathcal{C}$.
Here ``source'' and ``sink'' are understood, not dynamically but, with respect to edge direction.
To avoid any confusion with dynamic $i=0$ sinks and $i=2$ sources, below, we call $\mathbf{N}$ and $\mathbf{S}$ the North and South pole, respectively.
Again we refer to fig.~\ref{fig:1.0} for an illustrative example.

The hemisphere translation table between $\mathcal{A}_f$ and $\mathcal{C}_f= \mathcal{C}$ is, of course, the following:
	\begin{equation}
	\begin{aligned}
	(\Sigma_-^0, \Sigma_+^0) \quad &\mapsto \quad (\mathbf{N}, \mathbf{S})\\
	(\Sigma_-^1, \Sigma_+^1) \quad &\mapsto \quad 
	(\mathbf{EW}, 		\mathbf{WE})\\
	(\Sigma_-^2, \Sigma_+^2) \quad &\mapsto \quad (\mathbf{W}, \mathbf{E})
	\end{aligned}
	\label{eq:1.17}
	\end{equation}
Here $\Sigma_\pm^j$ abbreviates $\Sigma_\pm^j(\mathcal{O})$.

\begin{thm}\label{thm:1.2}
\cite[theorems~1.2 and~2.6]{firo3d-2}.
A finite regular cell complex $\mathcal{C}$ coincides with the Thom-Smale dynamic complex $c_v= W^u(v) \in \mathcal{C}_f$ of a 3-ball Sturm attractor $\mathcal{A}_f$ if, and only if, $\mathcal{C}$ is a 3-cell template, with the above translation of the hemisphere decomposition of $\partial W^u(\mathcal{O})$.
\end{thm}

In \cite[theorem~4.1]{firo3d-1} we proved that the dynamic complex $\mathcal{C}$:= $\mathcal{C}_f$ of a Sturm 3-ball $\mathcal{A}_f$ indeed satisfies properties (i)--(iv) of definition~\ref{def:1.1} on a 3-cell template.
In our example of fig.~\ref{fig:1.0} this simply means the passage (a) $\Rightarrow$ (b).
In general, the 3-cell property~(i) of $c_{\mathcal{O}} = W^u(\mathcal{O})$ is obviously satisfied.
The bipolar orientation~(ii) of the edges $c_v$ of the 1-skeleton, alias the one-dimensional unstable manifolds $c_v = W^u(v)$ of $i(v)=1$ saddles $v$, is simply the strict monotone ordering from the lowest equilibrium vertex $\Sigma_-^0 (v)$ in the closure $\bar{c}_v$ to the highest equilibrium vertex $\Sigma_+^0(v)$ in $\bar{c}_v$. 
The ordering is uniform for $0 \leq x \leq 1$, and holds at $x \in \{0,1\}$, in particular.
The meridian cycle is the boundary $\Sigma^1$ of the two-dimensional fast unstable manifold.
Properties (iii) and (iv) are far less obvious, at first sight.

The proof of the converse passage, say from fig.~\ref{fig:1.0}(b) to fig.~\ref{fig:1.0}(a), requires the design of a 3-ball Sturm attractor $\mathcal{A}_f$ with a prescribed 3-cell template $\mathcal{C}_f=\mathcal{C}$ for the signed hemisphere decomposition \eqref{eq:1.17}.
This has been achieved via the notion of a 3-meander template $\mathcal{M}, \sigma$ which we explain below.
Suffice it here to say that we introduced a construction of a suitable SZS-pair of Hamiltonian paths
	\begin{equation}
	h_\iota :\quad
	\lbrace 1, \ldots , N\rbrace \longrightarrow \mathcal{E}\,,
	\label{eq:1.19a}
	\end{equation}
for $\iota \in \{0,1\}$, i.e. a pair of bijections onto the barycentric vertex set $v\in \mathcal{E}$ of the given 3-cell template $\mathcal{C}= (c_v)_{v \in \mathcal{E}}$.
In fact we constructed the abstract paths $h_\iota$ in $\mathcal{C}$, for $\iota \in \{0,1\}$, by recipe or decree ex cathedra, such that the abstract permutation
	\begin{equation}
	\sigma:= h_0^{-1}\circ h_1
	\label{eq:1.19b}
	\end{equation}
is a dissipative Morse meander and hence, by \cite{firo96}, a Sturm permutation $\sigma=\sigma_f$ for some concrete nonlinearity $f$.
See fig.~\ref{fig:1.0}(b) for an example of an SZS-pair $(h_0,h_1)$.

More precisely, \cite[theorem~5.2]{firo3d-1} showed that the construction \eqref{eq:1.19b} of $\sigma$, ex cathedra, results in a 3-meander template, for any prescribed 3-cell template $\mathcal{C}$.
In the example of fig.~\ref{fig:1.0} this amounts to the passage (b) $\Rightarrow$ (c). 
In \cite[theorem~3.1]{firo3d-2} we showed that the resulting permutation $\sigma$ is a dissipative Morse meander, and hence is a Sturm permutation $\sigma=\sigma_f$, for some suitable nonlinearities $f$ with Sturm attractor $\mathcal{A}_f$.
In \cite[theorem~5.1]{firo3d-2} we showed that $\mathcal{A}_f$, thus constructed, is indeed a Sturm 3-ball.
In the example of fig.~\ref{fig:1.0} this amounts to the passage (c) $\Rightarrow$ (a). 
In \cite[theorem~2.6]{firo3d-2}, finally, we showed that the Thom-Smale dynamic Sturm complex $\mathcal{C}_f$ of $\mathcal{A}_f$ coincides with the prescribed 3-cell template $\mathcal{C}$, i.e. $\mathcal{C}_f = \mathcal{C}$, by a cell homeomorphism which, in addition, preserves the signed hemisphere translation table \eqref{eq:1.17}.
In particular the 3-cell template $\mathcal{C}=\mathcal{C}_f$ determines the Sturm permutation $\sigma=\sigma_f$ uniquely \cite[theorem~2.7]{firo3d-2}.
In the example of fig.~\ref{fig:1.0} this amounts to the passage (a) $\Rightarrow$ (b), and culminates in the equivalence of all three descriptions (a), (b), (c) of Sturm global attractors.

It remains to recall the two main concepts mentioned in the above proof of theorem~\ref{thm:1.2}:
meanders $\mathcal{M}$ and SZS-pairs $(h_0,h_1)$ of Hamiltonian paths in $\mathcal{C}$.
See fig.~\ref{fig:1.0}(b),(c) for illustration.

Abstractly, a \emph{meander} is an oriented planar $C^1$ Jordan curve $\mathcal{M}$ which crosses a positively oriented horizontal axis at finitely many points.
The curve $\mathcal{M}$ is assumed to run from Southwest to Northeast, asymptotically, and all $N$ crossings are assumed to be transverse. See \cite{ar88, arvi89} for this original notion. For a recent algebraically minded monograph and beautiful survey on meanders see \cite{ka17}.

Note $N$ is odd in our setting.
Enumerating the $N$ crossing points $v \in \mathcal{E}$, by $h_0$ along the meander $\mathcal{M}$ and by $h_1$ along the horizontal axis, respectively, we obtain two labeling bijections \eqref{eq:1.19a}.
We define the \emph{meander permutation} $\sigma \in S_N$ by \eqref{eq:1.19b}.
We call the meander $\mathcal{M}$ \emph{dissipative} if
	\begin{equation}
	\sigma(1) =1, \quad \sigma(N) =N
	\label{eq:1.20}
	\end{equation}
are fixed under $\sigma$.

For $\mathcal{M}$-adjacent crossings $v=h_0(j)$, $\Tilde{v}= h_0(j+1)$ we recursively define \emph{Morse numbers} $i_{\Tilde{v}}$, $i_v$ by
	\begin{equation}
	\begin{aligned}
	i_{h_0(1)} &:= \quad i_{h_0(N)} := 0\,,\\
	i_{h_0(j+1)} &:= \quad i_{h_0(j)}+(-1)^{j+1}\,
	\text{sign} (\sigma^{-1}(j+1)-\sigma^{-1}(j))\,.
	\end{aligned}
	\label{eq:1.22a}
	\end{equation}
Equivalently, by recursion along $h_1$:
	\begin{equation}
	\begin{aligned}
	i_{h_1(1)} &:= \quad i_{h_1(N)} := 0\,,\\
	i_{h_1(j+1)} &:= \quad i_{h_1(j)}+(-1)^{j+1}\,
	\text{sign} (\sigma(j+1)-\sigma(j))\,.
	\end{aligned}
	\label{eq:1.22b}
	\end{equation}
Note how the enumeration of intersections $v\in \mathcal{E}$ by $h_\iota$: $\lbrace 1, \ldots , N\rbrace \rightarrow \mathcal{E}$ depends on $h_\iota$, of course.
The Morse numbers $i_v$, however, only depend on the Sturm permutation $\sigma$ which defines the meander $\mathcal{M}$.
We call the meander $\mathcal{M}$ \emph{Morse}, if all $i_v$ are nonnegative:
	\begin{equation}
	i_v \geq 0\,.
	\label{eq:1.23a}
	\end{equation}

We call $\mathcal{M}$ \emph{Sturm meander}, if $\mathcal{M}$ is a dissipative Morse meander; see \cite{firo96}. 
Conversely, given any permutation $\sigma \in S_N$, we label $N$ crossings along the axis in the order of $\sigma$. Define an associated curve $\mathcal{M}$ of arcs over the horizontal axis which switches sides at the labels $\{1,\ldots,N\}$, in successive order. This fixes $h_0=\mathrm{id}$ and $h_1=\sigma$.
A \emph{Sturm permutation} $\sigma$ is a permutation such that the associated curve $\mathcal{M}$ is a Sturm meander.
The main paradigm of \cite{firo96} is the equivalence of Sturm meanders $\mathcal{M}$ with shooting curves $\mathcal{M}_f$ of the Neumann ODE problem \eqref{eq:1.3}.
In fact, the Neumann shooting curve is a Sturm meander, for any dissipative nonlinearity $f$ with hyperbolic equilibria.
Conversely, for any permutation $\sigma$ of a Sturm meander $\mathcal{M}$ there exist dissipative $f$ with hyperbolic equilibria such that $\sigma = \sigma_f$ is the Sturm permutation of $f$.
In that case, the intersections $v$ of the meander $\mathcal{M}_f$ with the horizontal $v$-axis are the boundary values of the equilibria $v \in \mathcal{E}_f$ at $x=1$, and the Morse number $i_v$ are the Morse indices $i(v)$:
	\begin{equation}
	i_v = \dim c_v = \dim W^u(v) = i(v) \geq 0\,.
	\label{eq:1.23d}
	\end{equation}
This allows us to identify
	\begin{align}
	\mathcal{E}_f &= \mathcal{E}\,;
	\label{eq:1.23b}\\
	h_\iota^f &= h_\iota\,;\label{eq:1.23c}
	\end{align}
For that reason we have used closely related notation to describe either case.

In particular, \eqref{eq:1.23c} justifies the terminology of \emph{sinks} $i_v =0$, \emph{saddles} $i_v =1$, and \emph{sources} $i_v=2$ for abstract Sturm meanders.
We insist, however, that our above definition~\eqref{eq:1.20}--\eqref{eq:1.22b} is completely abstract and independent of this ODE/PDE interpretation.

We return to abstract Sturm meanders $\mathcal{M}$ as in \eqref{eq:1.20}--\eqref{eq:1.22b}.
For example, consider the case $i_{\mathcal{O}} =3$ of a single intersection $v= \mathcal{O}$ with Morse number $3$.
Suppose $i_v \leq 2$ for all other Morse numbers. Then \eqref{eq:1.22a} implies $i=2$ for the two $h_0$-\emph{neighbors} $h_0(h_0^{-1} (\mathcal{O})\pm1)$ of $\mathcal{O}$ along the meander $\mathcal{M}$.
In other words, both these neighbors are sources.
The same statement holds true for the two $h_1$-\emph{neighbors} $h_1(h_1^{-1}(\mathcal{O})\pm 1)$ of $\mathcal{O}$ along the horizontal axis.
To fix notation, we denote these $h_\iota$-neighbors by
	\begin{equation}
	w_\pm^\iota:= h_\iota (h_\iota^{-1} (\mathcal{O}) \pm 1)\,,
	\label{eq:1.24a}
	\end{equation}
for $\iota \in \{0,1\}$.
The $h_\iota$-\emph{extreme sources} are the first and last source intersections $v$ of the meander $\mathcal{M}$ with the horizontal axis, in the order of $h_\iota$.

Reminiscent of cell template terminology, we call the extreme sinks $\mathbf{N} = h_0(1) = h_1(1)$ and $\mathbf{S} = h_0(N) = h_1(N)$ the (North and South) \emph{poles} of the Sturm meander $\mathcal{M}$.
A \emph{polar} $h_\iota$-\emph{serpent}, for $\iota \in \{0,1\}$, is a set of $v =h_\iota (j) \in \mathcal{E}$, for a maximal interval of integers $j$, which contains a pole, $\mathbf{N}$ or $\mathbf{S}$, and satisfies
	\begin{equation}
	i_{h_\iota(j)} \in \lbrace 0,1\rbrace
	\label{eq:1.25}
	\end{equation}
for all $j$.
To visualize serpents we often include the meander or axis path joining $v$ in the serpent.
See figs.~\ref{fig:1.0}(c), \ref{fig:1.2} and sections~\ref{sec6}, \ref{sec7} for examples.
The \emph{polar arc}, i.e. the arc of the path $h_\iota$ adjacent to its pole, is part of any polar $h_\iota$-serpent, due to adjacency of the Morse indices of $h_\iota$-adjacent equilibria.
We call $\mathbf{N}$-polar serpents and $\mathbf{S}$-polar serpents \emph{anti-polar} to each other.
An \emph{overlap} of anti-polar serpents simply indicates a nonempty intersection.
We call a polar $h_\iota$-serpent \emph{full}, if it extends all the way to the saddle which is $h_{1-\iota}$-adjacent to the opposite pole.
Full $h_\iota$-serpents always overlap with their anti-polar $h_{1-\iota}$-serpent, of course, at least at that saddle.

\begin{figure}[t!]
\centering \includegraphics[width=\textwidth]{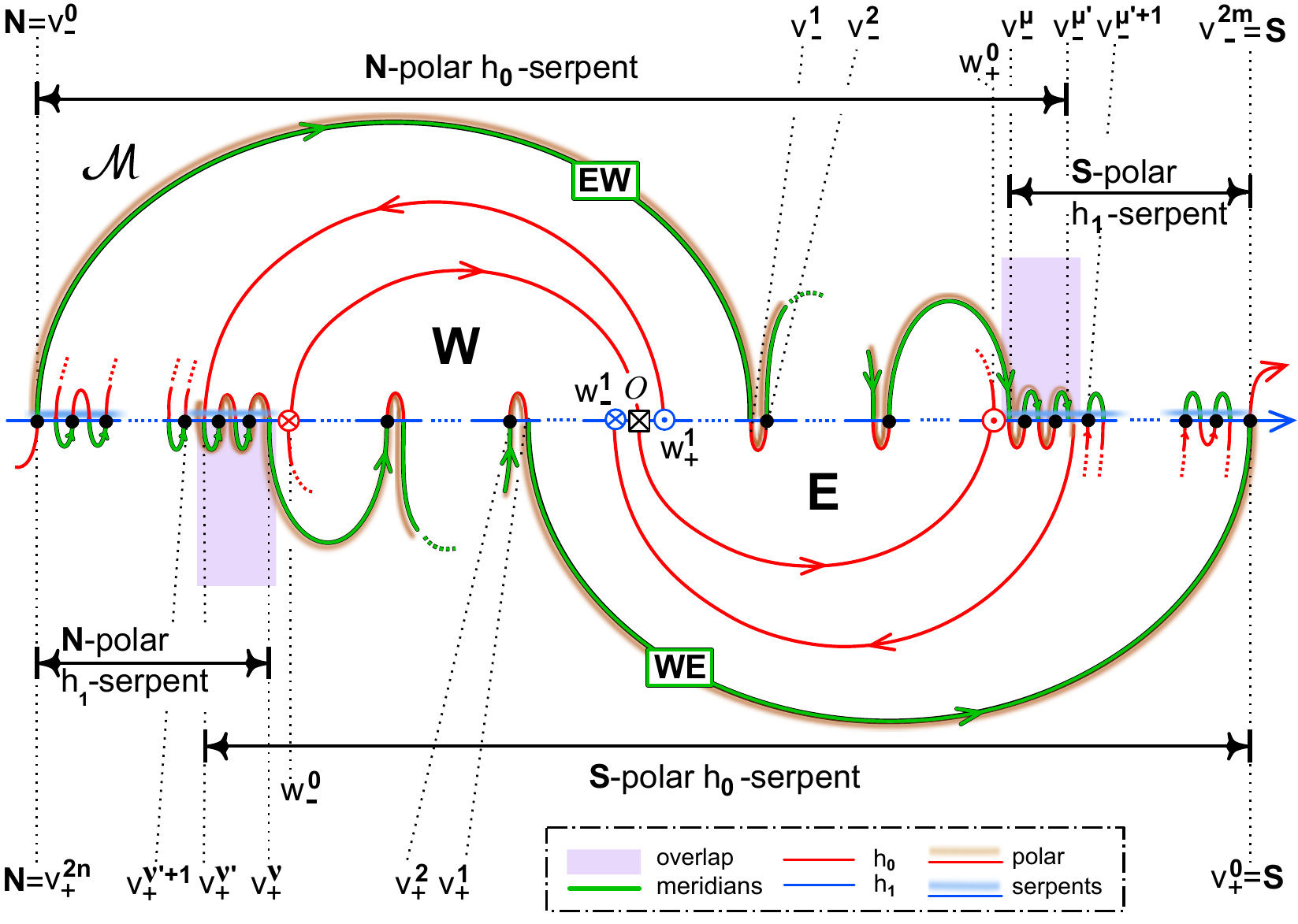}
\caption{\emph{
The 3-meander template.
Note the $\mathbf{N}$-polar $h_1$-serpent $\mathbf{N} = v_+^{2n} \ldots  v_+^\nu$ terminated at $v_+^\nu$ by the subsequent source $w_-^0$ which is, both, $h_1$-extreme minimal and the lower $h_0$-neighbor of $\mathcal{O}$.
This serpent overlaps the anti-polar, i.e. $\mathbf{S}$-polar, $h_0$-serpent $v_+^{\nu '} \ldots  v_+^\nu \ldots  v_+^0  = \mathbf{S}$, from $v_+^{\nu '}$ to $v_+^\nu$.
Similarly, the $\mathbf{N}$-polar $h_0$-serpent $\mathbf{N} = v_-^0 \ldots  v_-^{\mu '}$ overlaps the anti-polar, i.e. $\mathbf{S}$-polar, $h_1$-serpent $v_-^\mu  \ldots  v_-^{\mu '} \ldots  v_-^{2n} = \mathbf{S}$, from $v_-^\mu$ to $v_-^{\mu '}$.
The $h_1$-neighbors $w_\pm^1$ of $\mathcal{O}$ are  the $h_0$-extreme sources, by the two polar $h_0$-serpents.
Similarly, the $h_0$-neighbors $w_\pm^0$ of $\mathcal{O}$ define the $h_1$-extreme sources.
See also sections~\ref{sec6}, \ref{sec7} for specific examples.
}}
\label{fig:1.2}
\end{figure}

\begin{defi}\label{def:1.3}
An abstract Sturm meander $\mathcal{M}$ with intersections $v \in \mathcal{E}$ is called a \emph{3-meander template} if the following four conditions hold, for $\iota \in \{0,1\}$.

\begin{itemize}
\item[(i)] $\mathcal{M}$ possesses a single intersection $v = \mathcal{O}$ with Morse number $i_{\mathcal{O}} =3$, and  no other Morse number exceeds $2$.
\item[(ii)] Polar $h_\iota$-serpents overlap with their anti-polar $h_{1-\iota}$-serpents in at least one shared vertex.
\item[(iii)] The intersection $v=\mathcal{O}$ is located between the two intersection points, in the order of $h_{1-\iota}$, of the polar arc of any polar $h_\iota$-serpent.
\item[(iv)] The $h_\iota$-neighbors $w^\iota_\pm$ of $v=\mathcal{O}$ are the $i=2$ sources which terminate the polar $h_{1-\iota}$-serpents.
\end{itemize}
\end{defi}

See fig.~\ref{fig:1.2} for an illustration of 3-meander templates.
Property (iv), for example, asserts that the $h_\iota$-neighbor sources $w_\pm^\iota$ of $\mathcal{O}$ are the $h_{1-\iota}$-extreme sources, for $\iota \in \{0,1\}$.
See also the example of fig.~\ref{fig:1.0}(c).

The passage from 3-cell templates to 3-meander templates is based on a detailed construction of an SZS-pair $(h_0,h_1)$ of paths in the given 3-cell template.
The construction relies heavily on our previous trilogy \cite{firo09, firo08, firo10} on the planar case.
In section~\ref{sec2} we construct $h_0$ and $h_1$, separately, for each closed hemisphere $\mathbf{W}$ and $\mathbf{E}$.
Each closed hemisphere, by itself, will be viewed as a planar Sturm attractor in \cite{firo3d-2}.

The remaining paper is organized as follows.
In section~\ref{sec2} we recall the construction of the SZS-pair $(h_0,h_1)$ of Hamiltonian paths for any 3-cell template $\mathcal{C}$.
Section~\ref{sec3} comments on the effects of the trivial equivalences $x \mapsto 1-x$ and $u \mapsto -u$ on 3-cell templates and SZS-pairs.
In section~\ref{sec4} we discuss face lifts from certain planar disk complexes to 3-cell complexes via attachment of a Western hemisphere which consists of a single cell.
Duality, a useful tool in the analysis of planar Sturm attractors, is lifted to 3-balls in section~\ref{sec5}.
With these general preparations, and based on the results in our planar Sturm trilogy \cite{firo09, firo08, firo10}, we enumerate all 3-ball Sturm attractors with at most 13 equilibria, in section~\ref{sec6}.
Section~\ref{sec7} is devoted to the Platonic solids as Sturm global attractors.
We conclude, in section~\ref{sec8}, with the ``Snoopy burger'':
a regular cell complex $\mathcal{C}$ of two 3-cells and a total of only 9 equilibria, which cannot be realized as a Sturm dynamic complex $\mathcal{C}_f$.


\textbf{Acknowledgments.}
Extended mutually delightful hospitality by the authors is gratefully acknowledged.
Gustavo~Granja generously shared his deeply topological view point, precise references included.
Anna~Karnauhova has contributed all illustrations with great patience, ambition, and her inimitable artistic touch.
Original typesetting was accomplished by Ulrike~Geiger. This work
was partially supported by DFG/Germany through SFB 910 project A4, and by FCT/Portugal through project UID/MAT/04459/2013.


\section{Hamiltonian pairs in 3-cell templates}
\label{sec2}

We recall results from \cite[section~2]{firo3d-1}.
The design and enumeration of 3-ball Sturm attractors $\mathcal{A}_f$ with prescribed 3-cell template $\mathcal{C} = (c_v)_{v\in \mathcal{E}}$ is based on the construction, by recipe, of an SZS-pair $(h_0,h_1)$ of Hamiltonian paths $h_\iota$: $\lbrace 1, \ldots , N\rbrace \rightarrow \mathcal{E}$.
See definition~\ref{def:2.1}.
The definition turned out to be mandated by the boundary orders $h_\iota^f$ of 3-ball Sturm attractors $\mathcal{A}_f$,
via identifications $\mathcal{E} = \mathcal{E}_f$ and $h_\iota = h_\iota^f$ of \eqref{eq:1.23b}, \eqref{eq:1.23c}.

\begin{figure}[b!]
\centering \includegraphics[width=0.75\textwidth]{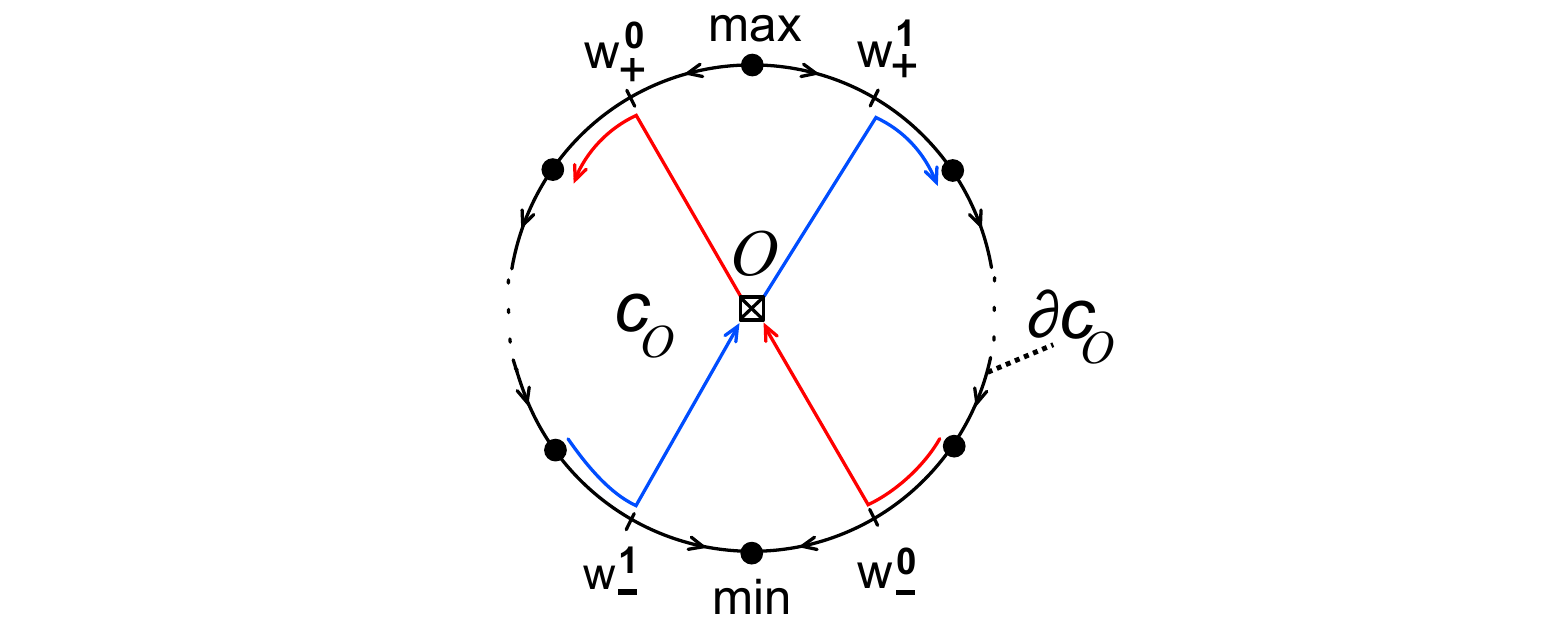}
\caption{\emph{
Traversing a face vertex $\mathcal{O}$ by a ZS-pair $h_0, h_1$.
Note the resulting shapes ``Z'' of $h_0$ (red) and ``S'' of $h_1$ (blue).
The paths $h_\iota$ may also continue into adjacent neighboring faces, beyond $w_\pm^\iota$, without turning into the face boundary $\partial c_{\mathcal{O}}$.
}}
\label{fig:2.1}
\end{figure}

\begin{figure}[t!]
\centering \includegraphics[width=\textwidth]{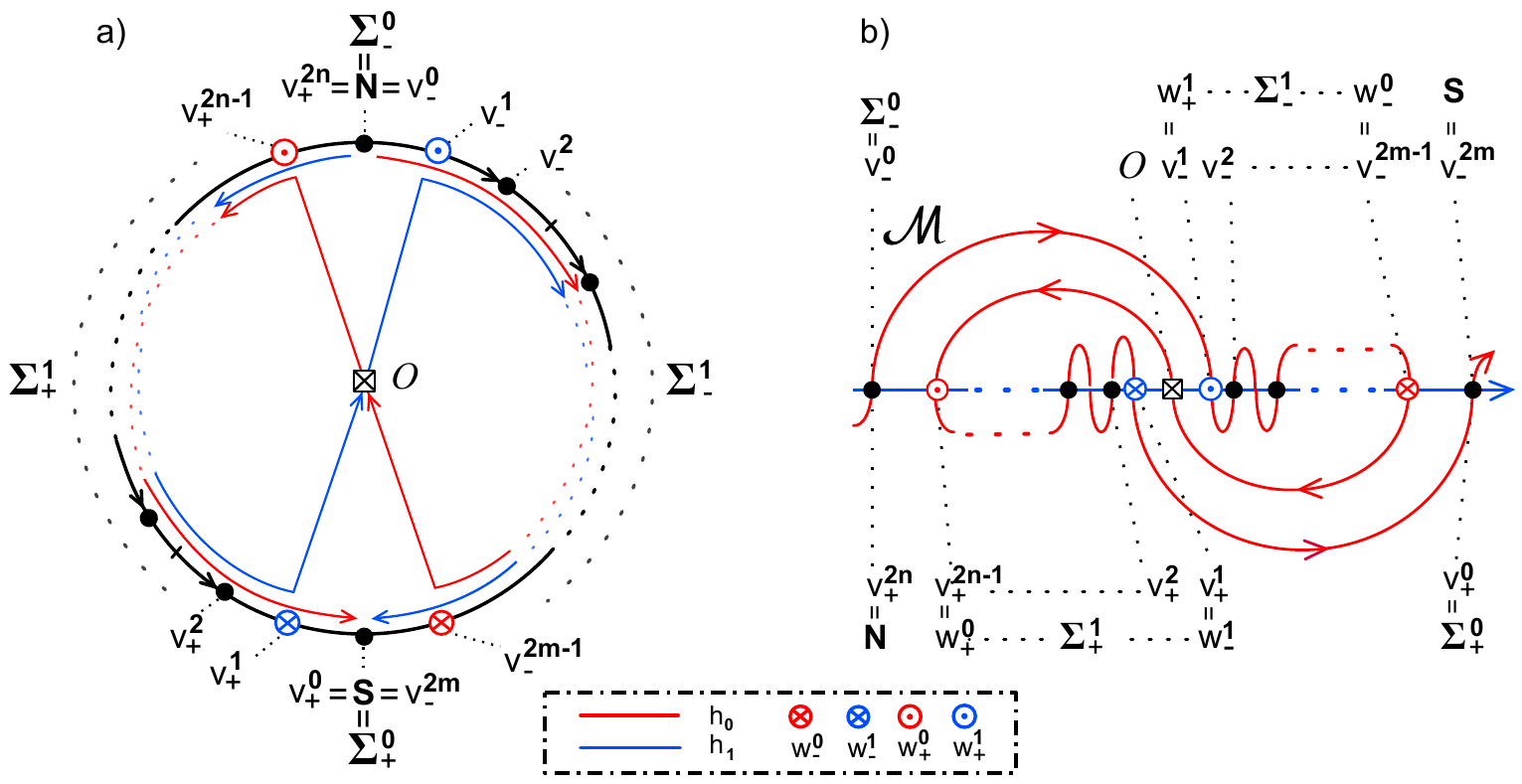}
\caption{\emph{
The Sturm $(m,n)$-gon disk with source $\mathcal{O}$, $m+n$ sinks, $m+n$ saddles, and hemisphere decomposition $\Sigma_\pm^j$, $j \in \{0,1\}$, of $\mathcal{A} = \text{clos } W^u(\mathcal{O})$.
Saddles and sinks are enumerated by $v_\pm^k$ with odd and even exponents $k$, respectively.
(a) The associated Thom-Smale dynamic complex $\mathcal{C}$.
Arrows on the circular boundary indicate the bipolar orientation of the edges of the 1-skeleton.
Edges are the whole one-dimensional unstable manifolds of the saddles; the orientation of the edge runs against the time direction on half of each edge.
The poles $\mathbf{N}$, $\mathbf{S}$ are the extrema of the bipolar orientation.
Geometrically, we obtain an $(m+n)$-gon, with $m$ edges to the right of the poles and $n$ edges to the left.
The resulting bipolar orientation determines the ZS-pair $(h_0,h_1)$, by definition~\ref{def:2.1}.
Colors $h_0$ (red), $h_1$ (blue).
(b) The meander $\mathcal{M}$ defined by the ZS-pair $(h_0, h_1)$ of the $(m,n)$-gon (a).
Along the horizontal axis, equilibria $v \in \mathcal{E}$ are ordered according to the directed path $h_1$ (blue). 
The directed path $h_0$ (red) defines the arcs of the meander $\mathcal{M}$.
Note the two full polar $h_0$-serpents $v_-^0v_-^1 \ldots v_-^{2m-1}$ and $v_+^0v_+^1 \ldots v_+^{2n-1}$.
The two full polar $h_1$-serpents are $ v_+^{2n} \ldots v_+^1$ and $v_-^{1} \ldots v_-^{2m}$. 
Also note how the $h_\iota$-neighboring saddles $w_\pm^\iota$ to the source $\mathcal{O}$, at $x=\iota$, become the $h_{1-\iota}$-extreme saddles at the opposite boundary.
}}
\label{fig:2.2}
\end{figure}

To prepare our construction, we first consider planar regular CW-complexes $\mathcal{C}$, abstractly, with a bipolar orientation of the 1-skeleton $\mathcal{C}^1$.
Here bipolarity requires that the unique poles $\mathbf{N}$ and $\mathbf{S}$ of the orientation are located at the boundary of the regular complex $\mathcal{C} \subseteq \mathbb{R}^2$.

To label the vertices $v \in \mathcal{E}$ of a planar complex $\mathcal{C}$, in two different ways, we construct a pair of directed Hamiltonian paths
	\begin{equation}
	h_0, h_1: \quad \lbrace 1, \ldots , N \rbrace \rightarrow \mathcal{E}
	\label{eq:2.1}
	\end{equation}
as follows.
Let $\mathcal{O}$ indicate any source, i.e. (the barycenter of) any 2-cell  face $c_{\mathcal{O}}$ in $\mathcal{C}$.
(We temporarily deviate from the standard 3-ball notation, here, to emphasize analogies with the passage of $h_\iota$ through a 3-cell, later.)
By planarity of $\mathcal{C}$ the bipolar orientation of $\mathcal{C}^1$ defines unique extrema on the boundary circle $\partial c_{\mathcal{O}}$ of the 2-cell $c_\mathcal{O}$.
Let $w_-^0$ denote the saddle on $\partial c_{\mathcal{O}}$ (of the edge) to the right of the minimum, and $w_+^0$ the saddle to the left of the maximum.
Similarly, let $w_-^1$ be the saddle to the left of the minimum, and $w_+^1$ to the right of the maximum.
See fig.~\ref{fig:2.1}.
Then the following definition serves as a construction recipe for the pair $(h_0,h_1$).

\begin{defi}\label{def:2.1}
The bijections $h_0, h_1$ in \eqref{eq:2.1} are called a \emph{ZS-pair} $(h_0, h_1)$ in the finite, regular, planar and bipolar cell complex $\mathcal{C} = \bigcup_{v \in \mathcal{E}} c_v$ if the following three conditions all hold true:

\begin{itemize}
\item[(i)] $h_0$ traverses any face $c_\mathcal{O}$ from $w_-^0$ to $w_+^0$;
\item[(ii)] $h_1$ traverses any face $c_\mathcal{O}$ from $w_-^1$ to $w_+^1$
\item[(iii)] both $h_\iota$ follow the same bipolar orientation of the 1-skeleton $\mathcal{C}^1$, unless defined by (i), (ii) already.
\end{itemize}

We call $(h_0,h_1)$ an \emph{SZ-pair}, if $(h_1, h_0)$ is a ZS-pair, i.e. if the roles of $h_0$ and $h_1$ in the rules (i) and (ii) of the face traversals are reversed.
\end{defi}

Properties (i)-(iii) of definition \ref{def:2.1} construct the ZS-pair $(h_0,h_1)$ uniquely, provided the resulting paths turn out Hamiltonian. 
Settling existence, the summary in  \cite[section~2]{firo3d-1} guarantees that they are.

In fig.~\ref{fig:2.2} we illustrate definition~\ref{def:2.1} for the simple case of a single 2-disk with $m+n$ sinks and $m+n$ saddles on the boundary, and with a single source $\mathcal{O}$.
The bipolar orientation of the 1-skeleton, in (a), in fact follows from the boundary $\Sigma^0 = \lbrace \mathbf{N}, \mathbf{S}\rbrace$ of the fast unstable manifold $W^{uu}(\mathcal{O})$.
Indeed $z(v-\mathcal{O}) = 0_\pm$ uniquely characterizes $v \in \Sigma_\pm^0$.
Geometrically, the closed disk can be viewed as an $(m+n)$-gon with $m+n$ sink vertices and $m+n$ boundary edges, alias unstable manifolds of the saddles.
The poles separate the boundary into $m$ edges to the right, and $n$ edges to the left.
We therefore call the resulting Sturmian cell complex an $(m,n)$\emph{-gon}.

The planar trilogy \cite{firo08, firo09, firo10} contains ample material on the planar case.
In particular it has been proved that a regular finite cell complex $\mathcal{C}$ is the Thom-Smale dynamic cell complex $\mathcal{C}_f$ of a planar Sturm attractor $\mathcal{A}_f$ if, and only if, $\mathcal{C} \subseteq \mathbb{R}^2$ is planar and contractible with bipolar 1-skeleton $\mathcal{C}^1$.
See \cite[theorem~2.1]{firo3d-1}.
Moreover we can identify $\mathcal{C}_f = \mathcal{C}$ via $\mathcal{E}_f = \mathcal{E},\ h_\iota^f = h_\iota$, as in \eqref{eq:1.23b}, \eqref{eq:1.23c}.
See \cite{firo08, firo09, firo10} for proofs and many more examples.

For a later comeback as hemisphere constituents $\text{clos } \mathbf{W}, \ \text{clos } \mathbf{E}$, in 3-cell templates $\mathcal{C}$, we now single out bipolar topological disk complexes which already satisfy the properties (ii) and (iii) of definition~\ref{def:1.1}.
We recall that a topological disk complex may contain (finitely) many faces.

\begin{defi}\label{def:2.2}
A bipolar topological disk complex $\text{clos } \mathbf{E}$ with poles $\mathbf{N}, \mathbf{S}$ on the circular boundary $\partial \mathbf{E}$ is called \emph{Eastern disk}, if any edge of the 1-skeleton in $\mathbf{E}$, with at least one vertex $v \in \partial \mathbf{E} \setminus \mathbf{S}$, is directed inward, i.e. away from that boundary vertex $v$.
Similarly, we call such a complex $\text{clos } \mathbf{W}$ \emph{Western disk}, if any edge of the 1-skeleton in $\mathbf{W}$, with at least one vertex $v \in \partial \mathbf{W} \setminus \mathbf{N}$, is directed outward, i.e. towards that boundary vertex $v$.
\end{defi}

See fig.~\ref{fig:2.3}.
For SZ- and ZS-pairs $(h_0,h_1)$ this leads to full polar serpents as follows.

\begin{figure}[t!]
\centering \includegraphics[width=\textwidth]{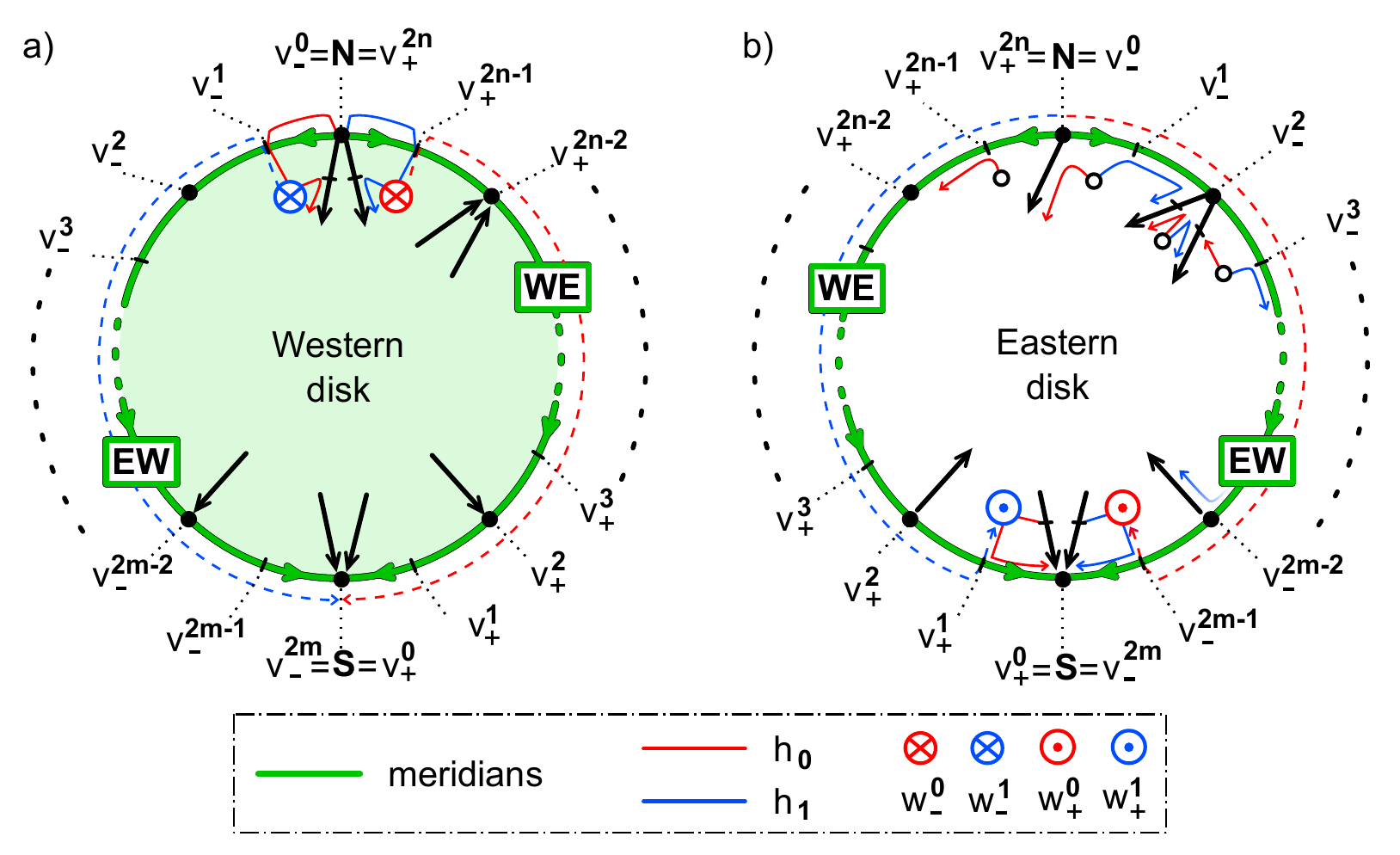}
\caption{\emph{
Western $(\mathbf{W})$ and Eastern $(\mathbf{E})$ planar topological disk complexes.
In $\mathbf{W}$, (a), all edges of the 1-skeleton $\mathbf{W}^1$ with a vertex $v \neq \mathbf{N}$ on the disk boundary are directed outward, i.e. towards $v$.
In $\mathbf{E}$, (b), all 1-skeleton edges with a vertex $v\neq \mathbf{S}$ on the disk boundary are directed inward, i.e. away from $v$.
Note the respective full $\mathbf{S}$-polar $h_0,h_1$-serpents $v_+^{2n-1} \ldots v_+^0 = \mathbf{S}$, $v_-^1 \ldots v_-^{2m} = \mathbf{S}$, dashed red/blue in (a), and the full $\mathbf{N}$-polar $h_0, h_1$-serpents $\mathbf{N} = v_-^0 \ldots v_-^{2m-1}$,
$\mathbf{N} = v_+^{2n} v_+^{2n-1} \ldots v_+^1$, dashed red/blue in (b).
Here we use ZS-pairs $(h_0,h_1)$ in $ \mathbf{E}$, but SZ-pairs $(h_0,h_1)$ in $\mathbf{W}$.
}}
\label{fig:2.3}
\end{figure}

\begin{lem}\label{lem:2.3}
\cite[lemma~2.7]{firo3d-1}
Let  $\mathbf{W}, \mathbf{E}$ be bipolar topological disk complexes with poles $\mathbf{N}, \mathbf{S}$ on their circular boundaries.
Let $(h_0,h_1)$ denote an SZ- or ZS-pair.

Then the disk $\text{clos }\mathbf{W}$ is Western, if and only if the $\mathbf{S}$-polar $h_\iota$-serpents are full, for $\iota \in \{0,1\}$, i.e. they contain all points of their respective boundary half-circle, except the antipodal pole $\mathbf{N}$.

Similarly, the disk $\text{clos } \mathbf{E}$ is Eastern, if and only if the $\mathbf{N}$-polar $h_\iota$-serpents are full, for $\iota \in \{0,1\}$, i.e. they contain all points of their respective boundary half-circle, except the antipodal pole $\mathbf{S}$.
\end{lem}

After these preparations we can now return to general 3-cell templates $\mathcal{C}$ and define the SZS-pair $(h_0,h_1)$ associated to $\mathcal{C}$.

\begin{defi}\label{def:2.4}
Let $\mathcal{C} = \bigcup_{v \in \mathcal{E}} c_v$ be a 3-cell template with oriented 1-skeleton $\mathcal{C}^1$, poles $\mathbf{N}, \mathbf{S}$, hemispheres $\mathbf{W}, \mathbf{E}$, and meridians $\mathbf{EW}$, $\mathbf{WE}$.
A pair $(h_0, h_1)$ of bijections $h_\iota$: $ \lbrace 1, \ldots , N \rbrace \rightarrow \mathcal{E}$ is called the SZS-\emph{pair assigned to} $\mathcal{C}$ if the following conditions hold.
\begin{itemize}
\item[(i)] The restrictions of range $h_\iota$ to $\text{clos } \mathbf{W}$ form an SZ-pair $(h_0, h_1)$, in the closed Western hemisphere.
The analogous restrictions form a ZS-pair $(h_0,h_1)$ in the closed Eastern hemisphere $\text{clos } \mathbf{E}$.
See definition~\ref{def:2.1}.
\item[(ii)] In the notation of figs.~\ref{fig:1.1}, \ref{fig:2.3}, and for each $\iota \in \{0,1\}$, the permutation $h_\iota$ traverses $w_-^\iota, \mathcal{O}, w_+^\iota$, successively.
\end{itemize}
The swapped pair $(h_1,h_0)$ is called the ZSZ-pair of $\mathcal{C}$.
\end{defi}

See fig.~\ref{fig:1.1}(b) for a general illustration, and fig.~\ref{fig:1.0} for a specific example.
Condition (i) identifies the closed hemispheres as the Thom-Smale dynamic complexes of planar Sturm attractors; see lemma~\ref{lem:2.3}.
The resulting full polar serpents of $h_\iota$ are indicated by dashed lines.

It is easy to see why the SZS-pair $(h_0,h_1)$ is unique, for any given 3-cell template $\mathcal{C}$.
Indeed, the bipolar orientation of $\mathcal{C}$ fixes the orderings $h_0$ and $h_1$ uniquely on the 1-skeleton of $\mathcal{C}$.
The SZ- and ZS-requirements of (i) determine how $h_\iota$ traverses each face, except for the faces of the $h_\iota$-neighbors $w_\pm^\iota$ of $\mathcal{O}$.
That final missing piece is uniquely prescribed to be $w_-^\iota \mathcal{O}w_+^\iota$, by requirement (ii) of definition~\ref{def:2.4}.
This assigns a unique SZS-pair $(h_0, h_1)$ of Hamiltonian paths, from pole $\mathbf{N}$ to pole $\mathbf{S}$, for any given 3-cell template $\mathcal{C}$.

With the above construction of the SZS-pair $(h_0,h_1)$, for any given 3-cell template $\mathcal{C}$, the construction of the unique Sturm permutation $\sigma_f=\sigma =h_0^{-1}\circ h_1$ is complete.
This also identifies the unique 3-meander template $\mathcal{M}_f = \mathcal{M}$ and 3-ball Sturm attractor $\mathcal{A}_f$, up to $C^0$ flow-equivalence, with prescribed Thom-Smale dynamic complex $\mathcal{C}_f=\mathcal{C}$ and prescribed hemisphere decomposition \eqref{eq:1.17}.
In the example of fig.~\ref{fig:1.0}, this construction amounts to the cyclic implications (b) $\Rightarrow$ (c) $\Rightarrow$ (a) $\Rightarrow$ (b).


\section{Trivial equivalences}\label{sec3}

To reduce the number of cases in complete enumerations, a proper consideration of symmetries is mandatory.
For 3-cell templates $\mathcal{C}$ there are two sources of such symmetries.
First, there are the automorphisms of the cell complex $\mathcal{C}$ itself.
The isotropy subgroups of the orthogonal group $O(3)$ for the five Platonic solids provide a rich source of examples.
Second, there are certain trivial equivalences which arise from the signed hemisphere decomposition \eqref{eq:1.17} of $\mathcal{C}$; see definition~\ref{def:1.1}.
As an illustration, we digress with some observations on cell counts for flip-symmetric Sturm attractors.
For 3-cell templates, we then eliminate trivial equivalences rather easily, in an adhoc manner, based on certain choices of poles and bipolar orientations.
The effect of trivial equivalences, elementary as it may be, deserves some careful attention to avoid duplicates and omissions in the resulting lists of cases.
We summarize the pertinent results in fig.~\ref{fig:3.1} and table~\ref{tbl:3.1} below.

Already in \cite{firo96}, \emph{trivial equivalences} were defined as the Klein 4-group $\langle \kappa,\rho \rangle$ with commuting involutive generators
	\begin{align}
	(\kappa u)(x) &:= -u(x)\,; \label{eq:3.1}\\
	(\rho u)(x) &:= u(1-x)\,.\label{eq:3.2}
	\end{align}
In the PDE \eqref{eq:1.1}, the $u$-flip $\kappa$ induces a linear flow equivalence $\mathcal{A}_f \rightarrow \kappa \mathcal{A}_f = \mathcal{A}_{f^\kappa}$ where $f^\kappa (x,u,p)$:= $f(x,-u,-p)$.
Similarly, the $x$-reversal $\rho$ induces a linear flow equivalence $\mathcal{A}_f \rightarrow \rho \mathcal{A}_f = \mathcal{A}_{f^\rho}$ via $f^\rho (x,u,p)$:= $f(1-x,u,-p)$. Here and below $\mathcal{E},\ \mathcal{A},\ \mathcal{C},\ h_\iota,\ \sigma$ refer to $f$, whereas $\mathcal{E}^\gamma,\ \mathcal{A}^\gamma,\ \mathcal{C}^\gamma,\ h_\iota^\gamma,\ \sigma^\gamma$ will refer to $f^\gamma$.

For example, let us describe the effect of $\kappa = -\text{id}$ on the Hamiltonian paths $h_\iota$ and on the Sturm permutations $\sigma$ algebraically.
We abuse notation slightly and let $\kappa$ also denote the involution permutation
	\begin{equation}
	\kappa j := N+1-j
	\label{eq:3.8}
	\end{equation}
on $j \in \lbrace 1, \ldots , N\rbrace$.
Then $\kappa$ reverses the boundary orders of the equilibria $\mathcal{E}^\kappa= -\mathcal{E}$, at $x=\iota \in \{0,1\}$, respectively.
Therefore
	\begin{equation}
	h_\iota^\kappa = \kappa h_\iota \kappa\,,
	\label{eq:3.9}
	\end{equation}
and $\sigma = h_0^{-1}\circ h_1$ leads to the conjugation
	\begin{equation}
	\sigma^\kappa = \kappa \sigma \kappa\,.
	\label{eq:3.10}
	\end{equation}

For illustration, we consider the cell counts of flip-symmetric Sturm attractors, i.e. for Sturm permutations $\sigma$ such that $\sigma^\kappa := \kappa \sigma \kappa = \sigma$; see \eqref{eq:3.10}. 
Let
\begin{equation}
\label{eq:3.19}
\mathcal{E}(i) := \{v \in \mathcal{E} \,|\, i(v)=i\}
\end{equation}
denote the set of equilibria $v$ with given Morse index $i$. 
The \emph{cell counts} $c_i$ count the elements of $\mathcal{E}(i)$.
It is interesting to compare the following proposition with standard Morse theory, which asserts that the alternating sum of the cell counts $c_i$ over all Morse indices $i$ coincides with the Euler characteristic +1 of the global attractor $\mathcal{A}$.

\begin{prop}\label{prop:3.1}
Let $\sigma \in S_N$ be a flip-symmetric Sturm permutation $\sigma^\kappa := \kappa \sigma \kappa = \sigma$ of $N$ equilibria.

Then all cell counts $c_i$ are even, except for one unique odd Morse count $c_{i_*}$.
Moreover
\begin{equation}
\label{eq:3.20}
N \quad \equiv \quad
\begin{cases}
      & 1 \ (\mathrm{mod}\, 4),\ \text{if and only if}\  i_*\ \text{is even}, \\
      & 3 \ (\mathrm{mod}\, 4),\ \text{if and only if}\  i_*\ \text{is odd}.
\end{cases}
\end{equation}

\end{prop}

\begin{proof}[\textbf{Proof.}]
The flip $\kappa$ of \eqref{eq:3.1} amounts to a rotation by $180^\circ$ in the $(u,u_x)$-plane of the Sturm meander $\mathcal{M}$.
Since $\sigma^\kappa = \sigma$, the Sturm meander can be considered to be invariant under that rotation.
For the Morse indices, alias Morse numbers, the recursion \eqref{eq:1.22a} therefore implies 
\begin{equation}
\label{eq:3.20a}
i(h_0(\kappa j)) = i(h_0(j)),
\end{equation}
for all $j$. 
The same recursion also implies the parity switch from even/odd labels $j$ to odd/even Morse indices $i(h_0(j))$, respectively.

Let $j_* \in \{1,\ldots,N\}$ denote the unique fixed point of $\kappa$.
More specifically
\begin{equation}
\label{eq:3.20b}
j_* = \frac{1}{2} (N+1) =
\begin{cases}
      & \text{odd, for} \ N\equiv1 \ (\mathrm{mod}\, 4), \\
      & \text{even, for} \ N\equiv3 \ (\mathrm{mod}\, 4).
\end{cases}
\end{equation}
Here we used that $N$ itself is odd, due to the meander property.
Under the involution $\kappa^2=\mathrm{id}$ of the labels $1 \leq j \leq N$, all orbits $\{j,\kappa j\}$ are 2-cycles, except for the fixed point $j=j_*$. 

Now let $J(i) := \{j\,|\, i(h_0(j))=i\} = h_0^{-1} \mathcal{E}(i)$ enumerate the equilibria $v$ with given Morse index $i(v)=i$, via the Hamiltonian path $h_0$.
By \eqref{eq:3.20a}, each set $J(i)$ is invariant under $\kappa$,
and the sets are pairwise disjoint.

Define $i_* := i(h_0(j_*)$ as the Morse index associated to the fixed point $j_*$ of $\kappa$. 
Then $j_* \in J(i_*)$, by construction.
In view of the remaining 2-cycles of $\kappa$ in $J(i_*)$, the Morse count $c_{i_*}$ of the elements of $\mathcal{E}(i_*)=h_0J(i_*)$ is odd.
The other sets $J(i),\ i\neq i_*\,,$ are disjoint from $J(i_*)$, and therefore consist of 2-cycles of $\kappa$, only.
Hence the Morse count $c_{i}$ is even, for each other set $\mathcal{E}(i)$.

To prove the parity claim \eqref{eq:3.20}, we recall that the even/odd parities of $j_*$ and of $i_*=i(h_0(j_*))$ are opposite, by parity switch. Therefore \eqref{eq:3.20b} proves the proposition.
\end{proof}

\begin{cor}\label{cor:3.2}
Suppose $\sigma = \kappa \sigma \kappa =: \sigma^\kappa$ is the flip-symmetric Sturm permutation of a Sturm 3-ball global attractor with $N$ equilibria.

Then $N\equiv 3\ (\mathrm{mod}\ 4)$. Moreover each Morse count $c_0,c_1,c_2$ of i=0 sinks, i=1 sources, and i=2 faces, respectively, is even.
\end{cor}

\begin{proof}[\textbf{Proof.}]
By definition, Sturm 3-balls possess a single equilibrium of the odd Morse index $i_*=3$. Therefore proposition \ref{prop:3.1} proves the corollary.
\end{proof}

After this little digression we now return to the effect of general trivial equivalences on the Thom-Smale complexes $\mathcal{C}$, the Hamiltonian paths $h_\iota$, and the Sturm permutations $\sigma$.
We first describe the effect of trivial equivalences on the level of signed hemisphere complexes, via the actions of the group elements $\gamma = \kappa, \rho, \kappa \rho$ on the hemispheres $\Sigma_\pm^j(v)$.
By definition \eqref{eq:1.9b}, \eqref{eq:1.9c} of $\Sigma_\pm^j$ we observe
	\begin{align}
	\Sigma_\pm^{\kappa,j}(\kappa v) &=
	\phantom{\lbrace\rbrace}\kappa \Sigma_\mp^j(v)\,;
	\label{eq:3.3}\\
	\Sigma_\pm^{\rho,j}(\rho v) &=
	\left\lbrace
	\begin{aligned}
	&\rho \Sigma_\pm^j(v)\,,&\text{for}\ &j\  \text{even;}\\
	&\rho \Sigma_\mp^j(v)\,,&\text{for}\ &j\ \text{odd.}
	\end{aligned}
	\right.
	\label{eq:3.4}
	\end{align}

\begin{figure}[t!]
\centering \includegraphics[width=\textwidth]{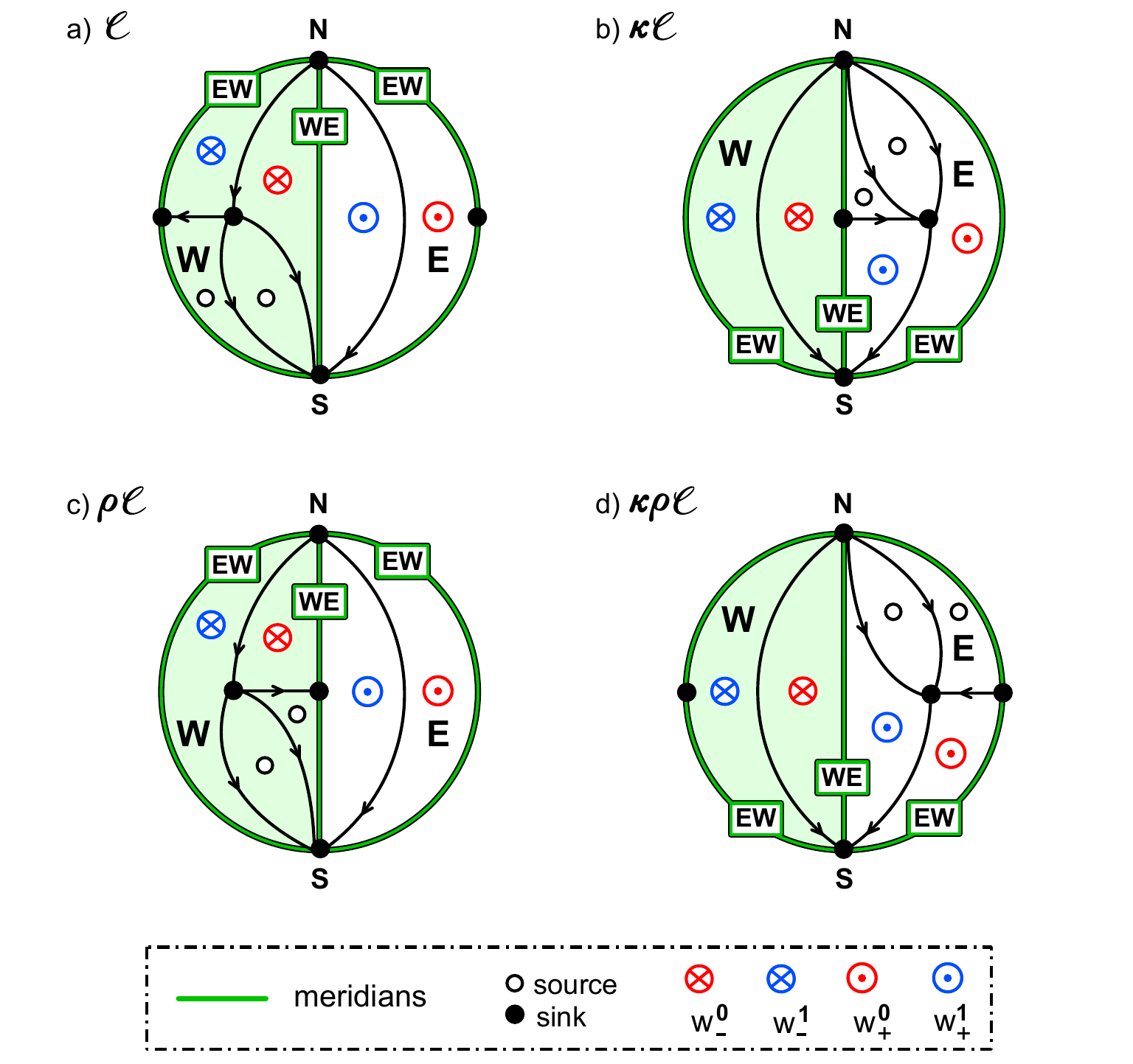}
\caption{\emph{
The effects of the trivial equivalences $\kappa, \rho$, and $\kappa\rho$ on a 3-cell template $\mathcal{C}$ with 19 equilibria.
The cell complex $\mathcal{C}$ is drawn as the boundary sphere $S^2= \partial c_{\mathcal{O}}$, in the style of figs.~\ref{fig:1.0}(b) and \ref{fig:1.1}.
(a) The original 3-cell template $\mathcal{C}$.
(b)--(d) The 3-cell templates $\gamma \mathcal{C},\ \gamma \in \lbrace \kappa, \rho, \kappa\rho \rbrace$.
Annotations refer to the resulting template with hemisphere decomposition given by \eqref{eq:1.17}.
See also the summary in table~\ref{tbl:3.1}, at the end of this section.
All entries of table~\ref{tbl:3.1} can be recovered from the figure, in principle, because corresponding cells are easily identified by their shape and connectivity.
Note how the Klein 4-group $\langle \kappa, \rho \rangle$ acts transitively on the four elements $w_\pm^\iota$. 
}}
\label{fig:3.1}
\end{figure}

We now specialize to 3-cell templates $\mathcal{C}=\mathcal{C}_f$. 
Let us consider the effect of the $u$-flip $\kappa = -\text{id}$ first.
See fig.~\ref{fig:1.1} again, and fig.~\ref{fig:3.1}(a),(b).
The orientation of $\mathcal{C}$ is reversed by $\kappa$.
The involution $\kappa$ also reverses the bipolar orientation, and swaps poles, meridians, hemispheres, and overlap faces as
	\begin{equation}
	\begin{aligned}
	 \mathbf{N} \quad &\longleftrightarrow \quad \mathbf{S}\,;\\
	 \mathbf{WE} \quad &\longleftrightarrow \quad \mathbf{EW}\,;\\
	\kappa:\qquad \mathbf{W} \quad &\longleftrightarrow \quad \mathbf{E}\,;\\
	 \mathbf{NE} \quad &\longleftrightarrow \quad  \mathbf{SE}\,;\\
	 \mathbf{NW} \quad &\longleftrightarrow \quad \mathbf{SW}\,.
	\end{aligned}
	\label{eq:3.5}
	\end{equation}
More precisely, let $\mathbf{N}^\kappa,\ \mathbf{S}^\kappa$ denote the North and South poles $\Sigma_-^{\kappa,0} (\kappa \mathcal{O}),\ \Sigma_+^{\kappa,0}(\kappa \mathcal{O})$ in $\kappa \mathcal{C} = -\mathcal{C}$, respectively.
Then
	\begin{equation}
	\mathbf{N}^\kappa = \kappa \mathbf{S}\,, \quad
	\mathbf{S}^\kappa = \kappa \mathbf{N}\,,
	\label{eq:3.6}
	\end{equation}
by \eqref{eq:3.3} with $j=0$ and $v =\mathcal{O}$.
The other claims of \eqref{eq:3.5} are understood as analogous abbreviations.
For example, let $w_-^{\kappa,0},\ w_-^{\kappa,1},\ w_+^{\kappa,0},\ w_+^{\kappa,1}$ denote the face centers of $\mathbf{NE}^\kappa,\ \mathbf{NW}^\kappa,\ \mathbf{SE}^\kappa,\ \mathbf{SW}^\kappa$, respectively.
Then
	\begin{equation}
	w_\pm^{\kappa, \iota} = \kappa w_\mp^\iota
	\label{eq:3.7}
	\end{equation}
for $\iota =0,\ 1$, in agreement with the last two lines of \eqref{eq:3.5}.
Note how $(h_0^\kappa, h_1^\kappa)$ remains an SZS-pair, for SZS-pairs $(h_0,h_1)$, albeit for the complex $\mathcal{C}^\kappa = -\mathcal{C}$ of reversed orientation.

In summary, we can visualize the geometric effect of the orientation reversing reflection $\kappa = -\text{id}$ in fig.~\ref{fig:1.1} as, first, a rotation of the 2-sphere $S^2$ by 180$^\circ$.
The rotation axis is defined by the intersections of the two meridians $\mathbf{EW}$ and $\mathbf{WE}$ with the equator of $S^2$.
Subsequently, we perform a reflection of each hemisphere through a $\pm 90$ degree meridian which bisects each hemisphere and interchanges the 0~degree Greenwich meridian WE with the date line EW at 180$^\circ$.
The effect of $\kappa$ on the 3-meander template is a simple rotation by 180$^\circ$.
The combinatorial effect is the conjugation \eqref{eq:3.10} of the Sturm permutation $\sigma$ by the involution $\kappa$ of \eqref{eq:3.8}.

We consider the effect of the $x$-reversal $(\rho u)(x) = u(1-x)$ next, in slightly condensed form.
See fig.~\ref{fig:3.1}(c).
Again $\rho$ reverses the orientation of $\mathcal{C}$.
Because $\rho$ only interchanges 1-hemispheres $\Sigma_\pm^1(v)$, however, the poles and bipolar edge orientations of the 1-skeleton $\mathcal{C}^1$ remain unaffected.
The hemispheres $\mathbf{W},\ \mathbf{E}$ preserve their labels, as sets, but each is reflected along its $\pm 90^\circ$ meridian, thus reversing all face orientations.
Consequently $\Sigma_\pm^1(\mathcal{O})$, i.e.~the Greenwich meridian and the date line, are interchanged.
Therefore \eqref{eq:3.5} becomes
	\begin{equation}
	\begin{aligned}
	\mathbf{WE} \quad &\longleftrightarrow \quad \mathbf{EW}\\
	\mathbf{NE} \quad &\longleftrightarrow \quad  \mathbf{NW}\\
	\mathbf{SE} \quad &\longleftrightarrow \quad \mathbf{SW}\,.
	\end{aligned}
	\label{eq:3.11}
	\end{equation}
with preserved roles of the poles $\mathbf{N}, \mathbf{S}$ and the hemispheres $\mathbf{W}, \mathbf{E}$.
In fact, \eqref{eq:3.6} and \eqref{eq:3.7} get replaced by
	\begin{align}
	\mathbf{N}^\rho = \rho \mathbf{N}\,, &\quad
	\mathbf{S}^\rho = \rho \mathbf{S}\,;	\label{eq:3.12}\\
	w_\pm^{\rho, \iota} &= \rho w_\pm^{1-\iota}\,. \label{eq:3.13}
	\end{align}
The $\iota$-swap of the $\mathcal{O}$-neighbors $w_\pm^\iota$ and the reversal of the planar, but not bipolar, orientation of each face $c_v$ in $\mathcal{C}$ imply that $(\rho h_0, \rho h_1)$ become a ZSZ pair, instead of an SZS pair, in $\rho \,\mathcal{C}$.
Therefore \eqref{eq:3.9}, \eqref{eq:3.10} now read	
	\begin{align}
	h_\iota^\rho &= \rho h_{1-\iota}\,;
	\label{eq:3.14}\\
	\sigma^\rho &= (\rho h_1)^{-1} \circ 
	(\rho h_0) = \sigma^{-1}\,,\label{eq:3.15}
	\end{align}
using $\sigma = h_0^{-1} \circ h_1$.
In the PDE setting, property \eqref{eq:3.14} also follows directly, by definition of the boundary orders $h_\iota^f = h_{1-\iota}^{f^\rho}$.

In summary, we can visualize the geometric effect of the orientation reversing $x$-reversal $\rho$ in fig.~\ref{fig:1.1} as just a reflection of each hemisphere through a $\pm 90$ degree meridian.
This just swaps the meridian $\mathbf{WE}$ with $\mathbf{EW}$ and converts $(h_0,h_1)$ to a ZSZ pair.
Combinatorially, the Sturm permutation $\sigma$ gets replaced by its inverse $\sigma^{-1}$.

\begin{table}[]
\centering \includegraphics[width=\textwidth]{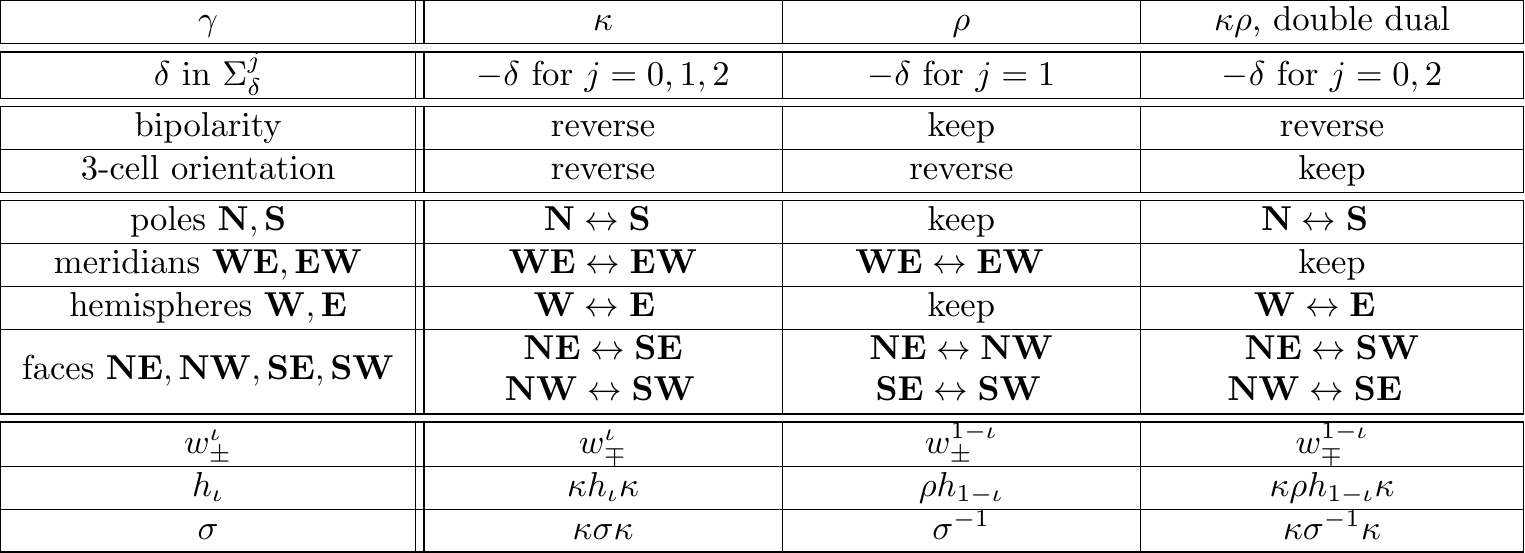}
\caption{\emph{
The effects of trivial equivalences $\gamma \in \langle \kappa, \rho \rangle$ on hemisphere decompositions, orientations, 3-cell complexes, Hamiltonian paths, and Sturm permutations.
For the double dual see \eqref{eq:5.6a}.
}}
\label{tbl:3.1}
\end{table}

The third nontrivial element of the Klein 4-group generated by $\kappa,\ \rho$ is the involution $\kappa\rho = \rho \kappa$ of course.
See fig.~\ref{fig:3.1}(d).
Combinatorially, this replaces $\sigma$ by the conjugate inverse, or inverse conjugate:
	\begin{equation}
	\sigma^{\rho \kappa} = (\sigma^\kappa)^{-1} =
	(\sigma^{-1})^\kappa = \kappa \sigma^{-1}\kappa\,.
	\label{eq:3.16}
	\end{equation}
Geometrically, the third involution $(\rho \kappa u)(x) = -u(1-x)$ acts on hemispheres by
	\begin{equation}
	\Sigma_\pm^{\kappa\rho ,j}(\kappa\rho v) =
	\left\lbrace
	\begin{aligned}
	&\kappa\rho \Sigma_\mp^j(v)\,,\quad &\text{for}\quad &j \in \{0,2\}\,;\\
	&\kappa\rho \Sigma_\pm^j(v)\,,\quad &\text{for}\quad &j=1\,.
	\end{aligned}
	\right.
	\label{eq:3.17}
	\end{equation}
In particular, $\kappa \rho$ reverses the bipolar orientation of all edges, 
but not the orientation of the 3-cell $c_{\mathcal{O}}$.
This  swaps poles, hemispheres, and overlap faces as
	\begin{equation}
	\begin{aligned}
	\mathbf{N} \quad &\longleftrightarrow \quad \mathbf{S}\,;\\
	\mathbf{W} \quad &\longleftrightarrow \quad  \mathbf{E}\,;\\
	\mathbf{NE} \quad &\longleftrightarrow \quad \mathbf{SW}\,;\\
	\mathbf{NW} \quad &\longleftrightarrow \quad \mathbf{SE}\,;
	\end{aligned}
	\label{eq:3.18}
	\end{equation}
but preserves the roles of the two meridians $\mathbf{WE}$ and $\mathbf{EW}$.
This can be visualized, in figs.~\ref{fig:1.0} and \ref{fig:1.1}, as a 180$^\circ$ rotation of the Sturm 3-ball through an axis defined by the intersections of the two meridians with the equator.

We summarize the results in table~\ref{tbl:3.1}.
Note that the Klein 4-group $\langle \kappa, \rho\rangle$ of trivial equivalences maps 3-cell templates to 3-cell templates and 3-meanders to 3-meanders.
Of course this claim can be checked against definitions~\ref{def:1.1}, \ref{def:1.3}, with the above remarks.
On the equivalent level of 3-ball Sturm attractors $\mathcal{A}_f$, however, this is trivial via the linear flow-equivalences \eqref{eq:3.1}, \eqref{eq:3.2} of the global attractors under the generators $\kappa,\ \rho$.


\section{Face and eye lifts}\label{sec4}

The characterization of 3-ball Sturmian Thom-Smale dynamic complexes $\mathcal{C}_f$ as 3-cell templates $\mathcal{C} = \text{clos } c_\mathcal{O}$, in definition~\ref{def:1.1}, can be described as the proper welding of two regular, bipolar topological disk complexes, $\mathcal{C}_-$ and $\mathcal{C}_+$, along their shared meridian boundary.
The welding succeeds to form the 2-sphere $S^2 = \partial c_\mathcal{O}$ of $\mathcal{C}$ if, and only if, two conditions hold.
First, the constituents $\mathcal{C}_-$ and $\mathcal{C}_+$ must be Western and Eastern disks, respectively, in the sense of definition~\ref{def:2.2} and lemma~\ref{lem:2.3}.
Second, the rather delicate overlap condition of definition~\ref{def:1.1}(iv) must be satisfied.
In this section we discuss EastWest complexes $\mathcal{C}_0$ which can serve, universally, as Western or Eastern disks alike.
The overlap condition is then automatically satisfied, whatever their complementing Eastern or Western disk may be.
Effectively this will allow us to lift any Eastern or Western disk $\mathcal{C}_\pm$ to a 3-cell template, by the faces of any EastWest complex $\mathcal{C}_0$.
This single \emph{face lift} will account for the majority of cases in the examples of sections~\ref{sec6} and \ref{sec7}.

\begin{figure}[]
\centering \includegraphics[width=\textwidth]{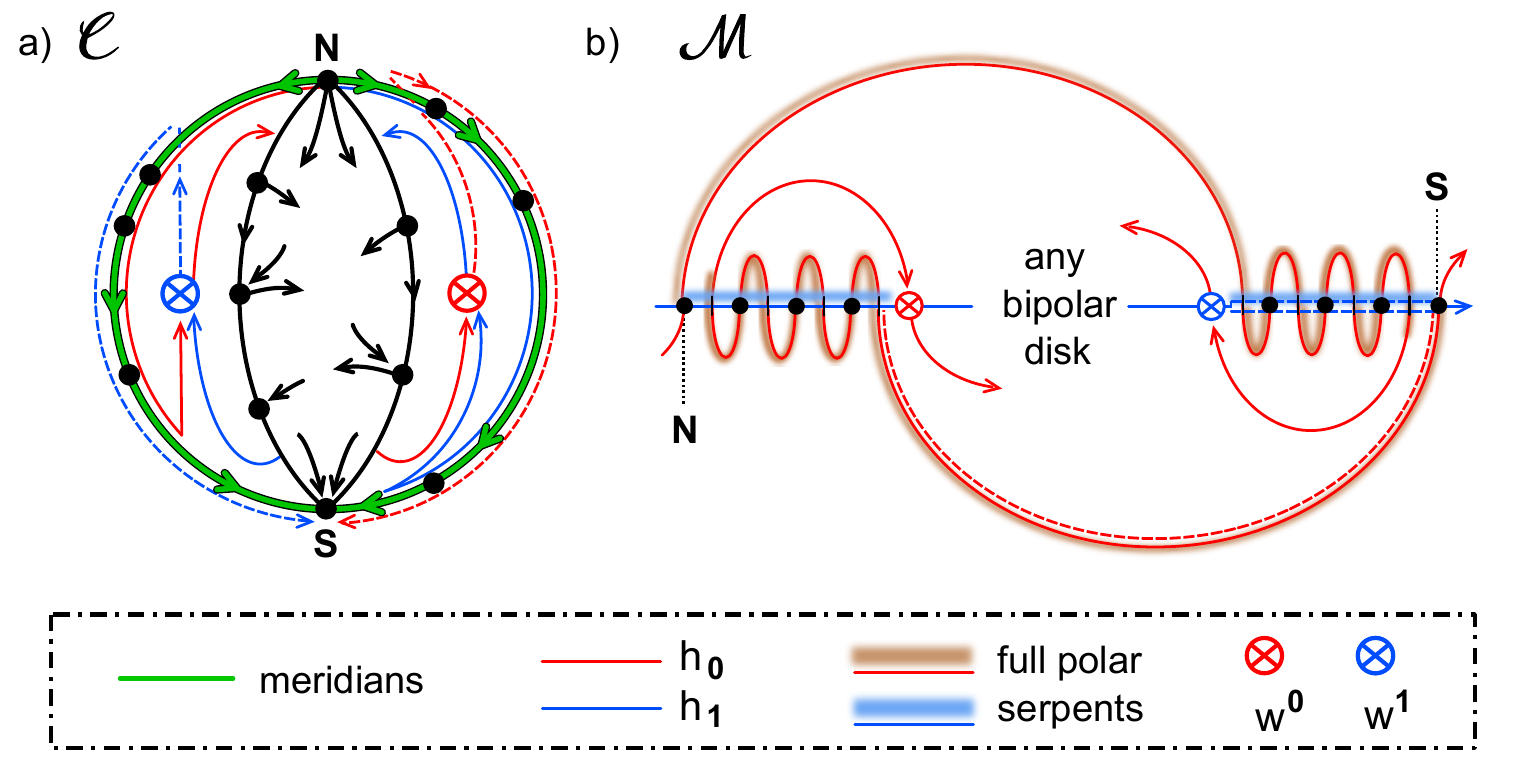}
\caption{\emph{
(a) Schematics of an EastWest complex $\mathcal{C}_0$.
Poles are $\mathbf{N},\ \mathbf{S}$.
Arrows indicate the bipolar orientation.
Green: pole-to-pole boundary paths.
The paths are contained in the boundaries $\partial c_{w^\iota}$ of the boundary faces $c_{w^\iota}$ with barycenters $w^\iota$, respectively, for $\iota \in \{0,1\}$.
(b) The meander $\mathcal{M}$ resulting from the SZ-pair $(h_0,h_1)$ in $\mathcal{C}$.
Note the full $\mathbf{N}$- and $\mathbf{S}$-polar serpents.
}}
\label{fig:4.1}
\end{figure}

\begin{defi}
Let $\mathcal{C}_0$ be a regular, bipolar topological disk complex.
We call $\mathcal{C}_0$ \emph{EastWest disk} if $\mathcal{C}_0$ is, both, Western and Eastern in the sense of definition~\ref{def:2.2}.
\label{def:4.1}
\end{defi}

See fig.~\ref{fig:3.1} for four examples; the EastWest complex is the closed hemisphere with two faces, in those cases.
See also fig.~\ref{fig:4.1} for the general case.

\begin{lem}
For any regular bipolar topological disk complex $\mathcal{C}_0$ the following three properties are equivalent:
\begin{itemize}
\item[(i)] $\mathcal{C}_0$ is an EastWest disk;
\item[(ii)] all polar serpents of the SZ- or ZS-pair $(h_0,h_1)$ of $\mathcal{C}_0$ are full polar serpents;
\item[(iii)] each pole-to-pole boundary path in $\mathcal{C}_0$ is contained in the boundary of some single face.
\end{itemize}
\label{lem:4.2}
\end{lem}

\begin{proof}[\textbf{Proof.}]
By lemma~\ref{lem:2.3}, the disk $\mathcal{C}_0$ is Western/Eastern if and only if the $\mathbf{S}$- and $\mathbf{N}$-polar serpents are full.
This proves that (i) and (ii) are equivalent.

To show that (i) implies (iii) we only have to show that interior edges in an EastWest complex $\mathcal{C}_0$ do not possess vertices $v$ on the boundary, other than the poles $\mathbf{N},\ \mathbf{S}$.

This is obvious because any such edge has to be directed away from $v$, since $\mathcal{C}_0$ is Eastern, and also towards $v$, since $\mathcal{C}_0$ is also Western.

To show that, conversely, (iii) implies (i) we only have to remark that all interior edges with boundary vertices $v$ are polar, and hence are exempt of any Western or Eastern orientation requirements.
This proves the lemma.
\end{proof}

\begin{defi}
Let $\mathcal{C}_+$ be an Eastern disk and $\mathcal{C}_0$ an EastWest disk such that $\partial \mathcal{C}_0$ coincides with the mirror image of $\partial \mathcal{C}_+$, in the chosen planar embedding.
We call the 3-cell template $\mathcal{C}$ defined by
	\begin{equation}
	\text{clos } \mathbf{W} := \mathcal{C}_0\,,\qquad
	\text{clos } \mathbf{E}:= \mathcal{C}_+
	\label{eq:4.1}
	\end{equation}
the \emph{West lift} of the Eastern disk $\mathcal{C}_+$ by the EastWest disk $\mathcal{C}_0$.\newline
The \emph{East lift} $\mathcal{C}$ of a Western disk $\mathcal{C}_-$ by a boundary compatible EastWest disk $\mathcal{C}_0$ is defined, analogously, by
	\begin{equation}
	\text{clos } \mathbf{W} := \mathcal{C}_-\,,\qquad
	\text{clos } \mathbf{E}:= \mathcal{C}_0\,.
	\label{eq:4.2}
	\end{equation}
\label{def:4.3}
\end{defi}

\begin{figure}[t!]
\centering \includegraphics[width=\textwidth]{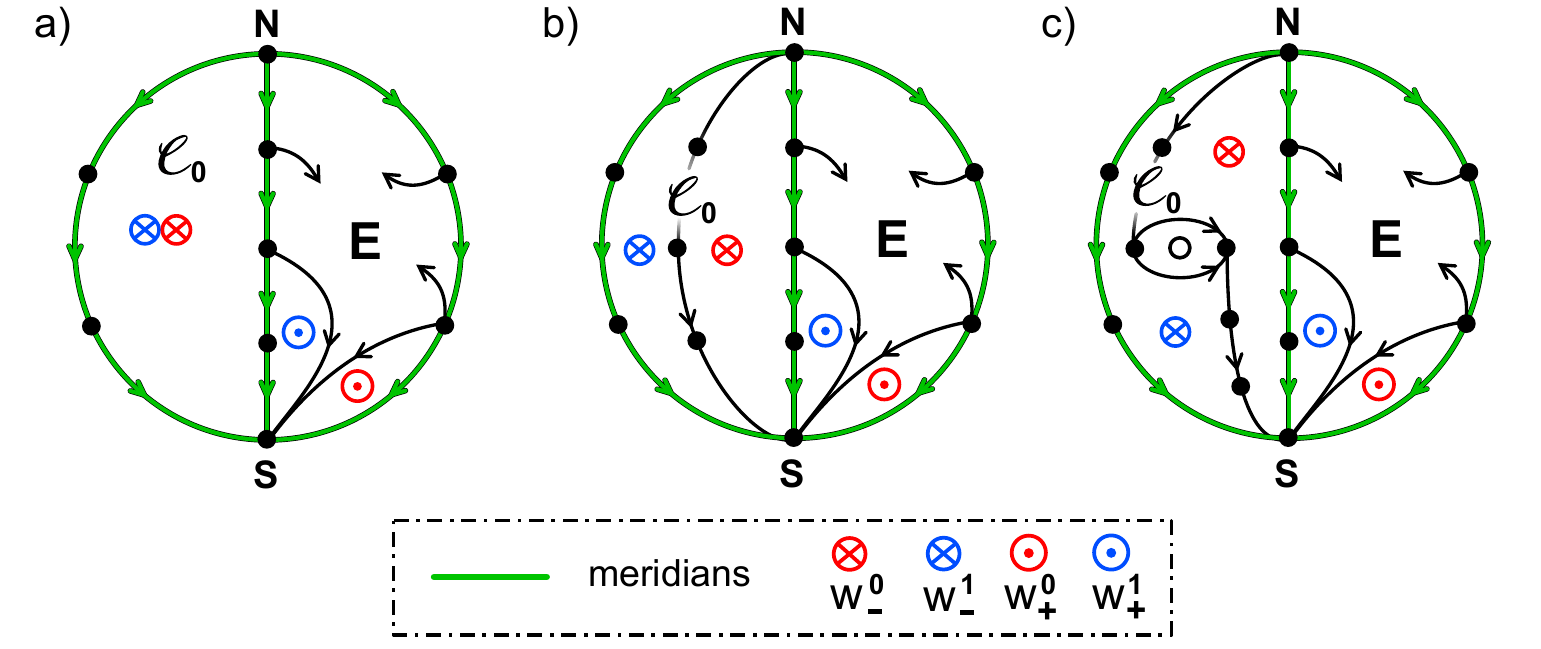}
\caption{\emph{
Three EastWest complexes $\mathcal{C}_0$ as Western, left, disks in 3-cell templates.
See also fig.~\ref{fig:1.0} for an example, and fig.~\ref{fig:1.1} for general notation.
(a) A single-face $(2,3)$-gon $\mathcal{C}_0$, with barycenter $w_-^0=w_-^1$.
See \ref{fig:2.2} for a general description of $(m,n)$-gons.
(b) A double face lift with two faces.
Each face takes care of one meridian.
(c) An eye lift where $\mathcal{C}_0$ possesses one interior closed face, the eye, which is detached from the two meridians.
}}
\label{fig:4.2}
\end{figure}

Note that the lift construction is compatible with the trivial equivalence group $\langle \kappa, \rho \rangle$ of section~\ref{sec3}.
The Eastern/Western property of $\mathcal{C}_\pm$ gets swapped by $\kappa$ but is invariant under $\rho$, by table~\ref{tbl:3.1}.
Hence the EastWest property of $\mathcal{C}_0$ is invariant under $\kappa, \rho$.
Therefore trivially equivalent Western or Eastern Sturm disks lift to trivially equivalent Sturm 3-balls by trivially equivalent EastWest disks.

The lift by an EastWest disk $\mathcal{C}_0$ is easily described, in terms of the resulting 3-cell template $\mathcal{C}$ and figs.~\ref{fig:1.1},  \ref{fig:4.2}.

\begin{figure}[t!]
\centering \includegraphics[width=0.95\textwidth]{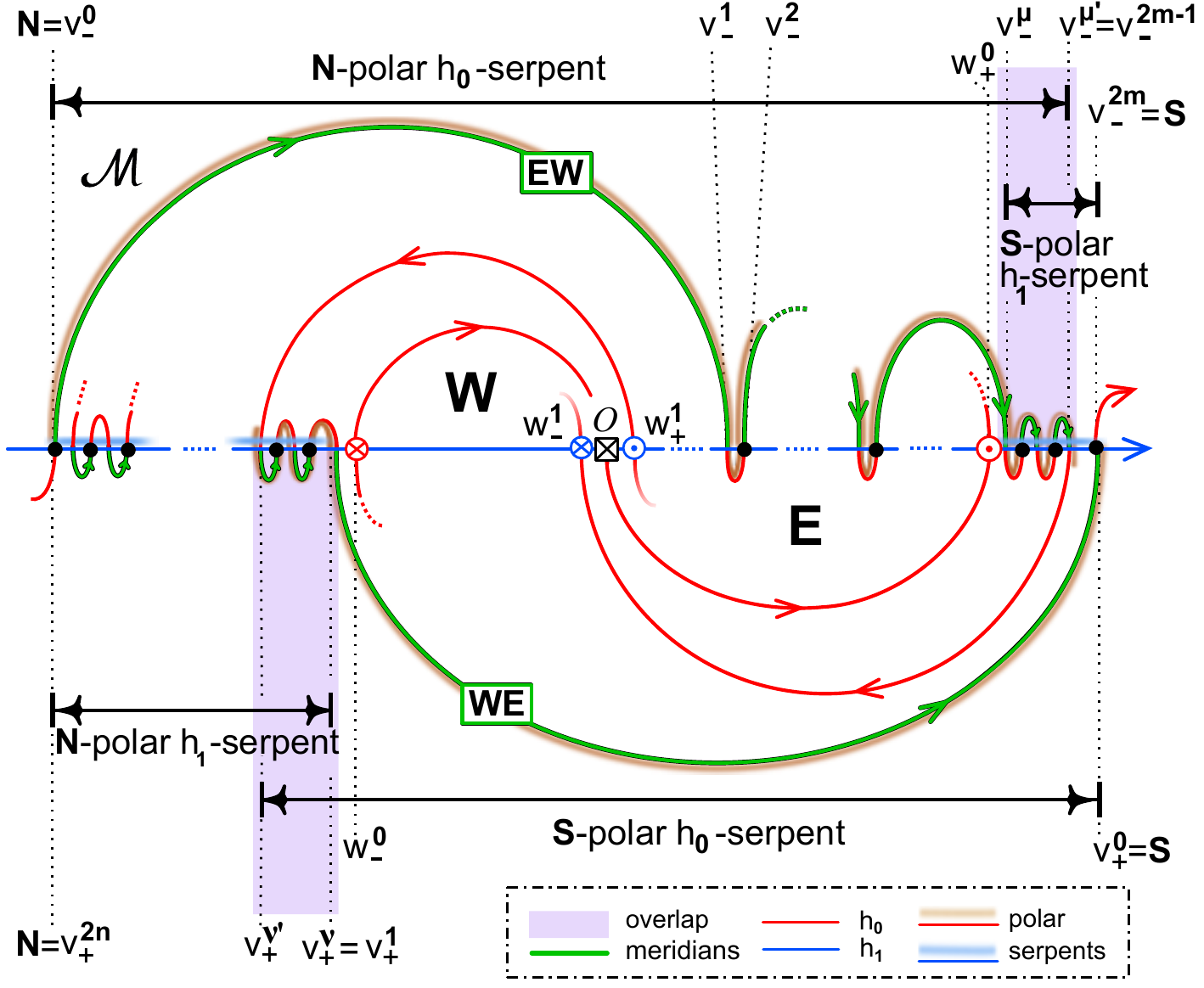}
\caption{\emph{
The 3-meander template of a West lift of an Eastern disk $\text{clos } \mathbf{E} = \mathcal{C}_+$ by a (Western) EastWest disk $\text{clos } \mathbf{W} = \mathcal{C}_0$.
Note the full $\mathbf{N}$-polar serpents, inherited from the Eastern disk $\mathbf{E}$.
The $\mathbf{S}$-polar serpents are not full, in general, but are overlapped completely by their full $\mathbf{N}$-polar counterparts.
This leads to a subtle simplification of the general 3-meander template, fig.~\ref{fig:1.2}.
}}
\label{fig:4.3}
\end{figure}

A West lift results in meridian faces $\mathbf{NW}$ and $\mathbf{NE}$ which stretch all the way to the South pole $\mathbf{S}$, by definition of $\text{clos }\mathbf{W} = \mathcal{C}_0$.
In the notation of figs.~\ref{fig:1.1} and \ref{fig:1.2},
	\begin{equation}
	\mu'=2m-1\,,\qquad \nu = 1\,.
	\label{eq:4.3}
	\end{equation}
In terms of the resulting 3-meander template, fig.~\ref{fig:1.2}, the non-overlap parts $v_-^{\mu'+1} \ldots v_-^{2m-1}$ and $v_+^{\nu-1} \ldots v_+^1$ of $\mathbf{S}$-polar serpents $h_\iota$ disappear.
This leads to the subtle difference between fig.~\ref{fig:1.2} and fig.~\ref{fig:4.3}.

For the East lift of a Western disk $\mathcal{C}_-$ by an (Eastern) EastWest disk $\mathcal{C}_0$ we analogously obtain
	\begin{equation}
	\mu = 1\,,\qquad \nu'= 2n-1\,.
	\label{eq:4.4}
	\end{equation}
Since the East lift is related to the West lift by the trivial equivalence $\kappa u=-u$, we can obtain the resulting 3-meander template of fig.~\ref{fig:1.2} or fig.~\ref{fig:4.3} by a 180$^\circ$ rotation of the shooting curve.

We mention the three most elementary examples of EastWest complexes $\mathcal{C}_0$ and their associated face lifts; see fig.~\ref{fig:4.2}.
We describe all lifts as West lifts, i.e. with $\mathcal{C}_0$ in the Western role.
The Eastern role can easily be obtained by the trivial equivalence $\kappa$ which reverses bipolar orientation; see fig.~\ref{fig:3.1} and table~\ref{tbl:3.1}.

A single-face disk $\mathcal{C}_0$ is always an $(m,n)$-gon and always EastWest; see fig.~\ref{fig:4.2}(a).
We call the lift by $\mathcal{C}_0$ simply a \emph{single-face lift}.
The most frequent case, below, will involve meridians which consist of a single directed edge, each.
We call this the \emph{minimal face lift}.
The associated $(1,1)$-gon $\mathcal{C}_0$ is the planar Chafee-Infante attractor $\mathcal{A}_{\text{CI}}^2$.

\begin{figure}[p!]
\centering \includegraphics[width=\textwidth]{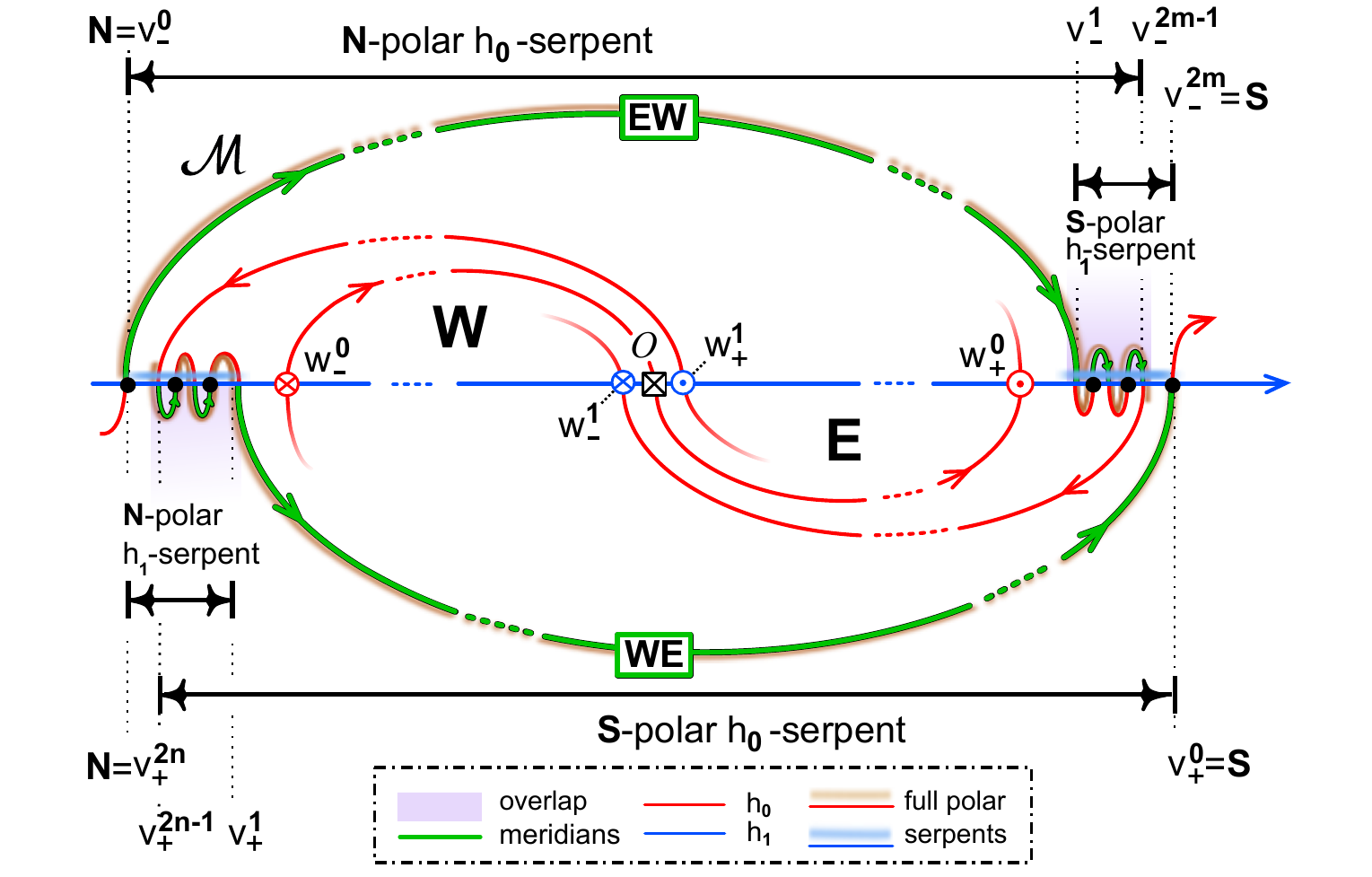}
\caption{\emph{
The 3-meander template resulting from the lift of a general (Western) EastWest disk $\mathcal{C}'_0$ by a general (Eastern) EastWest disk $\mathcal{C}_0$, welded at the shared $(m+n)$-gon meridian boundary.
Interchanging the Eastern and Western roles of $\mathcal{C}_0$ and $\mathcal{C}'_0$ interchanges $m$ and $n$.
}}
\label{fig:4.4}
\end{figure}

\begin{figure}[p!]
\centering \includegraphics[width=\textwidth]{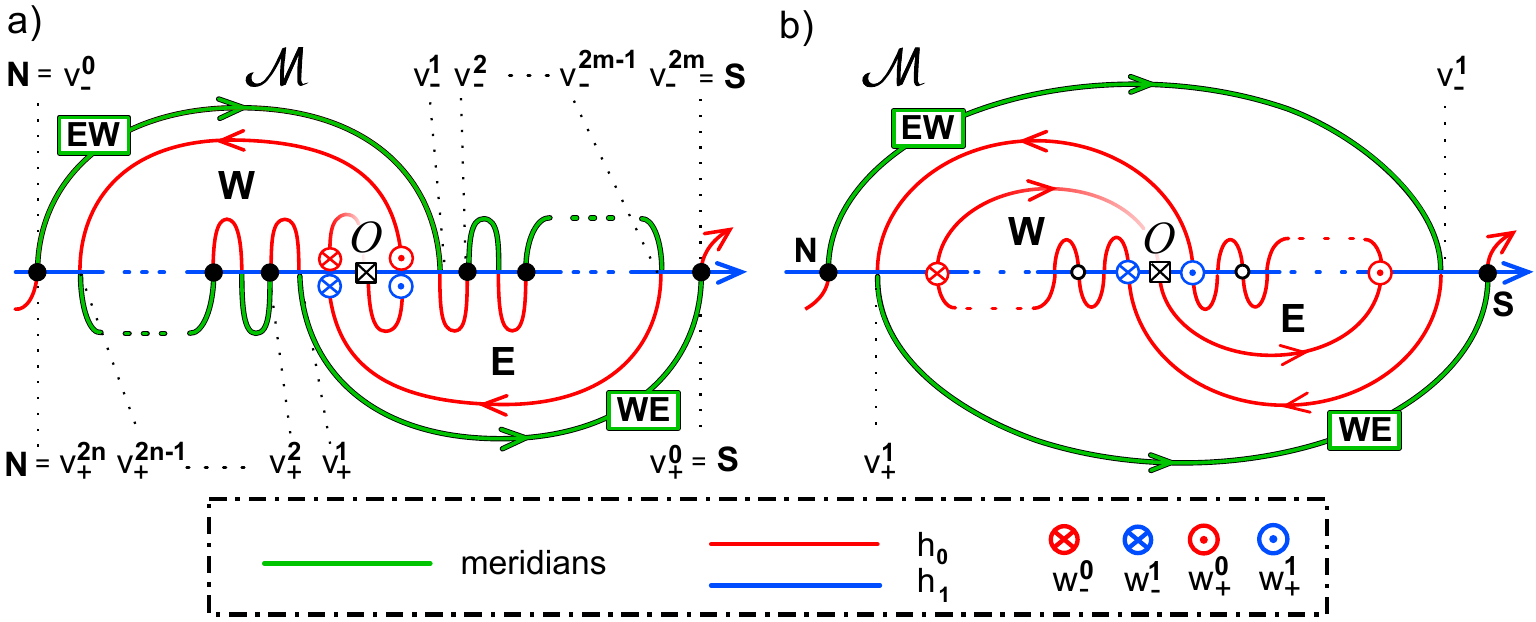}
\caption{\emph{
Two examples of 3-meander templates involving $(m+n)$-gons.
(a) The single face lift of a Western $(n,m)$-gon by an Eastern $(m,n)$-gon, called pitchfork lift.
Note the modification of the planar $(m,n)$-gon meander of fig.~\ref{fig:2.2}(b) by a pitchfork bifurcation of the face center $\mathcal{O}$.
(b) The lift of a Western $n$-striped disk by an Eastern $m$-striped disk.
Note the resulting unstable double cone suspension, of the planar $(m,n)$-gon meander of fig.~\ref{fig:2.2}(b), by a 180$\,^\circ$ rotation and the addition of two new polar arcs $\mathbf{N} v_-^1$ and $v_+^1 \mathbf{S}$.
}}
\label{fig:4.5}
\end{figure}

A \emph{double-face lift} involves any EastWest disk $\mathcal{C}_0$ with two faces.
The two distinct faces $c_{w^0}$ and $c_{w^1}$ are then separated by a third pole-to-pole path in the 1-skeleton $\mathcal{C}_0^1$, interior to $\mathcal{C}_0$, in addition to the two boundary paths.
See fig.~\ref{fig:4.2}(b).
If the three paths consist of a single directed edge, each, we speak of a \emph{minimal double-face lift}.

An \emph{eye lift} involves any EastWest disks $\mathcal{C}_0$ with three faces:
the two meridian faces $c_{w^0},\ c_{w^1}$, and a third face $c_v$ which we call the \emph{eye}.
See fig.~\ref{fig:4.2}(c) for the general configuration.
In general, the closure $\overline{c}_v$ will be interior to $\mathcal{C}_0$, detached from poles and meridians.
However, we also admit the \emph{degenerate eye lift} variants when
	\begin{equation}
	\partial c_v \,\cap \,\partial \mathcal{C}_0
	\label{eq:4.6}
	\end{equation}
consists of one or both poles.
The \emph{minimal eye lift} is the degenerate case where $\mathcal{C}_0$ is \emph{striped} vertically into three Chafee-Infante disks $\mathcal{A}_{\text{CI}}^2$, alias (1,1)-gons, by a total of four pole-to-pole edges, two interior plus two meridian boundaries.

We conclude this section with a brief look at lifts of EastWest disks $\mathcal{C}_0$ by boundary compatible EastWest disks $\mathcal{C}'_0$.
Then \eqref{eq:4.3}, \eqref{eq:4.4} imply
	\begin{equation}
	\mu=1\,,\ \mu'=2m-1\,,\quad
	\nu=1\,,\ \nu'=2n-1
	\label{eq:4.8}
	\end{equation}
because we may interpret the lift as, both, an East lift or a West lift.
In particular all polar serpents are full and their non-overlap regions disappear.
See fig.~\ref{fig:4.4} for the resulting 3-meander template.

For example, let $\mathcal{C}_0$ be a single-face $(m,n)$-disk as in fig.~\ref{fig:2.2}(a), and choose $\mathcal{C}'_0$ to be the mirrored $(n,m)$-disk.
The lift
	$$\text{clos }\mathbf{W} = \mathcal{C}_-:= \mathcal{C}'_0\,,\qquad
	\text{clos }\mathbf{E} = \mathcal{C}_+:= \mathcal{C}_0$$
provides the 3-cell template $\mathcal{C}$ of fig.~\ref{fig:1.1}(b).
Note however that
	\begin{equation}
	w_\pm^0 = w_\pm^1\,,
	\label{eq:4.7}
	\end{equation}
and all other interior vertices are missing, because the hemispheres are single faced.
See fig.~\ref{fig:4.5}(a) for the resulting 3-meander.
In fact the 3-ball meander arises from the planar $(m,n)$-gon of fig.~\ref{fig:2.2}(b) by a supercritical pitchfork bifurcation of the face center $\mathcal{O}$. 
The simplest case is the Chafee-Infante 3-ball $\mathcal{A}_{\text{CI}}^3$, which arises here as a face lift of the Chafee-Infante disk $\mathcal{C}_0 =\mathcal{A}_{\text{CI}}^2$, alias the (1,1)-gon, by itself.


\begin{figure}[t!]
\centering \includegraphics[width=0.85\textwidth]{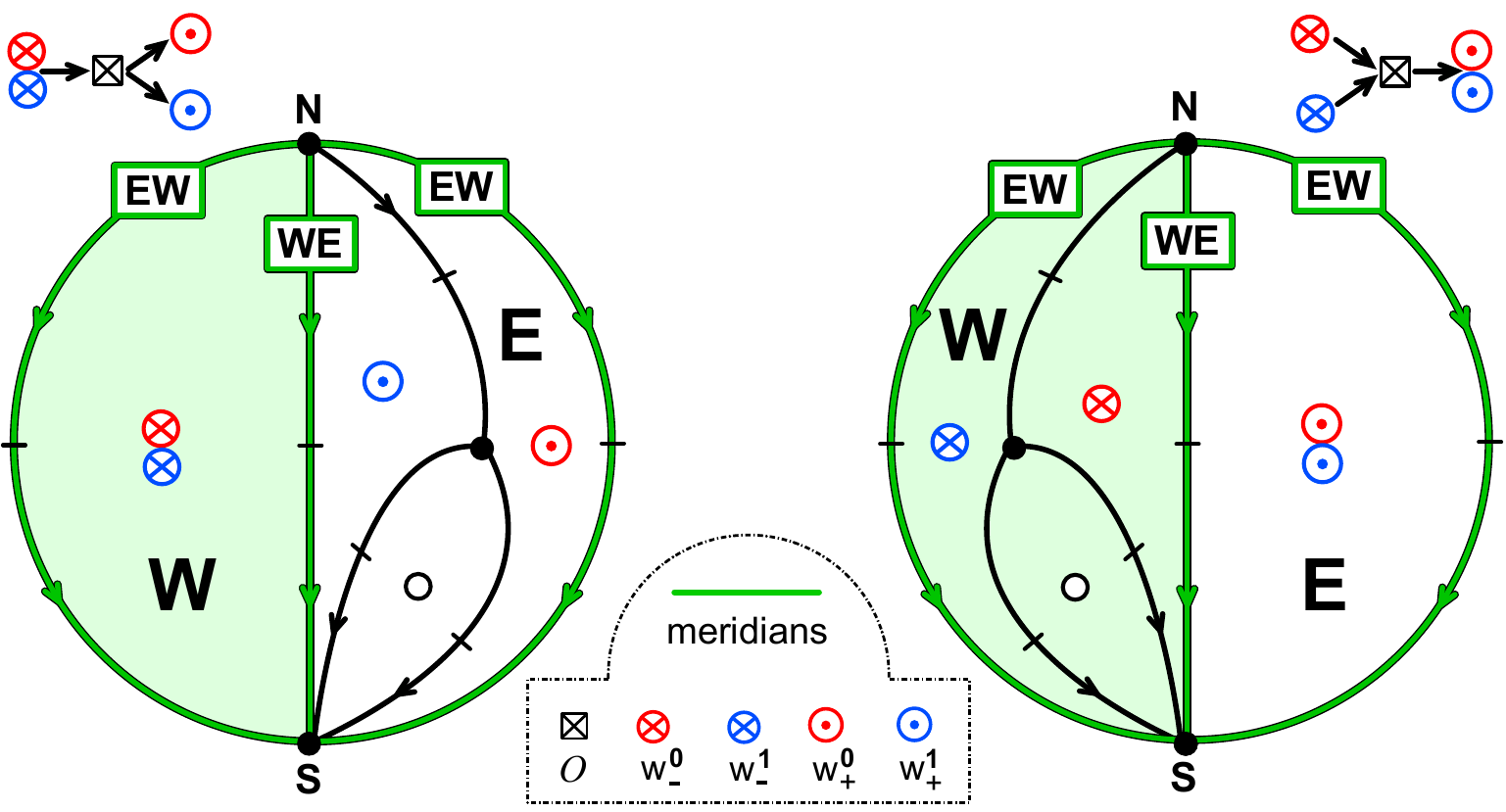}
\caption{\emph{
Two mirror-symmetric 3-ball attractors $\mathcal{A}^+$ (left) and $\mathcal{A}^-$ (right).
Note the degenerate eye lifts.
The attractors are not trivially equivalent, but result from lifts of the same EastWest disks $\mathcal{C}_0,\ \mathcal{C}'_0$ in swapped Eastern and Western roles.
}}
\label{fig:4.6}
\end{figure}

The $(m,n)$-\emph{striped} Sturm 3-ball is another example which involves the $(m,n)$-gon, though not at first sight.
Let $\mathcal{C}_0$ denote the $m$-striped EastWest disk which consists of $m$ Chafee-Infante disks, separated by $m-1$ single interior pole-to-pole edges.
The minimal double-face and eye disk above, correspond to the cases $m=2$ and $m=3$, respectively.
For $\mathcal{C}'_0$ we choose the $n$-striped EastWest disk.
Then we obtain the 3-meander of fig.~\ref{fig:4.5}(b).
This 3-meander coincides with the 2-meander of the $(m,n)$-gon of fig.~\ref{fig:2.2}(b), rotated by 180$^\circ$, with newly added overarching polar arcs $\mathbf{N}v_-^1$ and $v_+^1\mathbf{S}$.
In \cite{firo10} we have called the addition of such arcs a \emph{suspension}.
Indeed the resulting 3-ball attractor, in our case, is the one-dimensionally unstable suspension of the $(m,n)$-gon by a double cone construction with the resulting new poles $\mathbf{N},\ \mathbf{S}$ as attracting cone vertices.
The simplest case is the Chafee-Infante 3-ball $\mathcal{A}_{\text{CI}}^3$, again, which now arises as a suspension of $\mathcal{C}_0 =\mathcal{A}_{\text{CI}}^2$, alias the (1,1)-gon.

As a final caveat we recall an example from \cite{firo3d-1} which we redraw in fig.~\ref{fig:4.6}.
We have chosen a 3-face EastWest disk $\mathcal{C}_0$ of eye type, with the eye attached to the pole $\mathbf{S}$.
For $\mathcal{C}'_0$ we chose the EastWest Chafee-Infante disk $\mathcal{A}_{\text{CI}}^2$.
The two lifts only differ by the swapped roles of $\mathcal{C}_0$ and $\mathcal{C}'_0$ as Western and Eastern disks.
Note that the resulting 3-cell templates are not trivially equivalent.
Indeed, table~\ref{tbl:3.1} asserts that any trivial hemisphere swap $\mathbf{W} \leftrightarrow \mathbf{E}$ is accompanied by a corresponding pole swap $\mathbf{N} \leftrightarrow\mathbf{S}$, due to a reversal of bipolar orientations.
It is interesting to compare this example with our previous remark on the appropriate lifting of trivially equivalent Western or Eastern disks by EastWest disks.


\section{Duality}
\label{sec5}

For planar Sturm attractors, duality was introduced in \cite{firo08}, \cite[section~2.4]{firo10} to assist in the enumeration of all cases with up to 11 equilibria.
In the present section we explore duality on the boundary 2-sphere $S^2 = \partial c_\mathcal{O}$ of 3-cell templates $\mathcal{C}= \mathcal{C}_f$ for Sturm 3-ball attractors $\mathcal{A}_f$.
The properties are quite different, in the two settings.
In the plane the duals turned out to be bipolar, and thus provided a duality between planar attractors which, essentially, corresponded to time reversal (!) inside the attractor plane.
For 3-balls, duals turn out bipolar, interior to each hemisphere $\mathbf{W}$ and $\mathbf{E}$, separately.
Across the welding meridians, however, all polarity disappears and directed paths keep circling forever.

In the planar case we have defined the 1-skeleton $\mathcal{C}^{*,1}$ of the oriented dual $\mathcal{C}^*$ of $\mathcal{C}$ as follows.
Vertices of $\mathcal{C}^{*,1}$ are the $i=2$ barycenters $w$ of faces $c_w \in \mathcal{C}$.
Edges $e^*$ of $\mathcal{C}^{*,1}$ run between any two barycenters $w$ of faces $c_w \in \mathcal{C}$ which are adjacent along any shared  edge $e= c_v \in \mathcal{C}^1$ with $i=1$ barycenter saddle $v$.
The direction of $e^*$ is chosen such that $e^*$ crosses $e$ from left to right at the intersection $v$.
In other words
	\begin{equation}
	\det (e^*,e) = +1
	\label{eq:5.1}
	\end{equation}
for the direction vectors $e$ and $e^*$ at $v$.
This construction requires two artificial pole vertices $\underline{v} = \mathbf{N}^*$ and $\overline{v} = \mathbf{S}^*$ of $\mathcal{C}^*$ to be introduced, outside $\mathcal{C}$, to start and terminate all directed edges $e^*$ which enter and leave through the boundary of $\mathcal{C}$, respectively.
By this construction, the planar complex $\mathcal{C}^*$ became regular, bipolar and contractible, i.e. a planar Sturm complex, for any planar Sturm complex $\mathcal{C}$.
See \cite{firo08, firo10} for further details.

For 3-cell templates $\mathcal{C}$, i.e. for 3-ball Sturm attractors, we employ the same construction on the 2-sphere complex $S^2 = \partial c_\mathcal{O} = \mathcal{C}^2$.
As in figs.~\ref{fig:1.0} and \ref{fig:1.1}, we use the standard planar orientation of $S^2$, when viewed from outside.
Again we require $e^*$ to cross $e$, left to right, in this orientation.
This defines the \emph{dual} 2-\emph{sphere complex} $\mathcal{C}^{*,2}$ of $S^2$.
See fig.~\ref{fig:5.1} for a sufficiently general example.
Because we are on the sphere, this time, there is no need to add any extra poles.

However, the dual complex $\mathcal{C}^{*,2}$ fails to be bipolar.
The poles $\mathbf{N},\ \mathbf{S}$ of $\mathcal{C}^2$, in fact, become faces $\mathbf{N}^*$:= $c_{\mathbf{N}}^*,\ \mathbf{S}^*$:= $c_{\mathbf{S}}^*$ of the dual $\mathcal{C}^{*,2}$ with \emph{polar circles} $\partial \mathbf{N}^*,\ \partial \mathbf{S}^*$ as boundaries.
Note how $\partial \mathbf{N}^*$ is oriented clockwise, and $\partial \mathbf{S}^*$ counter-clockwise, in our chosen orientation of $S^2$ and $\mathcal{C}^{*,1}$.

\begin{figure}[t!]
\centering \includegraphics[width=\textwidth]{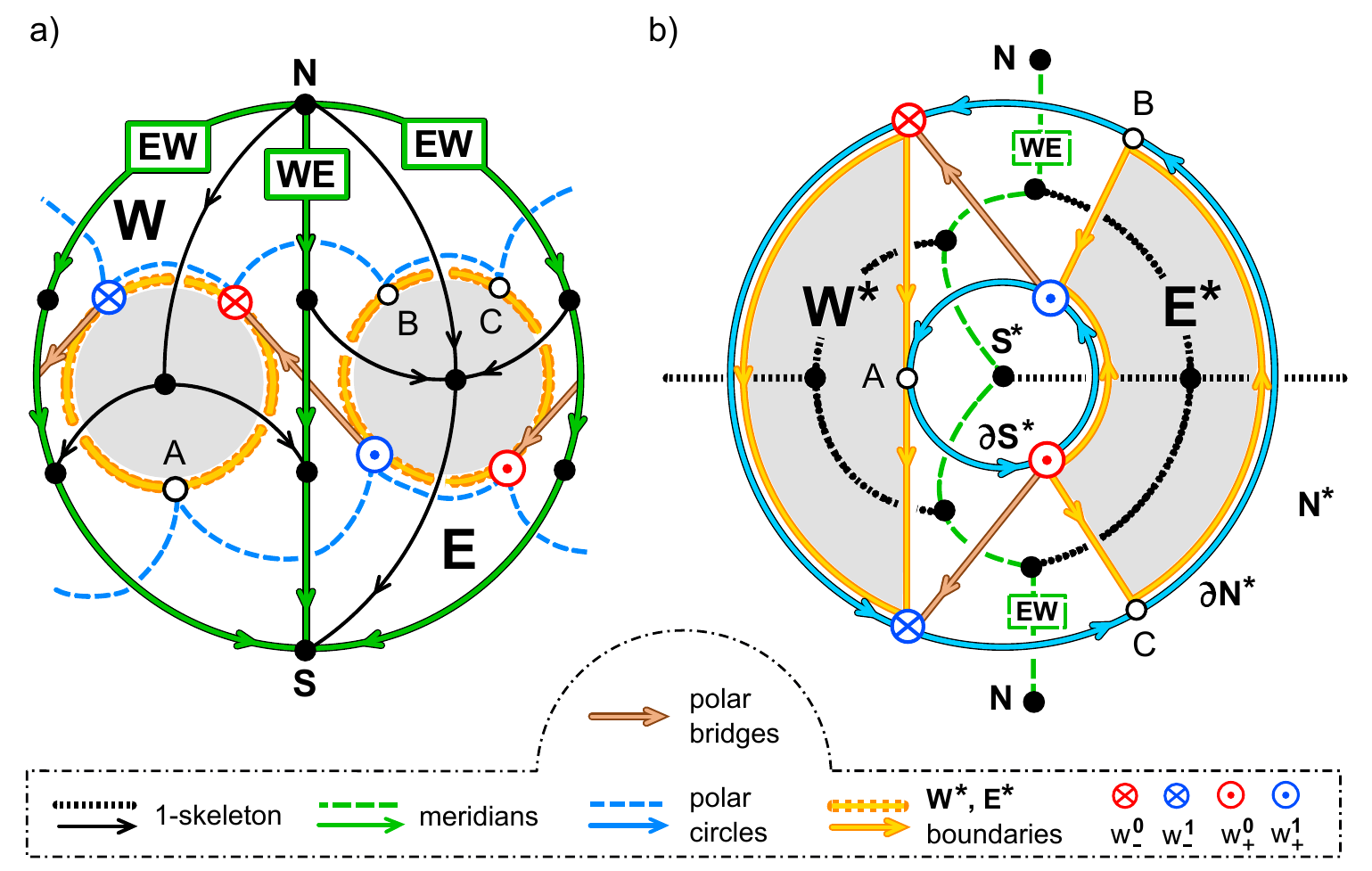}
\caption{\emph{
The dual complex $\mathcal{C}^*$ of an example 3-cell complex $\mathcal{C}$.
The seven face sources $\circ$ of $\mathcal{C}$ become vertices of the dual, $\mathcal{C}^*$.
The eight sources $\bullet$ of faces in $\mathcal{C}^*$ are the sink vertices of $\mathcal{C}$.
The orientation of $S^2 \subseteq \mathbb{R}^3$ is taken to be standard planar, when viewed from outside.
(a) The 1-skeleton $\mathcal{C}^{1}$ for the 3-cell template $\mathcal{C}$; see also fig.~\ref{fig:1.1}(a).
The boundaries $\mathbf{EW}$ are identified.
The thirteen edges $e$ of $\mathcal{C}$ (solid) are crossed by dual edges $e^*$ (dashed) from left to right.
(b) Schematics of the dual complex $\mathcal{C}^{*,1}$. 
Edges $e^*$ of $\mathcal{C}^*$ are solid, and $e$ of $\mathcal{C}$ are dashed.
The poles $\mathbf{N}, \mathbf{S}$ become the barycenters of the exterior face $\mathbf{N}^*$ and some interior face $\mathbf{S}^*$, respectively.
Note the annulus bounded by the dual polar circles $\partial \mathbf{N}^*, \partial \mathbf{S}^*$ (solid blue), which surround the pole $\mathbf{N}$ clockwise(!), and $\mathbf{S}$ counter-clockwise.
Bipolarity holds within each dual hemisphere core, $\mathbf{W}^* = \mathcal{C}_-^{2,*}$ and $\mathbf{E}^* = \mathcal{C}_+^{2,*}$ (both gray), with North poles $w_\pm^0$ and South poles $w_\pm^1$, and with dual meridians (orange).
The duals to overlap edges form single-edge bridges between the polar circles, directed from the dual South poles $w_\pm^1$ in one hemisphere core to the dual North poles $w_\mp^0$ in the opposite hemisphere core.
}}
\label{fig:5.1}
\end{figure}

By definition the polar circle $\partial c_{\mathbf{N}}^*$ contains a directed path segment from $w_-^0$ to $w_-^1$ in $\mathbf{W}$.
Likewise, $\partial c_{\mathbf{S}}^*$ contains a disjoint directed path segment from $w_+^0$ to $w_+^1$ in the Eastern hemisphere $\mathbf{E}$.

The edges $v$ and faces $w$ in $\mathcal{E}_\pm^2(\mathcal{O})$, i.e. inside each open hemisphere $\mathbf{W} = \Sigma_-^2(\mathcal{O}),\ \mathbf{E}= \Sigma_+^2(\mathcal{O})$, form a bipolar 1-skeleton $\mathcal{C}_-^{*,1},\ \mathcal{C}_+^{*,1}$, respectively.
Indeed, any dual circle of $\mathcal{C}_-^{*,1}$ in $\mathbf{W}$ would have to surround a source or sink of the bipolar orientation $\mathcal{C}^1$ in $\mathbf{W}$, as a face of $\mathcal{C}^{*,2}$, depending on the orientation of the dual cycle.
The only exceptions are the two faces of $w_\pm^\iota$, where the local poles of their boundaries both lie on the meridians, and one of their pole-to-pole boundaries, being contained in a meridian, has been deleted entirely.

Let $\mathbf{W}^* = \mathcal{C}_-^{+,2}$ and $\mathbf{E}^*= \mathcal{C}_+^{*,2}$ denote the resulting dual complexes if we include all the $i=0$ sinks of $\mathcal{C}$ in $\mathbf{W},\ \mathbf{E}$, respectively, as faces.
As planar bipolar, regular, and contractible cell complexes they must appear in our previous lists of planar Sturm attractors with the appropriate number of equilibria.
We call $\mathbf{W}^*$ and $\mathbf{E}^*$ the \emph{(dual) Western and Eastern core}, respectively.

Saddles $i_v=1$ of $\mathbf{WE}$ meridian edges $e=c_v$ generate dual edges $e^*$ which connect the Eastern core $\mathbf{E}^*= \mathcal{C}_+^{2,*}$ to the Western core $\mathbf{W}^*= \mathcal{C}_-^{2,*}$, directly.
For example, any overlap edge $e$ in $\mathbf{WE}$ guarantees a directed edge $e^*$ from $w_+^1$ in the South polar circle $\partial \mathbf{S}^*$ to $w_-^0$ in the North polar circle $\partial \mathbf{N}^*$.
Similarly, the $\mathbf{EW}$ overlap guarantees at least one directed dual edge from $w_-^1$ to $w_+^0$.
We call such edges \emph{polar bridges}.
Other directed edges across meridians may or may not exist.

With these remarks we have proved the only-if part of the following characterization of duals $\mathcal{C}^{2,*}$ of the 2-sphere complexes $\mathcal{C}^2$ of 3-cell templates $\mathcal{C}$.
Note that objects like the polar circles or $w_\pm^0,\ w_\pm^1$ may coincide, totally or in parts.
This leads to interesting special cases which we discuss afterwards.

\begin{lem}
A two-dimensional cell-complex $\mathcal{C}^{2,*}$ is the dual of the 2-sphere boundary complex $\mathcal{C}^2 = \partial c_\mathcal{O}$ of a 3-cell template $\mathcal{C} = \text{clos } c_\mathcal{O}$ if, and only if, the following conditions all hold.
\begin{itemize}
\item[(i)] $\mathcal{C}^{2,*}$ is a regular 2-sphere complex which decomposes into the disjoint union of
	\begin{itemize}
	\item[(a)] two polar faces $\mathbf{N}^*$ and $\mathbf{S}^*$;
	\item[(b)] Western and Eastern (dual) cores $\mathbf{W}^* = \mathcal{C}_-^{2,*}$ and $\mathbf{E}^* = \mathcal{C}_+^{2,*}$ which are closed planar Sturm complexes with North poles $w_\pm^0$ and South poles $w_\pm^1$, respectively;
	\item[(c)] two meridian duals  $\mathbf{EW}^*$ and $\mathbf{WE}^*$ of edges and faces.
	\end{itemize}
\item[(ii)] The polar circle $\partial \mathbf{N}^*$, seen from outside the 2-sphere, is right oriented, clockwise around $\mathbf{N}^*$.
The polar circle $\partial \mathbf{S}^*$, in contrast, is left oriented, counter-clockwise around $\mathbf{S}^*$.
The polar circles define disjoint directed polar \emph{segments} $w_-^0w_-^1 = \mathbf{W}^* \cap \partial \mathbf{N}^* \subset \partial \mathbf{W}^*$ from $w_-^0$ to $w_-^1$ on $\partial \mathbf{N}^*$, and $w_+^0w_+^1 = \mathbf{E}^* \cap \partial \mathbf{S}^*  \subset \partial \mathbf{E}^*$ from $w_+^0$ to $w_+^1$ on $\partial \mathbf{S}^*$, but the polar circles may intersect otherwise.
\item[(iii)] The pre-duals $e= -e^{**}$ to meridian edges $e^*$ in $\mathbf{EW}^*$ and $\mathbf{WE}^*$ define two disjoint directed paths $\mathbf{EW}$ and $\mathbf{WE}$, respectively, from the barycenter pole $\mathbf{N}$ of $\mathbf{N}^*$ to $\mathbf{S}$ of $\mathbf{S}^*$.
\item[(iv)] There exists at least one single-edge polar bridge $e_* = w_+^1 w_-^0 \in \mathbf{WE}^*$ from the South pole $w_+^1$ of the Eastern core $\mathbf{E}^* =\mathcal{C}_+^{2,*}$ to the North pole $w_-^0$ of the Western core $\mathbf{W}^* = \mathcal{C}_-^{2,*}$, and another single-edge polar bridge $e_* = w_-^1 w_+^0 \in \mathbf{EW}^*$ from the South pole $w_-^1$ of $\mathbf{W}^* = \mathcal{C}_-^{2,*}$ back to the North pole $w_+^0$ of $\mathbf{E}^* = \mathcal{C}_+^{2,*}$.
\end{itemize}
\label{lem:5.1}
\end{lem}

See fig.~\ref{fig:5.1}(b) for an illustration of conditions (i)--(iv) of the lemma.

\begin{proof}[\textbf{Proof.}]
For a proof of the only-if part see the remarks preceding the lemma.

To prove the if-part, we may assume that conditions (i)--(iv) of the lemma all hold. We have to show that the predual $\mathcal{C}$ of $\mathcal{C}^*$ defines a 3-cell template $\mathcal{C}$ which satisfies the required properties (i)--(iv) of definition~\ref{def:1.1}.
Strictly speaking we insert the 3-cell $c_\mathcal{O}$ here such that $\mathcal{C}^2 = \partial c_\mathcal{O}$ becomes the 2-sphere boundary; see condition (i) on $\mathcal{C}^{2,*}$.
This proves property (i) of definition~\ref{def:1.1}.

To prove the meridian decomposition property (ii) of definition~\ref{def:1.1} via $\mathcal{C}^1$, we first note that the predual vertices $\mathbf{N}$ and $\mathbf{S}$, i.e. the polar face barycenters of $\mathbf{N}^*$ and $\mathbf{S}^*$, respectively, become a bipolar source and sink vertex,   by condition (ii) and the edge directions \eqref{eq:5.1}.

By condition (iii) of the lemma, the disjoint directed meridian paths $\mathbf{EW}$ and $\mathbf{WE}$ are oriented from $\mathbf{N}$ to $\mathbf{S}$, as is appropriate.
We define the barycenters of $\mathbf{W}$ and $\mathbf{E}$ as the barycenters of the remaining core complexes $\mathbf{W}^* =\mathcal{C}_-^{2,*}$ and $\mathbf{E}^*=\mathcal{C}_+^{2,*}$, respectively, according to the given decomposition (i)(a)--(c).
The Sturm property of the dual cores, and in particular their bipolarity, implies the absence of cycles and poles within $\mathbf{W},\ \mathbf{E}$, separately, by standard duality.

Acyclicity on $\text{clos } \mathbf{W},\ \text{clos } \mathbf{E}$, as well as bipolarity on their union $\mathcal{C}^2$, will follow once we prove the orientation of edges towards and from the meridian boundaries, as required in property (iii) of definition~\ref{def:1.1}.
We only address $\mathbf{E},\ \mathbf{E}^*= \mathcal{C}_+^{2,*}$; the arguments for $\mathbf{W},\ \mathbf{W}^*=\mathcal{C}_-^{2,*}$ are analogous.
Consider the part of $\partial \mathbf{E}^*$ which is not part of the polar circle $\partial \mathbf{N}^*$.
In fig.~\ref{fig:5.1}(a), this is the lower part of $\partial \mathbf{E}^*$ (orange).
The saddle barycenters of edges $e^*$ in that boundary are precisely the barycenters of the (transverse) edges $e$ in $\mathbf{E}$ with one endpoint on $\partial \mathbf{E} \smallsetminus \mathbf{N}$.
We claim that such $e$ must be directed away from $\partial \mathbf{E}$.
This is obvious because $e^*$ must cross the directed edge $e$ left to right.

We note that the above argument remains valid, even if the two directed boundary paths from North pole $w_+^0$ to $w_+^1$ in $\partial \mathbf{E}^*$ overlap in parts, or coincide.
Still, all $e$ are then captured by the edges $e^*$ of $\partial \mathbf{E}^*$ which do not belong to the polar circle $\partial \mathbf{N}^*$.
This proves property (iii) of definition~\ref{def:1.1}.
It also completes the proof of property (ii) of definition~\ref{def:1.1}.

It remains to prove the overlap property (iv) of definition~\ref{def:1.1}, say, for the $\mathcal{C}^2$-faces
	\begin{equation}
	\mathbf{NE} = c_{w_-^0}\,,\qquad
	\mathbf{SW} = c_{w_+^1}\,.
	\label{eq:5.5}
	\end{equation}
Here $w_-^0$ is the North pole of $\mathbf{W}^* =\mathcal{C}_-^{2,*}$ and is located on the polar circle $\partial \mathbf{N}^*$.
Likewise $w_+^1$ is the South pole of $\mathbf{E}^*=\mathcal{C}_+^{2,*}$ and is located on the polar circle $\partial \mathbf{S}^*$.
By condition (iv) on $\mathcal{C}^{2,*}$ there exists a polar bridge $e_* \in \mathbf{WE}^*$ from $w_+^1$ to $w_-^0$.
By definition of duality, this means that the faces \eqref{eq:5.5} of $w_+^1$ and $w_-^0$ are edge adjacent to the predual $e \in \mathbf{WE}$ of $e_* \in \mathbf{WE}^*$.
This proves the overlap property (iv) of definition~\ref{def:1.1}.
The proof of the meridian edge overlap $\mathbf{NW}, \ \mathbf{SE}$ is analogous, indeed, and the lemma is proved.
\end{proof}

For later use we collect a few easy consequences of the previous lemma.

\begin{cor}
Let $\mathcal{C}^{2,*}$ be the dual of the 2-sphere boundary complex $\mathcal{C}^2 = \partial c_\mathcal{O}$ of a 3-cell template $\mathcal{C} = \text{clos } c_\mathcal{O}$.
Then the following six properties hold in $\mathcal{C}^{2,*}$.
\begin{itemize}
\item[(i)] Polar circles $\partial \mathbf{N}^*,\ \partial \mathbf{S}^*$ share a (dual) edge $e^*$ if, and only if, the pole distance $\delta$ between their barycenters $\mathbf{N},\ \mathbf{S}$ in the 1-skeleton $\mathcal{C}^1$ is 1.
\item[(ii)] The poles of the Western core $\mathbf{W}^*$ coincide, $w_-^0=w_-^1$, if, and only if, $\mathbf{W}^*= \lbrace w_-^0 \rbrace = \lbrace w_-^1 \rbrace$ is a singleton.
The analogous statement holds for the Eastern core $\mathbf{E}^*$ and $w_+^\iota$.
\item[(iii)] The edge distance between the two polar circles is at most 1, and is realized by at least two disjoint directed single-edge polar bridges. The bridges take the form $w_-^1 w_+^0 \in \mathbf{EW}^*$, directed from $w_-^1 \in \partial \mathbf{N}^*$ to $w_+^0 \in \partial \mathbf{S}^*$, and $w_+^1 w_-^0 \in \mathbf{WE}^*$, directed from $w_+^1 \in \partial \mathbf{S}^*$ to $w_-^0 \in \partial \mathbf{N}^*$. This provides at least one bridge between the polar circles, in each direction.
\item[(iv)] The disjoint directed polar segments $w_-^0 w_-^1$ and $w_+^0 w_+^1$ complement the directed single-edge bridges $w_-^1 w_+^0$, $w_+^1 w_-^0$ to at least one directed  cycle in $\mathcal{C}^{2,*}\,$:
	\begin{equation}
	w_-^0 w_-^1 w_+^0 w_+^1 w_-^0\,.
	\label{eq:5.6a}
	\end{equation}
\item[(v)] The directed polar segment $w_-^0w_-^1 \subseteq \partial \mathbf{N}^*$ is preceded and followed, on $\partial \mathbf{N}^*$, by the unique and disjoint intersections of the meridian duals $\mathbf{WE}^*$ and $\mathbf{EW}^*$ with $\partial \mathbf{N}^*$, respectively.
Analogously, the directed polar segment $w_+^0 w_+^1 \subseteq \partial \mathbf{S}^*$ is preceded and followed by the unique and disjoint edges $\mathbf{EW}^*\cap \partial \mathbf{S}^*$ and $\mathbf{WE}^* \cap \partial \mathbf{S}^*$, respectively.
\item[(vi)] Let $| \cdot |$ denote the edge length of paths and cycles.
Then the length of a polar segment relates to the length of its polar circle by
	\begin{equation}
	|w_-^0 w_-^1| \leq |\partial\mathbf{N}^*|-2\,,\qquad
	|w_+^0 w_+^1| \leq |\partial\mathbf{S}^*|-2\,.
	\label{eq:5.6b}
	\end{equation}
\end{itemize}
\label{cor:5.2}
\end{cor}

\begin{proof}[\textbf{Proof.}]
Claim (i) follows by definition because, equivalently, the predual edge $e$ connects the poles $\mathbf{N},\ \mathbf{S}$.
Claim (ii) follows, because the core $\mathbf{W}^*$ is bipolar.

Claim (iii) follows from lemma~\ref{lem:5.1}(iv) because the polar circles are connected by at least two single-edge bridges $e_\pm^*= w_\pm^1 w_\mp^0$ from $w_\pm^1$ to $w_\mp^0$, dual to edges on the two disjoint meridian paths from $\mathbf{N}$ to $\mathbf{S}$.
Claim (iv) follows from (iii).

Claim (v) follows from lemma~\ref{lem:5.1}(iii).
Indeed the maximal segment $w_-^0 w_-^1= \partial \mathbf{W}^* \cap \partial \mathbf{N}^*$ from core pole $w_-^0$ to $w_-^1$ on the polar circle $\partial \mathbf{N}^*$ must be preceded and followed by an edge dual to the meridian $\mathbf{WE}$ and $\mathbf{EW}$, respectively.
Since these meridians are disjoint, so are their duals.
Since this argument excludes at least two edges of $\partial \mathbf{N}^*$ from the segment, it also proves claim \eqref{eq:5.6b} of (vi).
This proves the corollary.
\end{proof}

\begin{figure}[t!]
\centering \includegraphics[width=\textwidth]{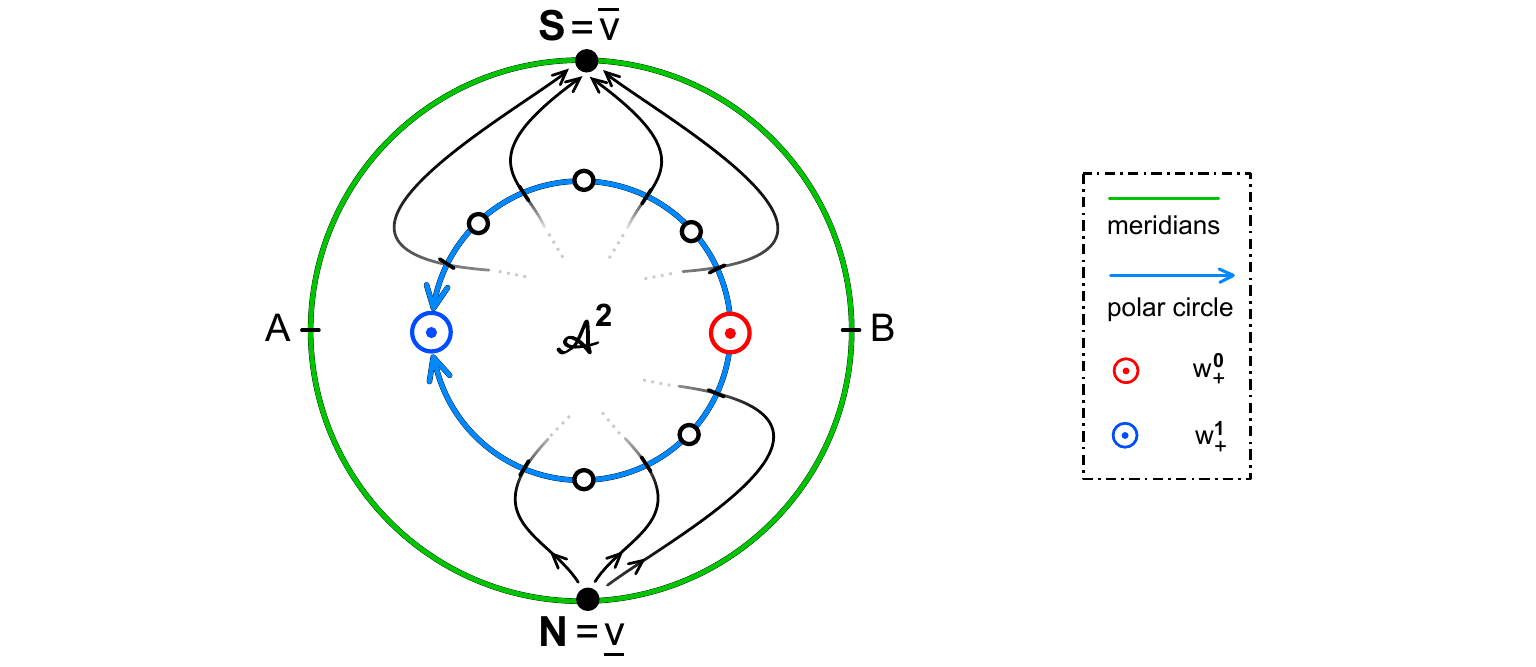}
\caption{\emph{
Minimal construction of an EastWest disk by the dual of any planar Sturm attractor $\mathcal{A}^2$.
Note how the poles $w_+^0$ and $w_+^1$ of $\mathcal{A}^2$ become faces adjacent to the extra edges of A and B which constitute the respective meridian boundaries.
Without these edges we obtain the standard construction of the planar dual Sturm attractor $\mathcal{A}^{2,*}$; see \cite{firo08, firo10}.
}}
\label{fig:5.2}
\end{figure}

We conclude this section with a few examples which relate the lift constructions of section~\ref{sec4} to duality.
First, we observe how duality allows us to, universally and minimally, convert any planar Sturm complex $\mathcal{A}^2$ to an EastWest disk.
We just replace the attractors $\mathcal{A}^2$ by its planar dual $\mathcal{A}^{2,*}$ and then surround $\mathcal{A}^{2,*}$ by the edges of two exterior saddles, as in fig.~\ref{fig:5.2}.

The new exterior poles $\mathbf{N}$ and $\mathbf{S}$ coincide with $\underline{v}$ and $\overline{v}$ of the planar duality construction in \cite{firo08, firo10}.
The extra edges of $A$ and $B$ close the dual $\mathcal{A}^{2,*}$ to become a Sturm EastWest disk.

The special case of a trivial one-point attractor $\mathcal{A}^2$ leads to the minimal single-face EastWest disk.
The special case of the trivial line $\sigma = \text{id} \in S_N$, with odd $N =2m-1$, leads to the minimal $m$-striped EastWest disk.
Note how $\mathcal{A}^2$ is one-dimensional.
The planar Chafee-Infante attractor $\mathcal{A}_\text{CI}^2$ leads to a double-face EastWest disk where only the interior pole-to-pole path  consists of two edges.
The Western eye disk of fig.~\ref{fig:4.2}(c) arises, in turn, if we take this resulting double-faced EastWest disk as the original attractor $\mathcal{A}^2$, and repeat the EastWest construction.


\section{The 31 Sturm 3-balls with at most 13 equilibria}
\label{sec6}

In this section we enumerate the Thom-Smale complexes of all 31 Sturm 3-balls, alias 3-cell templates $\mathcal{C}$, with at most $N = 13$ equilibria, up to the trivial equivalences of section~\ref{sec3}.
Our enumeration is based on the decomposition of their boundary 2-sphere $S^2= \partial c_\mathcal{O}$ into closed Eastern and Western Sturm disks with $\bar{N}_\mathbf{E}$ and $\bar{N}_\mathbf{W}$ equilibria.
We could invoke the results of \cite{firo10} on all planar Sturm attractors with at most 11 equilibria, select Eastern and Western Sturm disks, and weld shared boundaries.
To be more self-contained, and to prepare for section~\ref{sec7}, we proceed via the duals of section~\ref{sec5}, instead.
Brute force would yield 383 Sturm global attractors with 13 equilibria, up to trivial equivalences.
See \cite{fi94}.
We could simply extract all 3-ball cases, and dump them here.
Alas, what would we have understood?

Let $N_\mathbf{E}^*$ and $N_\mathbf{W}^*$ count the equilibria of the nonempty dual cores $\mathbf{E}^* = \mathcal{C}_+^{2,*}$ and $\mathbf{W}^* = \mathcal{C}_-^{2,*}$ of $\mathbf{E}$ and $\mathbf{W}$, respectively.
From lemma~\ref{lem:5.1}(i)(b) we recall that $\mathbf{E}^*$ and $\mathbf{W}^*$ are closed planar Sturm complexes.
By duality, the counts $N_\mathbf{E}^*, N_\mathbf{W}^*$ coincide with the equilibrium counts $N_\mathbf{E}, N_\mathbf{W}$ in the open hemispheres $\mathbf{E}, \mathbf{W}$, respectively.
We build up all Sturm 3-balls with $N\leq 13$ equilibria from the dual cores.
Notationally, we think of $\mathbf{E}$ and $\mathbf{W}$ as the originals here, and of $\mathbf{E}^*$ and $\mathbf{W}^*$ as their duals.

Let $M \geq 0$ denote the total number of non-polar sinks on the two meridians, which separate $M+2$ meridian edges.
Then the decomposition property of lemma~\ref{lem:5.1}(i) implies
	\begin{equation}
	2+M + (M+2) +N_\mathbf{E}^* +N_\mathbf{W}^*+1\quad = \quad
	N\quad \leq\quad 13\,.
	\label{eq:6.1}
	\end{equation}
Here the first summand~2 accounts for the two poles and the last summand~1 for the $i=3$ center $\mathcal{O}$ of the Sturm 3-ball.
Core and closed hemisphere counts are  related by
	\begin{equation}
	\begin{aligned}
	2+2 (M+1) &+ N_\mathbf{E}^* &=& \quad \bar{N}_\mathbf{E}^*\,;\\
	2+2 (M+1) &+ N_\mathbf{W}^* &=& \quad \bar{N}_\mathbf{W}^*\,.
	\end{aligned}
	\label{eq:6.2}
	\end{equation}
Since $M \geq 0$ and $N_\mathbf{E}^*,\ N_\mathbf{W}^* \geq 1$ we immediately obtain
	\begin{equation}
	2 \leq N_\mathbf{E}^* + N_\mathbf{W}^* \leq 8
	\label{eq:6.3}
	\end{equation}
from \eqref{eq:6.1}, and hence $N\geq 7$.
Since the total equilibrium count $N$ is odd, this leaves us with
	\begin{equation}
	N \in \lbrace 7,9,11,13 \rbrace\,.
	\label{eq:6.4}
	\end{equation}
Since $\mathbf{E}^*\ \mathbf{W}^*$ are (planar) Sturm attractors, the equilibrium counts $N_\mathbf{E}^*$ and $N_\mathbf{W}^*$ are also odd.
The trivial equivalence $\kappa$ allows us to interchange $\mathbf{W}$ and $\mathbf{E}$, if necessary.
In particular we may assume	
	\begin{equation}
	1\leq N_\mathbf{W}^* \leq N_\mathbf{E}^* \leq 7\,,\qquad \text{odd}\,,
	\label{eq:6.5}
	\end{equation}
without loss of generality.
This leaves us with the cases
	\begin{align}
	N_\mathbf{W}^*=1\,,\qquad N_\mathbf{E}^* 
	&\in \lbrace 1,3,5,7\rbrace\,, \qquad \text{and}\label{eq:6.6}\\
	N_\mathbf{W}^*=3\,,\qquad N_\mathbf{E}^* 
	&\in \lbrace 3,5\rbrace\,.\label{eq:6.7}
	\end{align}
In section~\ref{subsec6.1} we therefore list all planar Sturm attractors, alias cores $\mathbf{E}^*$, with up to 7 equilibria.
In \ref{subsec6.2} we discuss the $(m,n)$-gon suspensions, and in \ref{subsec6.3} the simple stripe patterns introduced in section~\ref{sec4}.
The two non-equivalent Sturm 3-balls of fig.~\ref{fig:4.6} are discussed in section~\ref{subsec6.4}.
In \ref{subsec6.5} we list the remaining cases arising from EastWest disks $\text{clos } \mathbf{W},\ \text{clos } \mathbf{E}$.
Purely Eastern, non-EastWest, disks are listed in section~\ref{subsec6.6}.
We summarize all results in the final section~\ref{subsec6.7}; see figs.~\ref{fig:6.3}, \ref{fig:6.4}, and tables~\ref{tbl:6.5}, \ref{tbl:6.6}.

\begin{figure}[t!]
\centering \includegraphics[width=\textwidth]{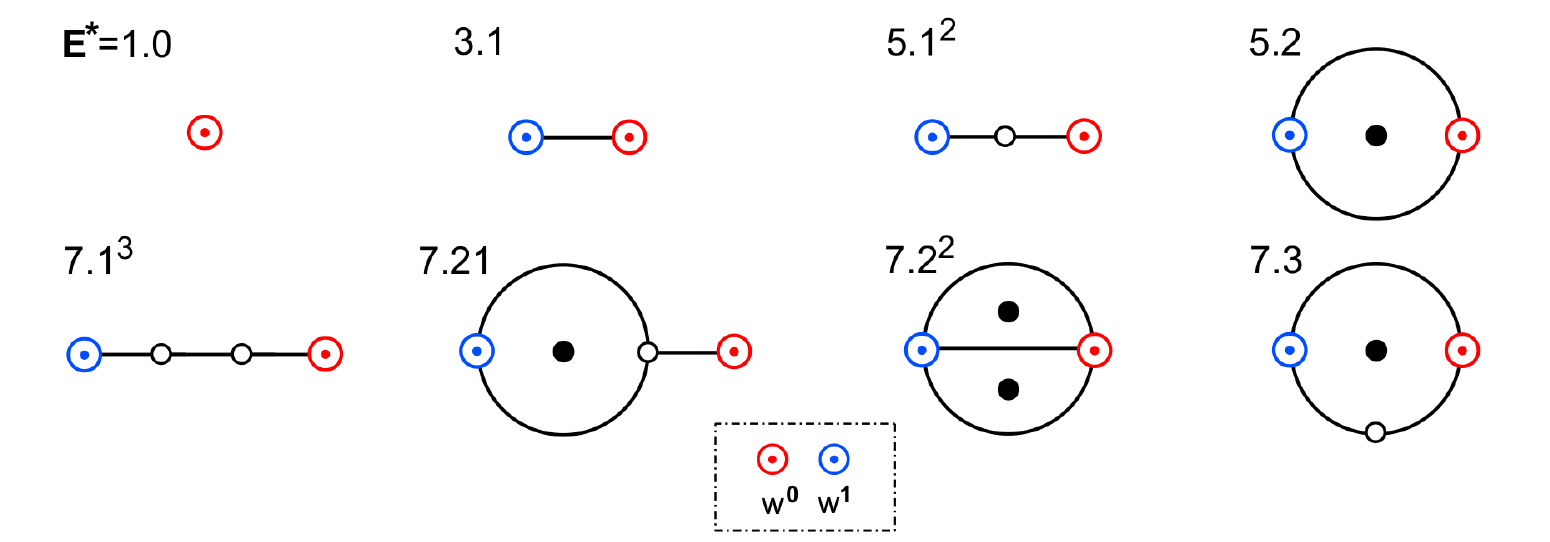}
\caption{\emph{
The eight Eastern dual cores $\mathbf{E}^*$, up to trivial equivalences, written as planar Sturm attractors with $N^* \leq 7$ equilibria.
See \eqref{eq:6.8} for the classification scheme.
Circles ``$\circ$'' indicate vertices of $\mathbf{E}^*$ and Morse index $i=2$ face barycenters of $\mathbf{E}$.
Dots ``$\bullet$'' are face barycenters of $\mathbf{E}^*$ and Morse stable $i=0$ vertices of $\mathbf{E}$.
The bipolar orientation of $\mathbf{E}^*$ runs from $w^0$ (red) to $w^1$ (blue), in each case.
}}
\label{fig:6.1}
\end{figure}

\subsection{The eight planar Sturm attractors with up to seven equilibria}\label{subsec6.1}

Let $\mathbf{E}^*$ be a planar Sturm attractor with $N^* \in \lbrace 1,3,5,7\rbrace$ equilibria.
Following \cite[section~3]{firo10} we choose the notation
	\begin{equation}
	N^*.\,n^{k_n}(n-1)^{k_{n-1}} \ldots 1^{k_1}- \ell\,;
	\label{eq:6.8}
	\end{equation}
where $n^{k_n}$ indicates a count $k_n$ of $n$-gon faces in the 1-skeleton of $\mathbf{E}^*$.
Edges which are not face boundaries are assigned $n=1$.
The postfix $\ell$ simply enumerates multiple configurations in somewhat arbitrary order.
We omit exponents 1.
The enumeration of cases is trivial, by planar Euler characteristic.
The results for odd $N^* \leq 7$ are listed in fig.~\ref{fig:6.1}.
To emphasize that $\mathbf{E}^*$ is a dual core to an Eastern hemisphere $\mathbf{E}$ we denote $i=0$ sinks of $\mathbf{E}^*$ by circles, ``$\circ$'', to indicate sources of $\mathbf{E}$, and $i=2$ sources of $\mathbf{E}^*$ by dots, ``$\bullet$'', to indicate sinks of $\mathbf{E}$.

\subsection{Pitchforked (m,n)-gons}\label{subsec6.2}

Let $N_\mathbf{E}^* = N_\mathbf{W}^* =1$, i.e., let $\mathbf{E}, \mathbf{W}$ be single faces, each.
Then the closed hemisphere disks $\text{clos } \mathbf{E}$ and
$\text{clos }\mathbf{W}$ are $(m,n)$-gons and $(n,m)$-gons, respectively, for compatibility.
See fig.~\ref{fig:2.2} and figs.~\ref{fig:4.4}, \ref{fig:4.5}(a).
In particular \eqref{eq:6.1} implies
	\begin{equation}
	M = m+n-2\,; \qquad m+n=(N-3)/ 2 \leq 5\,.
	\label{eq:6.9}
	\end{equation}
The trivial equivalence $\rho$, in table~\ref{tbl:3.1}, preserves each hemisphere but interchanges the boundaries by reflection through the $\pm 90$ ~degree meridian line.
This swaps $m$ and $n$, so that we may assume
	\begin{equation}
	1 \leq m\leq n
	\label{eq:6.10}
	\end{equation}
without loss of generality.
With the notation $(2(m+n)+1).(m+n)$, for $(m,n)$-gons and $(n,m)$-gons alike, we use the notation
	\begin{equation}
	\left( 2(m+n)+1\right).(m+n) | \left( 2(m+n)+1\right).(m+n)
	\label{eq:6.11}
	\end{equation}
for the resulting Sturm 3-ball.
We thus arrive at the case list of table~\ref{tbl:6.1}.
See also tables~\ref{tbl:6.5}, \ref{tbl:6.6}, cases 1, 3, 9, 10, 30, 31, and fig.~\ref{fig:6.3} for the six resulting 3-cell complexes.

\begin{table}[h!]
\centering \includegraphics[width=\textwidth]{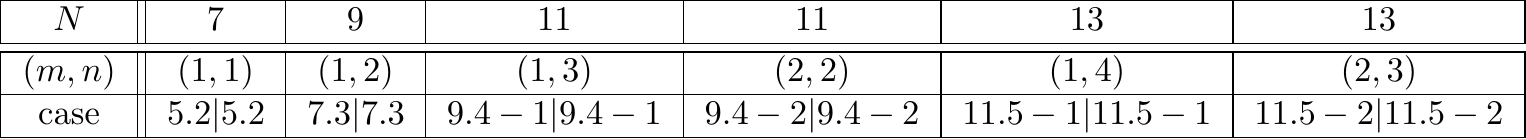}
\caption{\emph{
List of all six pitchforked $(m,n)$-gon 3-ball Sturm attractors with $N \leq 13$ equilibria, up to trivial equivalences.
The right entry duplicates the left entry, in each case label.
This covers all cases with trivial Sturm cores $N_{\mathbf{E}}^*= N_{\mathbf{W}}^*=1$.
}}
\label{tbl:6.1}
\end{table}

\begin{table}[b!]
\centering \includegraphics[width=0.75\textwidth]{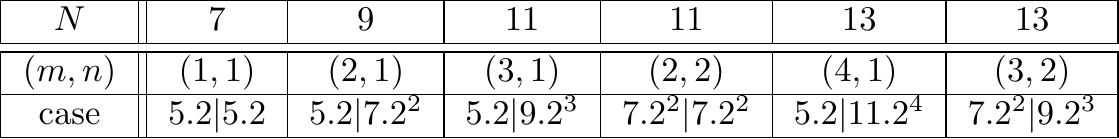}
\caption{\emph{
List of all six suspended $(m,n)$-gons, alias $(m,n)$-striped 3-ball Sturm attractors, with $N\leq 13$ equilibria, up to trivial equivalences.
Note the duplicated Chafee-Infante 3-ball $\mathcal{A}_{\text{CI}}^3 = (5.2|5.2)$ with $N=7$ equilibria, which also appears in table~\ref{tbl:6.1}, for $m=n=1$.
This covers all cases with one-dimensional Sturm cores $\dim \mathbf{E}^* = \dim \mathbf{W}^* =1$ and absent meridian sinks, $M=0$.
}}
\label{tbl:6.2}
\end{table}

\subsection{Striped suspensions of (m,n)-gons}\label{subsec6.3}

Let both dual cores $\mathbf{E}^* =(2m-1).1^{m-1},\ \mathbf{W}^*= (2n-1).1^{n-1}$ be one-dimensional, $m,n \geq 1$, and assume absence $M=0$ of non-polar meridian sinks.
Then $\text{clos } \mathbf{E} = (2m+3).2^m$ and $\text{clos } \mathbf{W} = (2n+3).2^n$ are planar with $m-1$ and $n-1$ pole-to-pole non-meridian edges, in addition to the two meridian edges, and with $m$ and $n$ faces, respectively.
This is the $(m,n)$-striped Sturm 3-ball of fig.~\ref{fig:4.5}(b), alias the unstably suspended $(m,n)$-gon.
In particular \eqref{eq:6.1} implies
	\begin{equation}
	m+n = (N-3)/2 \leq 5\,,
	\label{eq:6.12}
	\end{equation}
as in \eqref{eq:6.9}.
The trivial equivalence $\kappa$ lets us swap $\mathbf{W}$ and $\mathbf{E}$ so that 
	\begin{equation}
	\begin{aligned}
	2n-1 = N_\mathbf{W}^* \leq N_\mathbf{E}^* &= 
	2m-1\,,\quad \text{i.e.}\\
	n &\leq m
	\end{aligned}
	\label{eq:6.13}
	\end{equation}	
holds, without loss of generality.
Analogously to section~\ref{subsec6.2} and table~\ref{tbl:6.1} this provides the case list of table~\ref{tbl:6.2}.
Again see fig.~\ref{fig:6.3} and tables~\ref{tbl:6.5}, \ref{tbl:6.6} for the resulting cases 1, 2, 4, 6, 11, 16, five of them new.
For case 2 see also fig.~\ref{fig:1.0}.

\subsection{The triangle core}\label{subsec6.4}

Consider the triangle core $\mathbf{E}^* =7.3$; see fig.~\ref{fig:6.1}.
Then $N_{\mathbf{E}^*}=7$, and \eqref{eq:6.3}, \eqref{eq:6.1} imply $N_{\mathbf{W}^*}=1,\ M=0$.
Three edges $e$ of $\mathbf{E}$ cross the three edges of the dual triangle $\mathbf{E}^*$.
By the bipolar orientation of $\mathbf{E}^*$, from $w^0$ to $w^1$, two of these edges $e$ must be directed to $\mathbf{S}$, and the third edge must enter from $\mathbf{N}$.
This results in the  closed Eastern hemisphere, and hence the left 3-ball attractor $\mathcal{A}^+$, of fig.~\ref{fig:4.6}.
Swapping $\mathbf{W},\ \mathbf{E}$ by a 180$^\circ$ rotation of $\mathcal{A}^-$ in fig.~\ref{fig:4.1}, right, by trivial equivalence $\kappa \rho$, and reversing bipolar orientation, we obtain the inequivalent case where the core triangle $\mathbf{E}^*$ is flipped upside-down.
Derived from $\mathcal{A}^\pm$ we call these two cases
	\begin{equation}
	(5.2|11.3^22\pm)\,,
	\label{eq:6.14}
	\end{equation}
respectively.
See fig.~\ref{fig:6.3} and table~\ref{tbl:6.6}, cases 13 and 14.

\subsection{Multi-striped Sturm 3-balls}\label{subsec6.5}

All examples, so far, have been based on welding two compatible EastWest disks at their shared meridians.
We complete this list, in the present section.
The remaining cases, where at least one of the hemispheres is not of EastWest type, will be addressed in \ref{subsec6.6}.

Both hemispheres are EastWest disks if, and only if, only poles can be meridian vertices of interior edges $e\in \mathbf{E} \cup \mathbf{W}$.
In other words, edges of $\mathbf{E},\ \mathbf{W}$ can neither emanate from, nor terminate at, a sink vertex in $\mathbf{EW} \cup \mathbf{WE}$.
Equivalently, each boundary of the duals $\mathbf{E}^*,\ \mathbf{W}^*$ coincides with a polar circle segment of barycenters in $\mathbf{E},\ \mathbf{W}$, respectively.

The core list of fig.~\ref{fig:6.1} identifies the dual triangle~\ref{subsec7.3} as the only possibility where a directed edge path of $\mathbf{E},\ \mathbf{W}$ can branch at an interior sink.
This case has been treated in section~\ref{subsec6.4}, already.
All other interior sinks have degree two.
By the EastWest property, the same is true for the meridians.
We call Sturm disks with this degree two property \emph{multi-striped}.
Indeed all directed edge paths must then emanate from $\mathbf{N}$ and terminate at $\mathbf{S}$, because bipolarity excludes cycles.

It is therefore easy to enumerate all cases.
We simply place $N_0 \geq 1$ additional sinks inside any edges of any simply striped complex from section \ref{subsec6.3} and use trivial equivalences to reduce the number of cases.
The simply striped reference complex only has
	\begin{equation}
	N-2N_0 \geq 7
	\label{eq:6.15}
	\end{equation}
cells, of course; see \eqref{eq:6.4}.
Also note that at most one interior edge path may accomodate any additional sinks, and their number may only be one or two; see cases $5.2,\ 7.21,\ 7.2^2$ of fig.~\ref{fig:6.1}.
The restriction $N_\mathbf{W}^* \in \lbrace 1,3\rbrace$ of \eqref{eq:6.6}, \eqref{eq:6.7} does not accomodate interior sinks in $\mathbf{W}$.
Therefore it is sufficient to study $\text{clos }\mathbf{W}$ with $\text{clos }\mathbf{W}$ only inheriting the $M$ shared meridian sinks.
The case $N_\mathbf{W}^*=3$ of two faces in $\mathbf{W}$ simply amounts to one less interior sink available for $\mathbf{E}$.

The results are summarized in table~\ref{tbl:6.3}, ordered by the total number $N$ of equilibria and the total number $M$ of meridian sinks.
The Chafee-Infante ball $N=7$ has been treated  in sections~\ref{subsec6.2}, \ref{subsec6.3} already.
Consider $N=9$ next, with $N_0=1$.
Then the reference complex \eqref{eq:6.15} is the Chafee-Infante ball with one additional sink, necessarily on a meridian: $M=N_0=1$.
But any Chafee-Infante reference only leads to the pitchforked $(m,n)$-gon cases with
	\begin{equation}
	m+n= M+2 = N_0+2=(N-3)/2\,;
	\label{eq:6.16}
	\end{equation}
see \eqref{eq:6.9} and table~\ref{tbl:6.1}.
In particular the case $N=9$ can be omitted as a duplicate.

\begin{table}[]
\centering \includegraphics[width=0.85\textwidth]{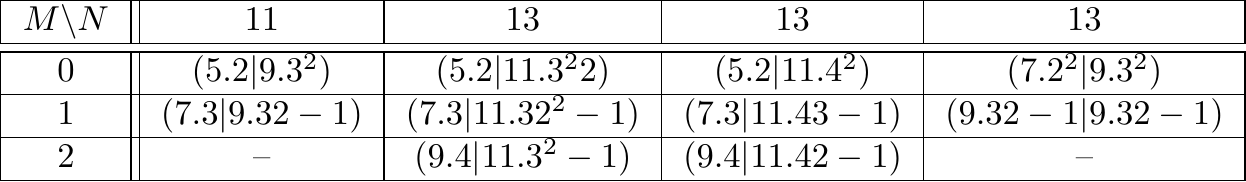}
\caption{\emph{
List of all 10 multi-striped 3-ball Sturm attractors with $N \leq 13$ equilibria, up to trivial equivalences.
Rows are ordered by the reference $(m,n)$-gon suspensions.
Chafee-Infante duplicates with the pitchforked $(m,n)$-gons of \ref{tbl:6.1} are omitted.
This covers all remaining cases of EastWest pairs of closed hemispheres.
}}
\label{tbl:6.3}
\end{table}

Consider $N=11$ next, first with $N_0=2$.
The reference complex \eqref{eq:6.15} then has $N-2N_0=7$ cells, and is omitted as a Chafee-Infante pitchforked $(m,n)$-gon, again.
Therefore $N_0=1$ and we have the unique $N-2N_0=9$ simply striped reference complex $(5.2|7.2^2)$ of table~\ref{tbl:6.2}, alias the triangle suspension.
Invoking trivial equivalence $\rho$ to interchange meridians, if necessary, we may assume the extra sink $N_0=1$ to either appear on the meridian $\mathbf{WE}$, with $M=1$, or interior to $\mathbf{E}$, with $M=0$.
This proves the $N=11$ column of table~\ref{tbl:6.3}.
See also cases 7 and 5 in table~\ref{tbl:6.5} and fig.~\ref{fig:6.3}.

For the remaining three columns, $N=13$.
The non-Chafee-Infante options are $N_0=1$ and $N_0=2$.
Consider $N_0=2$ first.
The simply striped reference complex has $N-2N_0=9$ equilibria and is the known triangle suspension.
For $M=N_0=2$, and up to trivial equivalence by $\rho$, we may either place the two extra sinks $N_0$ on the same meridian $\mathbf{WE}$, or else on one meridian each.
This proves the $M=2$ row of table~\ref{tbl:6.3}.
See also cases 27 and 25 in table~\ref{tbl:6.6} and fig.~\ref{fig:6.3}.

Next consider $N=13,\ N_0=2,\ M=1$.
Then we must place one extra sink on a meridian, say on $\mathbf{WE}$ by $\rho$, and the other extra sink on the only interior polar edge of the triangle suspension.
This yields case $(7.3|11.43-1)$.
For $N_0=2,\ M=0$ both extra sinks $N_0$ must go to the interior edge:
$(5.2|11.4^2)$.
This completes the third column of table~\ref{tbl:6.3}, and the case $N_0=2$.
See also cases 21 and 15 in table~\ref{tbl:6.6} and fig.~\ref{fig:6.3}.

It remains to consider $N=13,\ N_0=1$ with $N-2N_0=11$ reference equilibria.
This provides the two simply striped reference cases $(5.2|9.2^3)$ and $(7.2^2|7.2^2)$ of pitchforked quadrangles, in table~\ref{tbl:6.2}.
In table~\ref{tbl:6.5} and fig.~\ref{fig:6.3} these were the simply striped cases 4 and 6.
Placing the one extra sink $N_0$ on a meridian, $M=N_0=1$, or on any one of the interior edges, $M=0$, we obtain the remaining four cases of table~\ref{tbl:6.3}, up to trivial equivalences.
See also  cases 18, 23 and 12, 17 in the summarizing fig.~\ref{fig:6.3} and tables~\ref{tbl:6.5}, \ref{tbl:6.6}.

\subsection{Non-EastWest disks}\label{subsec6.6}

Non-EastWest disks require meridian sinks as targets.
Therefore $M \geq 1$.
Interior branchings of directed edge paths have been dealt with in section~\ref{subsec6.4} and can now be excluded.
Consider any directed path in $\mathbf{E}$.
By the boundary orientation of edges in 3-cell template hemispheres, definition~\ref{def:1.1}(iii), such a directed path must emanate from $\mathbf{N}$ or a meridian $i=0$ vertex, and has to terminate at $\mathbf{S}$.
In $\mathbf{W}$, similarly, any directed path has to emanate from $\mathbf{N}$ and must terminate at a meridian $i=0$ vertex or at $\mathbf{S}$.
We may therefore push any such directed path to emanate and terminate at the respective poles.
This provides a multiply striped 3-ball, with the exact same number of equilibria of the respective Morse numbers.
Conversely, we obtain all Non-EastWest disk Sturm 3-balls, by nudging at least one interior pole-to-pole directed path of the multiply striped 3-ball to start or terminate at an already existing $i=0$ meridian vertex, instead.
This leaves us with the rows $M=1$ and $M=2$ of table~\ref{tbl:6.3} as a reference for path nudging.

Consider $N=11$, for example, with $M=1$ reference $(7.3|9.32-1)$, case 7, of tables~\ref{tbl:6.3}, \ref{tbl:6.5} and fig.~\ref{fig:6.3}.
Nudging the unique interior edge $e \in\mathbf{E}$ to emanate from the unique extra $M=1$ sink on the meridian $\mathbf{WE}$, instead of $\mathbf{N}$, produces the unique case 8,
	\begin{equation}
	(7.3|9.32-2)
	\label{eq:6.17}
	\end{equation}
of a Sturm 3-ball with $N=11$ equilibria, where the Eastern disk $\text{clos } \mathbf{E}$ is not an EastWest disk.
See tables~\ref{tbl:6.4}, \ref{tbl:6.5} and fig.~\ref{fig:6.3}.

Consider the two reference cases $25,\ (9.4|11.3^2-1)$ and $27,\ (9.4|11.42-1)$ with $N=13,\ M=2$ next.
In case $25,\ (9.4|11.3^2-1)$, each meridian contains one extra sink.
The unique interior edge $e \in \mathbf{E}$ may emanate from the unique extra sink in $\mathbf{WE}$ or, equivalently under trivial equivalence $\rho$, from $\mathbf{EW}$.
This provides case 28,
	\begin{equation}
	(9.4|11.42-2)\,.
	\label{eq:6.18}
	\end{equation}
Similarly, nudgings of case $27,\ (9.4|11.42-1)$ lead to cases 26 and 29,
	\begin{equation}
	(9.4|11.3^2-2)\quad \text{and} \quad (9.4|11.42-3)\,.
	\label{eq:6.29}
	\end{equation}

\begin{table}[h!]
\centering \includegraphics[width=\textwidth]{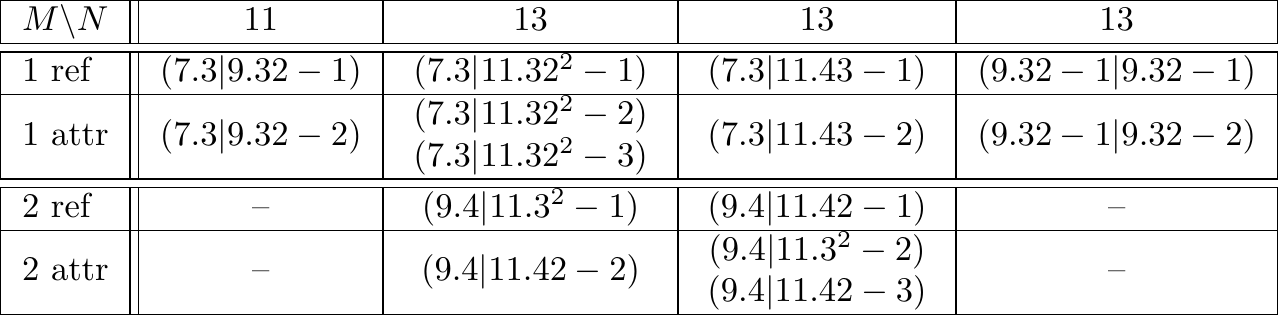}
\caption{\emph{
List of all eight non-EastWest 3-ball Sturm attractors with $N\leq 13$ equilibria, up to trivial equivalences.
Rows $\emph{M ref}$ with $M \geq 1$ meridian sinks refer to the multi-striped 3-ball Sturm attractors of table~\ref{tbl:6.3}, prior to nudging.
Rows $\emph{M attr}$ enumerate the resulting 3-ball Sturm attractors, after nudging.
This completes the listings of all 3-ball Sturm attractors with up to 13 equilibria.
}}
\label{tbl:6.4}
\end{table}

By similar arguments for $N=13,\ M=1$, the two reference cases $18,\ (7.3|11.32^2-1)$ and $21,\ (7.3|11.43-1)$ lead to cases 19, 20 and 22,
	\begin{equation}
	(7.3|11.32^2-2)\,,\ (7.3|11.32^2-3)\quad \text{and}\quad
	(7.3|11.43-2)\,,
	\label{eq:6.20}
	\end{equation}
respectively.
The remaining reference $23,\ (9.32-1|9.32-1)$ of table~\ref{tbl:6.3} only leads to the single case 24,
	\begin{equation}
	(9.32-1|9.32-2)
	\label{eq:6.21}
	\end{equation}
with the same nudged Eastern disk as in case 8, \eqref{eq:6.17}.
Nudging the Western disk, only, is trivially equivalent under $\kappa$.
Note that nudging of both hemispheres would violate the overlap condition of definition~\ref{def:1.1}(iv).
This completes the listing of all eight non-EastWest cases 8, 19, 20, 22, 24, 26, 28, 29, as summarized in tables~\ref{tbl:6.4}--\ref{tbl:6.6} and fig.~\ref{fig:6.3}.

\subsection{Summary}\label{subsec6.7}

The above hemisphere decompositions of 3-ball Sturm attractors define regular cell complexes of $S^2$, with additional structure.
It turns out that poles and meridians already define the bipolar orientation, for  $N\leq 13$ equilibria.
We therefore list the regular cell complexes, first, and then indicate the possible choices of poles and meridians, in each case.
We conclude with a list of all resulting Sturm permutations, their trivial isotropies and other elementary properties.

\begin{figure}[t!]
\centering \includegraphics[width=\textwidth]{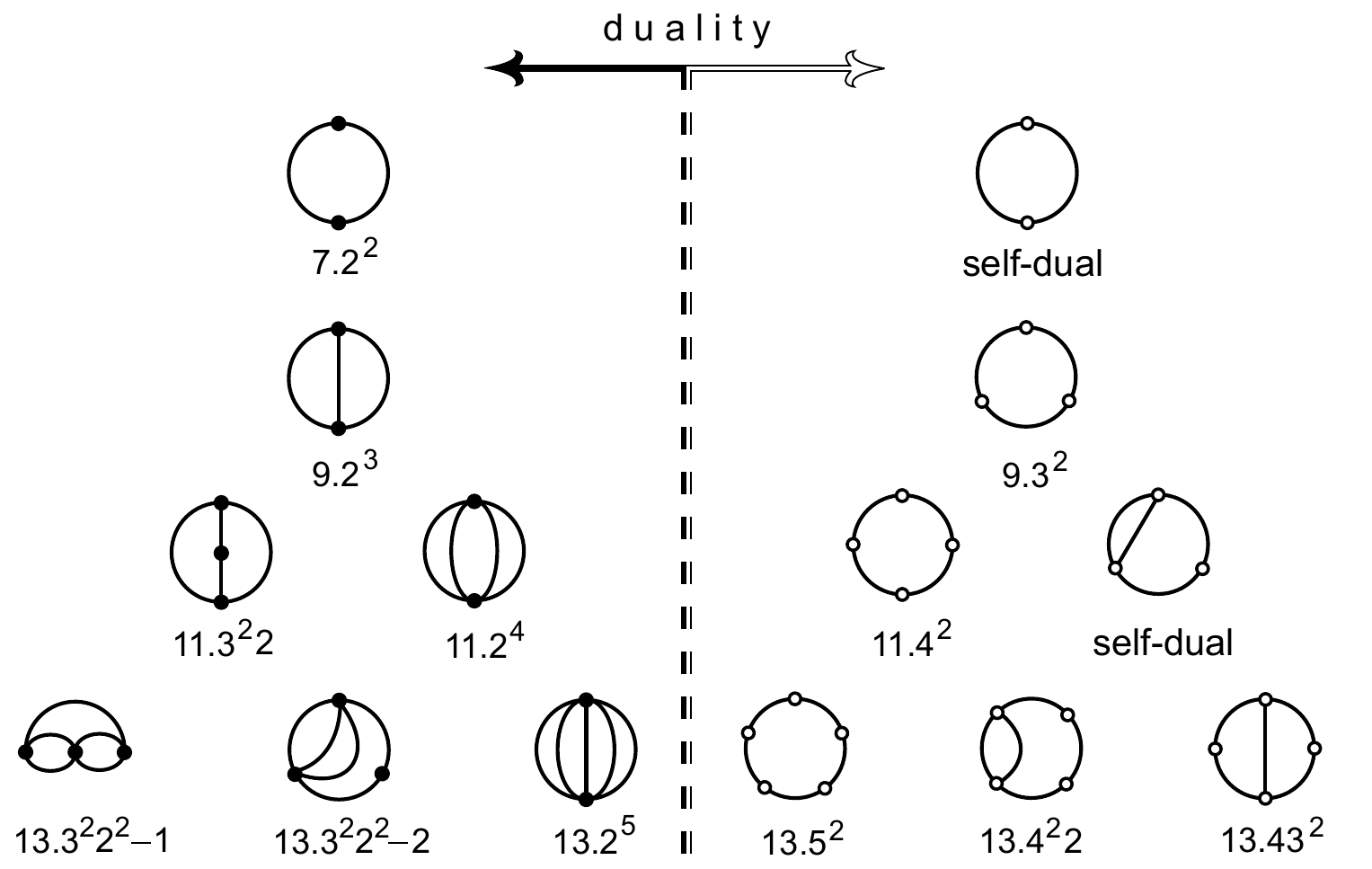}
\caption{\emph{
The twelve regular 2-sphere complexes, with at most $N-1 \leq 12$ cells on the sphere $\partial c_\mathcal{O}$.
The omitted 3-cell barycenter $\mathcal{O}$ contributes to the total count $N$ of cells, in the notation $N.n^{k_n}\ldots$ of \eqref{eq:6.8}.
The 2-sphere  is represented by one-point compactification of the plane, i.e. each 1-skeleton is drawn as a Schlegel graph.
In other words, we also consider the exterior as a face.
Left: $c_0 \leq c_2$, i.e. at least as many faces as zero-cell vertices.
Right: $c_0 \geq c_2$, by standard duality.
Note the two self-dual cases $7.2^2$ and $11.3^2 2$.
The cases $13.3^2 2^2-1$ and $-2$ differ by degrees 433 and 442 at their vertices, respectively.
See fig.~\ref{fig:6.3} for the associated Sturmian Thom-Smale complexes, which turn out to be nonunique quite frequently.
}}
\label{fig:6.2}
\end{figure}

See fig.~\ref{fig:6.2} for a list of all twelve regular $S^2$-complexes with at most $N-1=12$ cells.
(We omitted the 3-cell of $\mathcal{O}$.)
See \eqref{eq:6.8} for our notation of cases by face counts.
The list is easily derived as follows.
The cell counts $c_i$ count the cells of dimension $i$.
By Euler characteristic, $c_0-c_1+c_2=2$.
The total count is $c_0+c_1+c_2=N-1$.
Therefore 
	\begin{equation}
	c_0+c_2=(N+1)/2\,,\qquad c_1=(N-3)/2\,.
	\label{eq:6.22}
	\end{equation}
By standard duality, and because there exist at least two poles, we may assume
	\begin{equation}
	2 \leq c_0 \leq c_2
	\label{eq:6.23}
	\end{equation}
to obtain the left side of fig.~\ref{fig:6.2}.
Since $N \leq 13$, we only have to discuss the cases $c_0=2$ and $c_0=3$.

Consider $c_0=2$ first.
Then any of the $c_1=(N-3)/2 \geq 2$ edges must connect these two vertices directly, and each face is a 2-gon.
This provides the four cases $N.2^{c_1}$ of fig.~\ref{fig:6.2}, $N\in \lbrace 7,\ 9,\ 11,\ 13\rbrace$.

Consider $c_0=3$ next, and let $2 \leq d_1 \leq d_2\leq d_3$ denote the degrees at the three vertices.
Then \eqref{eq:6.22}, \eqref{eq:6.23} imply
	\begin{equation}
	N= 2(c_0+c_2)-1 \geq 4c_0-1=11\,,
	\label{eq:6.24}
	\end{equation}
i.e. $N\in \lbrace 11,\ 13\rbrace$.
Consider $N=11$ first.
Then
	\begin{equation}
	d_1+d_2+d_3=2c_1=N-3=8
	\label{eq:6.25}
	\end{equation}		
implies $d_3d_2d_1=422$ or $d_3d_2d_1=332$.
The four edges emanating from vertex 3 must terminate at vertices 1 and 2, in pairs.
The resulting complexes are not regular.
Therefore $d_3d_2d_1=332$.
Removing vertex 1, by $d_1=2$, leads to a case with $N=9$ and vertex degree 3.
This reduces to case $9.2^3$ and provides case $11.3^22$.

It only remains to consider $c_0=3,\ N=13$.
Then \eqref{eq:6.25} with $N-3=10$ implies
	\begin{equation}
	d_3d_2d_1 \in \lbrace 622, 532, 442, 433\rbrace\,.
	\label{eq:6.26}
	\end{equation}	
Absence of loops eliminates the case 622. In regular cell complexes it also eliminates the case  532.	
The case 442 reduce to $N=11$, $d_3d_2 = 44$, by removal of vertex 1, $d_1=2$.
The case $d_3d_2 =44$ occurs as case $11.2^4$, and leads to $13.3^22^2-2$.
In the remaining case 433 of \eqref{eq:6.26}, the absence of loops implies that the four edges of vertex 3 must terminate at vertices 1, 2, in pairs.
The remaining edge must join vertices 1 and 2.
This provides case $13.3^22^2-1$ and completes the list of fig.~\ref{fig:6.2}.

Based on the list of twelve regular $S^2$ cell complexes we could, in principle, determine all 3-cell templates, according to definition~\ref{def:1.1}, via the characterization of the duals in lemma~\ref{lem:5.1}.
We will follow such an approach for the Platonic solids, in section~\ref{sec7}.
Here we just summarize the results of  sections~\ref{subsec6.1}--\ref{subsec6.6} and assign the 31 known 3-cell templates to the 12 regular $S^2$ complexes; see fig.~\ref{fig:6.3}.

\begin{table}[p!]
\centering \includegraphics[width=\textwidth]{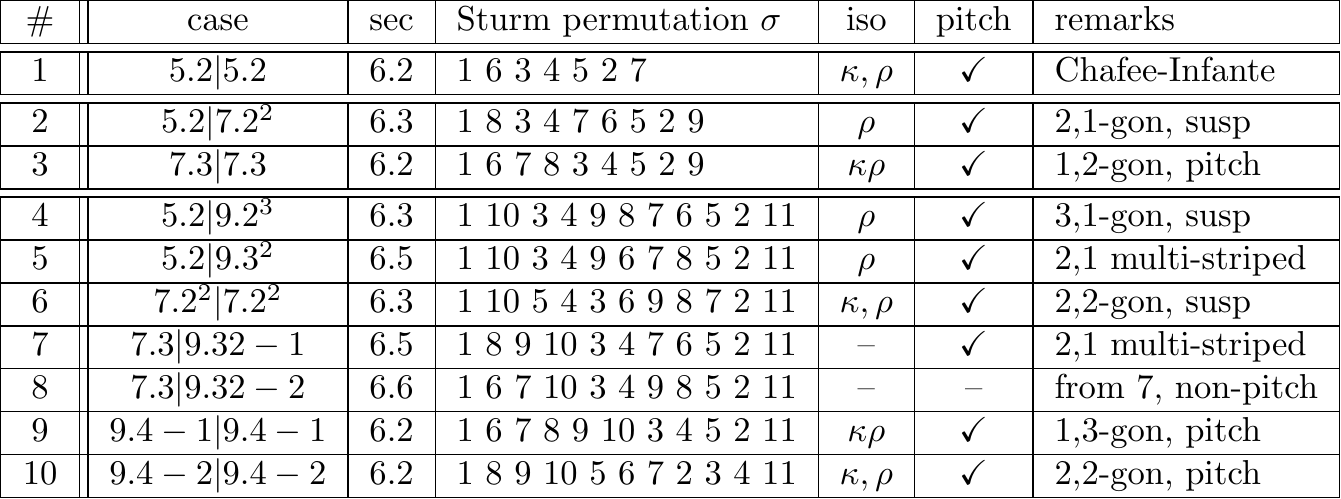}
\caption{\emph{
The 10 Sturm permutations $\sigma$ of 3-ball Sturm attractors with at most 11 equilibria, up to trivial equivalences.
The cases 1--10 are ordered by the $\mathbf{W} | \mathbf{E}$ hemisphere notation \eqref{eq:6.6}.
See section numbers for a detailed derivation, as indicated by ``remarks''.
The column ``iso'' lists the generators of trivial equivalences which leave $\sigma$ invariant, i.e. the trivial isotropy of $\sigma$ and the attractor.
For example $\rho$ indicates that $\sigma = \sigma^{-1}$ is an involution.
Note how the cases \#1, and 6, 10 of flip-isotropy $\kappa$ possess $N=7$ and $N=11$ equilibria, respectively, in compliance with corollary \ref{cor:3.2}.
The column ``pitch'' indicates that only example 8 is non-pitchforkable in the sense of \cite{furo91}.
}}
\label{tbl:6.5}
\end{table}

\begin{table}[p!]
\centering \includegraphics[width=\textwidth]{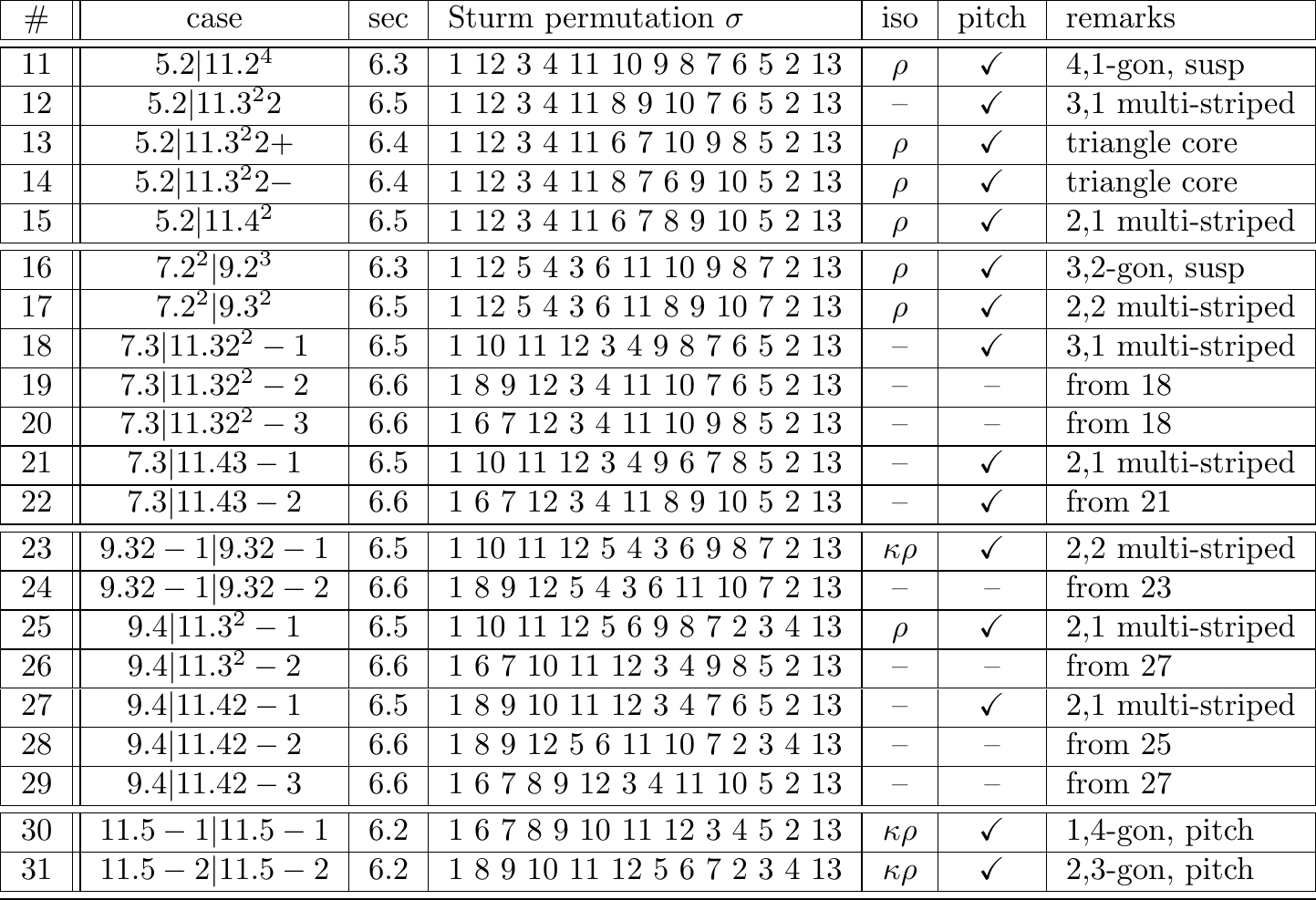}
\caption{\emph{
The 21 Sturm permutations $\sigma$ of 3-ball Sturm attractors with 13 equilibria.
For ordering and notation see table~\ref{tbl:6.5}.
Absence of flip-isotropy $\kappa$ for $N=13 \equiv 1\ (\mathrm{mod}\ 4)$ follows from corollary \ref{cor:3.2}.
The non-pitchforkable cases 19, 20, 24, 26, 28, 29 all reduce to case 8 of table~\ref{tbl:6.5} by a single pitchfork step.
}}
\label{tbl:6.6}
\end{table}

The bipolar orientations result, in each case, from the meridian and pole locations, together with the assignments of hemisphere labels $\mathbf{E},\ \mathbf{W}$.
The Sturm permutations $\sigma$ which generate each 3-cell template then follow from the SZS-pairs $(h_0,h_1)$.
See tables~\ref{tbl:6.5}, \ref{tbl:6.6} for the full list, and fig.~\ref{fig:6.4} for the resulting 3-meander templates.

By our derivation, any 3-ball Sturm global attractor with at most 13 equilibria appears in fig.~\ref{fig:6.3} and tables~\ref{tbl:6.5}, \ref{tbl:6.6}.
We order cases lexicographically, according to the notation \eqref{eq:6.6} for the closed hemisphere disks of $\mathbf{W}|\mathbf{E}$, and refer to the section where each case was defined and constructed.
Not surprisingly, each regular $S^2$-complex with at most 12 cells appears.
In fact, any regular $S^2$-complex is realizable a priori in the class of 3-ball Sturm attractors; see \cite{firo14}.

From the group $\langle \kappa, \rho \rangle$ of trivial equivalences in table~\ref{tbl:3.1}, nontrivial isotropy subgroups of trivial self-equivalences arise, occasionally, which leave the Sturm permutation and 3-cell template invariant.
These subgroups are manifest as symmetries, in fig.~\ref{fig:6.3}, or algebraically in the tables.
Absence of non-identity isotropy is marked by ``--'' in tables~\ref{tbl:6.5}, \ref{tbl:6.6}.
In view of corollary \ref{cor:3.2}, flip isotropy $\kappa$ can, and does, arise for numbers $N\equiv 3\ (\mathrm{mod}\ 4)$, only, i.e.~for $N=7$ and $N=11$ in our tables. See cases \#1, 6, 10, and note the even Morse counts $c_0, c_1, c_2$ in all those cases.

It is also interesting to compare the triangle core cases 13 and 14, i.e. $5.2|11.3^22\pm$, from the isotropy perspective.
See section~\ref{subsec6.4}.
Because each case is $\rho$-isotropic, only, we obtain the only trivially equivalent, but non-identical, case via the rotation $\kappa \rho$ (and orientation reversal).
This maps case 14 to the left case of the inequivalent examples in fig.~\ref{fig:4.5}.
The mirror symmetric, but inequivalent, right case is case 13, of course.
Inequivalence occurs due to the isotropy $\rho$:  the group orbit consists of only two elements.
Therefore a single group orbit of four trivial equivalences cannot cover all four reflected possibilities.

The column ``pitch'' indicates pitchfokable 3-balls, in the sense of \cite{furo91}.
These attractors can be generated, from the trivial $N=1$ attractor, by an \emph{increasing} sequence of pitchfork bifurcations. Here increasing means that each pitchfork in the sequence replaces
one equilibrium by three new ones. We do not allow the sequence to contain pitchforks which collapse
three equilibria into a single one.

\begin{figure}[p!]
\centering \includegraphics[width=0.935\textwidth]{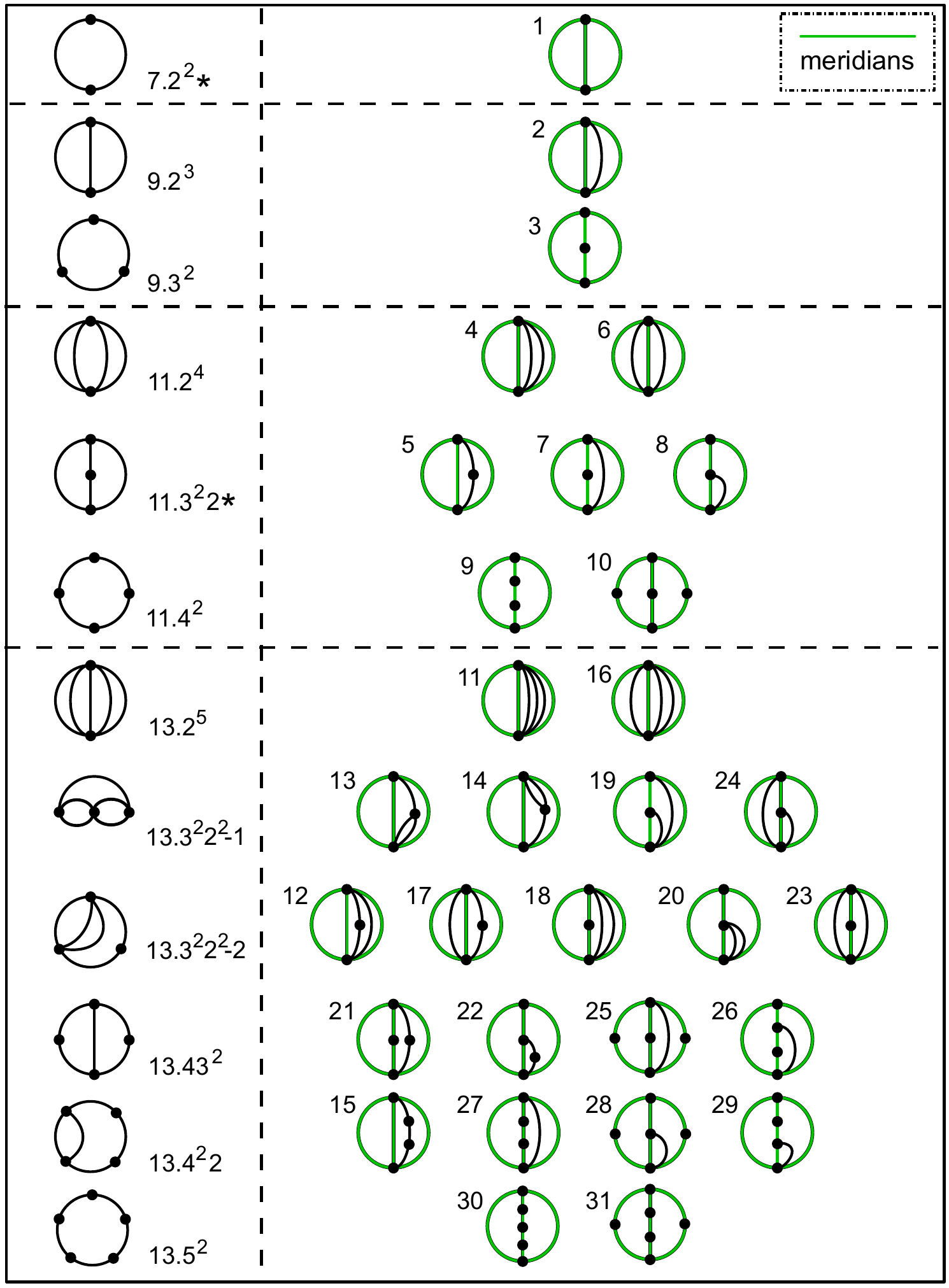}
\caption{\emph{
The 31 3-cell templates of 3-ball Sturm attractors with at most 13 equilibria, up to trivial equivalences.
See tables~\ref{tbl:6.5}, \ref{tbl:6.6} for case numbers 1--31, hemisphere notation, and Sturm permutations.
Cases are arranged in rows, on right, according to the twelve regular Thom-Smale $S^2$-complexes of fig.~\ref{fig:6.2}, listed left.
The two self-dual $S^2$-complexes are marked by $*$.
On the left, $S^2$ is the compactified plane.
On the right, with meridians in green, the right and left $\mathbf{EW}$ meridian have to be identified.
All omitted bipolar orientations result from the poles $\mathbf{N}$, top, versus, $\mathbf{S}$, bottom, and the hemisphere assignments $\mathbf{W}$, left, versus $\mathbf{E}$, right.
}}
\label{fig:6.3}
\end{figure}

\begin{figure}[p!]
\centering \includegraphics[width=\textwidth]{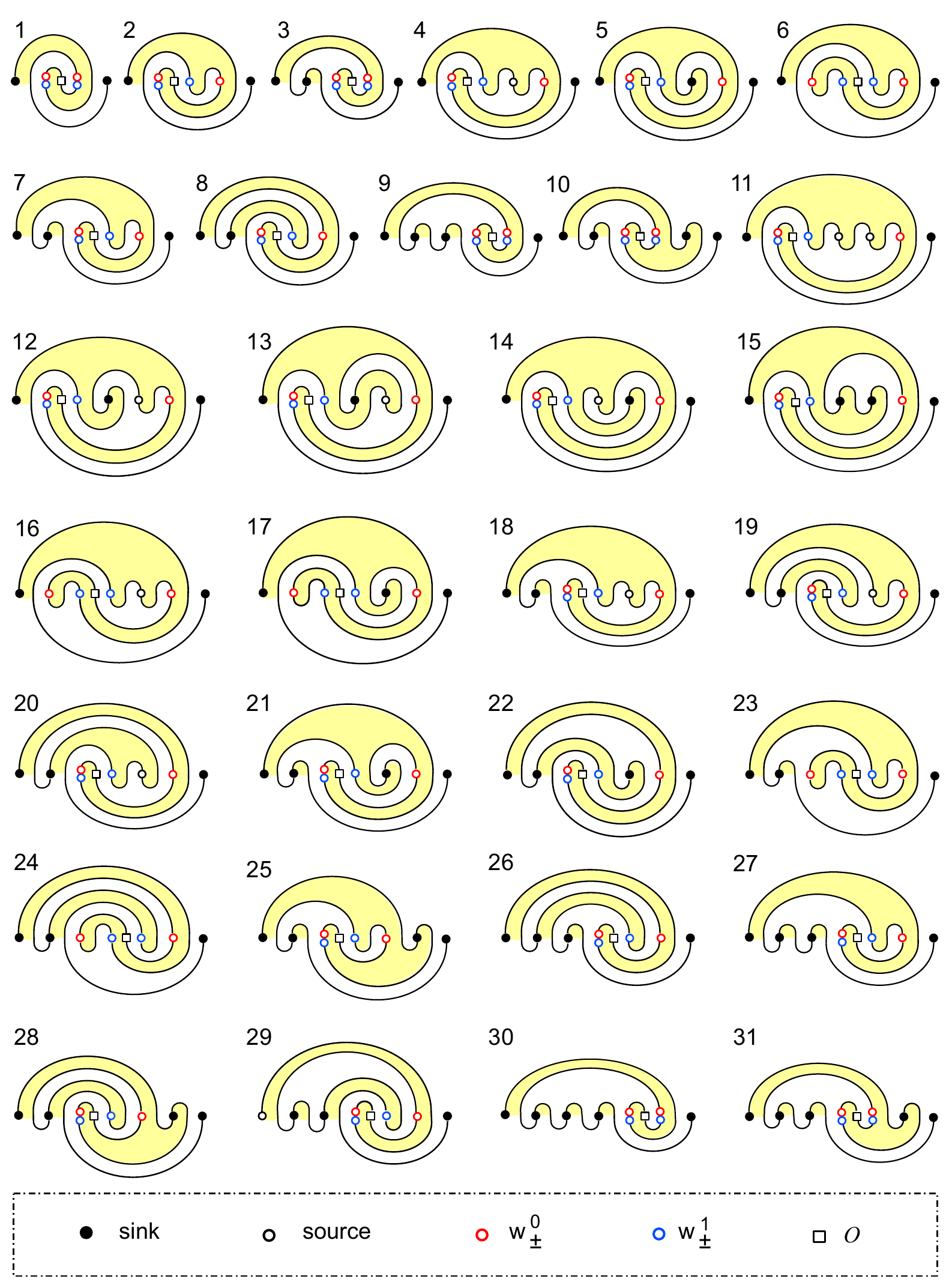}
\caption{\emph{
The 31 3-meander templates of 3-ball Sturm attractors with at most 13 equilibria, up to trivial equivalences.
Horizontal $h_1$ axis omitted.
See tables~\ref{tbl:6.5}, \ref{tbl:6.6}, and fig.~\ref{fig:6.3} for case numbers 1--31.
}}
\label{fig:6.4}
\end{figure}	

The first non-pitchforkable Sturm attractor has been constructed in \cite{ro91}.
It is the only self-dual planar Sturm attractor with at most 11 equilibria, other than the pitchforkable planar Chafee-Infante attractor $5.2$; see \cite[section~3.6]{firo10}.
We were not aware, so far, that the \emph{only} other non-pitchforkable Sturm attractor with (at most) 11 equilibria is the 3-ball given by case~8 in fig.~\ref{fig:6.3} and table~\ref{tbl:6.5}.
All six non-pitchforkable 3-ball Sturm attractors with 13 equilibria arise from case~8 by a pitchfork bifurcation.

The isotropy element $\rho$ characterizes Sturm involutions $\sigma = \sigma^{-1}$; see table~\ref{tbl:3.1}.
We encounter 13 such cases in 3-ball Sturm attractors with at most 13 equilibria.
In \cite{firowo12} we have characterized the Sturm permutations of Hamiltonian (pendulum) type nonlinearities $f= f(u)$ which only depend on $u$.
A necessary, but not sufficient, condition was $\sigma = \sigma^{-1}$ to consist of 2-cycles, only.
Due to absence of inversion isotropy $\rho$, none of the Sturm 3-balls with up to 13 equilibria is Hamiltonian, i.e. none is realizable by a nonlinearity $f=f(u)$ of pendulum type -- except the well-known Chafee-Infante attractor, case 1, \cite{chin74}.
This may be one reason why, to our knowledge, none of the cases 4--31 has appeared in the literature so far.
See \cite{fi94} for cases 2, 3.


\section{The Sturm Platonic solids}\label{sec7}


\begin{figure}[b!]
\centering \includegraphics[width=0.75\textwidth]{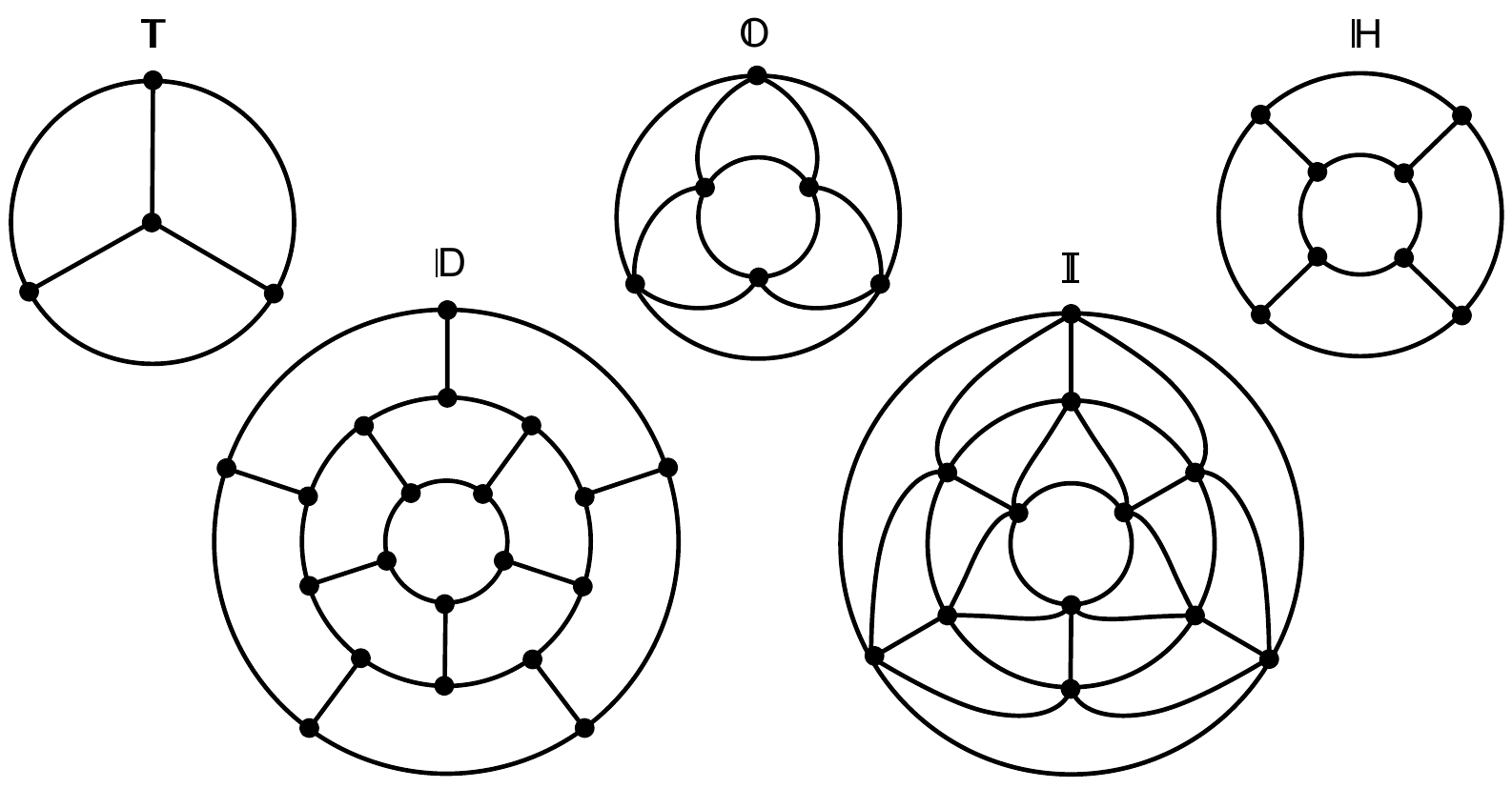}
\caption{\emph{
The five Platonic solids $c_\mathcal{O}\,$: tetrahedron $(\mathbb{T})$, octahedron $(\mathbb{O})$, cube or hexahedron $(\mathbb{H})$, dodecahedron $(\mathbb{D})$, and icosahedron $(\mathbb{I})$.
In each case we depict the planar Schlegel graph of the 1-skeleton $\mathcal{C}^1$ for the regular cell complex $\mathcal{C}^2$ of the boundary sphere $S^2= \partial c_\mathcal{O}$.
Again, we consider the exterior as another face in the one-point compactification of the plane.
}}
\label{fig:7.1.1}
\end{figure}

In this section we present Thom-Smale complexes of Sturm 3-cell templates and 3-meander templates for the five Sturm Platonic solids.
We outline their basic properties and graphical representations in section~\ref{subsec7.1}.
In \ref{subsec7.2} we present the two tetrahedra.
All five octahedra are obtained in \ref{subsec7.3} and all seven cubes in \ref{subsec7.4}.
These lists are complete, up to the trivial equivalences of section~\ref{sec3}.
We conclude with some remarks and examples on dodecahedra and icosahedra, in \ref{subsec7.5}.
We did not find any Platonic solid, in this investigation, which would be realizable by a Hamiltonian (pendulum) type nonlinearity $f=f(u)$ which only depends on $u$.

\subsection{The five Platonic solids}\label{subsec7.1}	

The five Platonic solids arise as the convex 3-dimensional polyhedra with regular $n$-gons as boundaries and identical degree $d$ at each vertex.
In other words, they are the convex hulls of non-planar orbits under discrete subgroups of the orthogonal group $SO(3)$, via the standard action on $\mathbb{R}^3$.
We study these examples because it is far from obvious how to accommodate bipolarity and the hemisphere structure of Sturm 3-cell templates in these highly symmetric objects, and how to obtain them as Sturmian Thom-Smale complexes.

Let $c_i$ count the cells of dimension $i=0,1,2$ of the 2-sphere boundary $S^2=\partial c_\mathcal{O}$ of the single 3-cell $c_\mathcal{O}$.
We then obtain table~\ref{tbl:7.1.1} and fig.~\ref{fig:7.1.1} as specific lists, from the convexity condition $d(1-2/n) <2$ and the Euler characteristic $c_0-c_1+c_2=2$.
The duals are defined by standard graph duality on $S^2$.
We also indicate the edge diameter $\vartheta$, on $S^2$, as an upper bound for the edge distance $\delta$ of the Sturm poles.

\begin{table}[t!]
\centering \includegraphics[width=\textwidth]{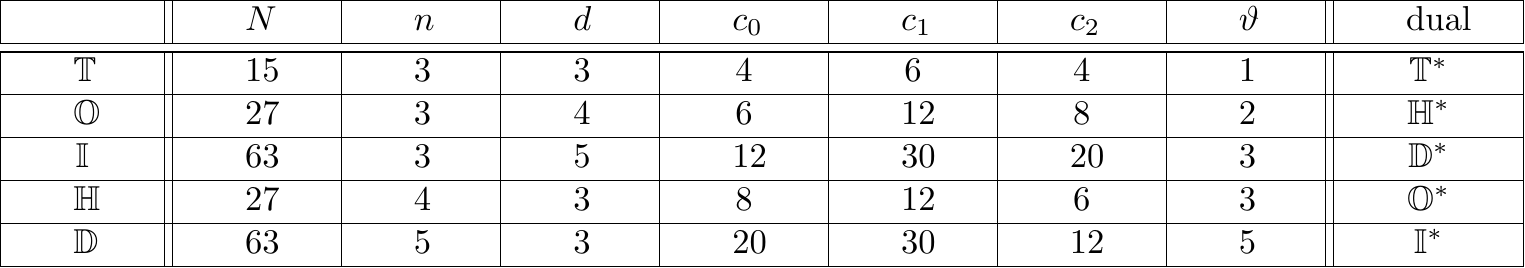}
\caption{\emph{                
The five convex Platonic solids with $N$ cells, characterized by regular $n$-gon faces and vertex degree $d$.
The columns $c_i$ count $i$-cells, and $\vartheta$ indicates the edge diameter, i.e. the maximal edge distance, on $S^2$, of vertices.
Standard $S^2$ duality is indicated in the last column.
}}
\label{tbl:7.1.1}
\end{table}

\subsection{The two Sturm tetrahedra}\label{subsec7.2}

The self-dual tetrahedron $\mathbb{T} = \mathbb{T}^*$, alias the 3-simplex, consists of $c_2=4$ faces, $c_0=4$ (sink) vertices, all of degree $d=3$, and of $c_1=6$ (saddle) edges. 
Since $n=3$, each face is a 3-gon.
Without loss of generality, we have to discuss the pole distance 
	\begin{equation}
	\delta = \vartheta=1\,, \quad \text{with} \quad 1 \leq \eta \leq c_2/2=2
	\label{eq:7.2.1}
	\end{equation}
face vertices of the Western dual core $ \mathbf{W}^*$.
Indeed, the poles have distance $\delta =1$, as any two vertices do, and we may choose $ \mathbf{W}^*$ as the smaller dual hemisphere.

See fig.~\ref{fig:7.2.1} for the unique single-face lift	$\eta=1$.
The Western face $\mathbf{W}$ is the compactified exterior of the Schlegel diagram, and the meridian circle is the boundary.
The trivial equivalence $\rho$ fixes 3d orientation.
The bipolar orientation is determined uniquely; see in particular definition~\ref{def:1.1}(iii) for the hemisphere $\mathbf{E}$.
This determines the SZS-pair $(h_0,h_1)$, and the Sturm meander permutation $\sigma = h_0^{-1} \circ h_1$, as illustrated.

\begin{figure}[t!]
\centering \includegraphics[width=\textwidth]{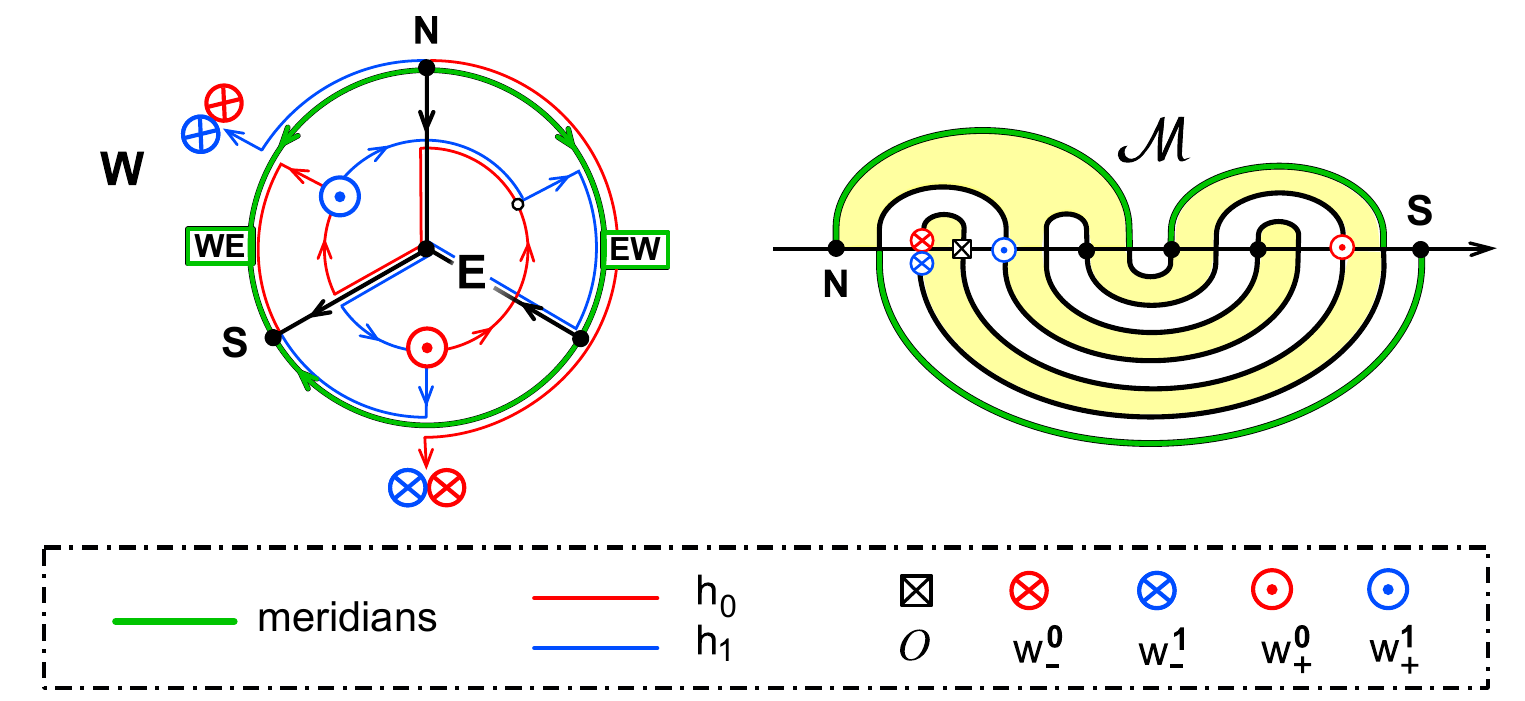}
\caption{\emph{
Left: the unique Sturm tetrahedron $\mathbb{T}.1$ with a single Western face (exterior).
The bipolar orientation on $S^2= \partial \mathbb{T}$ is uniquely determined by the pole location and the hemisphere decomposition.
Right: the Sturm meander $\mathcal{M}$ determined from the SZS-pair $(h_0,h_1)$ on the left.
}}
\label{fig:7.2.1}
\end{figure}

\begin{figure}[t!]
\centering \includegraphics[width=\textwidth]{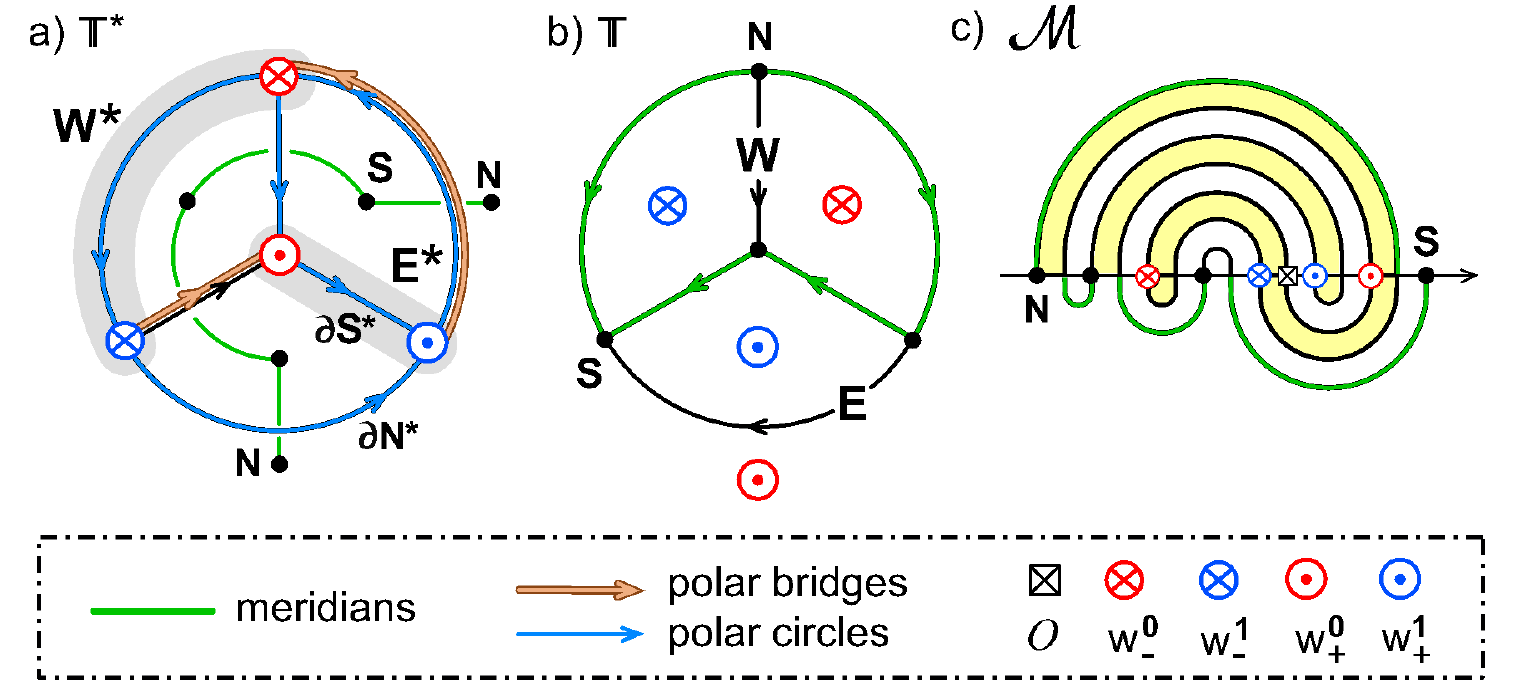}
\caption{\emph{
The unique Sturm tetrahedron $\mathbb{T}.2$ with two Western faces, (b).
The bipolar orientation on $S^2=\partial \mathbb{T}$ is uniquely determined by the pole location and the hemisphere decomposition.
Left, (a): the dual tetrahedron $\mathbb{T}^*$ with the one-dimensional dual core attractors $\mathbf{W}^*,\ \mathbf{E}^*$ (both shaded gray), dual poles $w_\pm^\iota$, polar circles $\partial \mathbf{N}^*,\ \partial \mathbf{S}^*$ (blue) and the predual meridian circle (green).
All orientations follow from lemma~\ref{lem:5.1}.
Right, (c): the Sturm meander $\mathcal{M}$ and Sturm permutation $\sigma = h_0^{-1} \circ h_1$ resulting from the (omitted) SZS-pair $(h_0,\ h_1)$ in the 3-cell template (b).
Note the 1, 2, 4 nested lower, and 7 nested upper arcs.
}}
\label{fig:7.2.2}
\end{figure}

The case of $\eta =2$ Western faces, i.e. of $\eta=2$ sinks in the dual core $\mathbf{W}^*$, leads to the one-dimensional attractor $\mathbf{W}^*$ with a single directed edge.
Indeed dual $n=2$-gons cannot be accommodated in $\mathbb{T}^*=\mathbb{T}$.
See sections~\ref{sec5}, \ref{subsec6.1} and figs.~\ref{fig:5.1}, \ref{fig:6.1}.
To derive the unique Sturm permutation $\sigma$, up to trivial equivalence, we start from the single directed edge $w_-^0w_-^1$ which defines $\mathbf{W}^* \subseteq \mathbb{T}^*$.
The dual pole face $\mathbf{N}^*$ must be edge adjacent to the right of the directed edge $w_-^0w_-^1$, i.e. exterior to the Schlegel triangle in fig.~\ref{fig:7.2.2}.
This defines the polar circle $\partial \mathbf{N}^*$ to be that boundary triangle, with left rotating orientation.
The dual edge $\mathbf{W}^*$ is surrounded by the meridian circle.
This only leaves one other edge $w_+^0w_+^1$ for $\mathbf{E}^*$.
The trivial equivalence $\rho$ is able to reverse the 3d orientation of $T^*$, and hence the direction of that edge, as well as the location of $w_+^\iota$ on that edge.
We choose the orientation of fig.~\ref{fig:7.2.2}(a) and obtain the polar face $\mathbf{S}^*$ to the left of $w_+^0w_+^1$,  with left oriented polar circle $\partial \mathbf{S}^*$.
The two polar circles overlap along the bridge from $w_+^1$ to $w_-^0$.
This settles all orientations (a) in the 1-skeleton $\mathcal{C}^{1,*}$, and hence the orientations (b) in $\mathcal{C}^1$.
The omitted SZS-pair $(h_0,h_1)$ of (b) then defines the Sturm permutation and meander $\mathcal{M}$ of (c).

\begin{table}[t!]
\centering \includegraphics[width=0.85\textwidth]{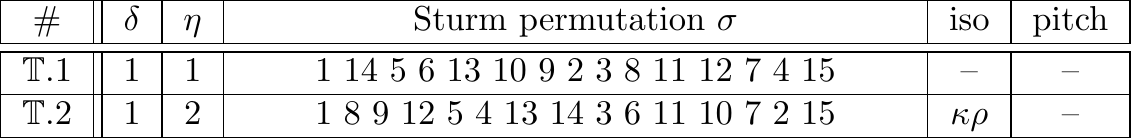}
\caption{\emph{
The two Sturm tetrahedra $\mathbb{T}$.
Pole distance $\delta=1$.
The number $\eta$ of Western faces is 1 or 2, with unique resulting Sturm permutations in either case, up to trivial equivalences.
}}
\label{tbl:7.2.1}
\end{table}

In summary, we obtain the two tetrahedral Sturm permutations of table~\ref{tbl:7.2.1} classified by their number $\eta =1,2$ of Western faces.
None of the examples is pitchforkable.
Note the isotropy generator $\kappa\rho$, for $\eta=2$.
Due to the absence of inversion isotropy $\rho$, however, none of the examples is Hamiltonian, i.e. none is realizable by a nonlinearity $f=f(u)$ of pendulum type.

\subsection{The five Sturm octahedra}\label{subsec7.3}

The octahedron $\mathbb{O}$, with dual hexahedral cube $\mathbb{O}^*=\mathbb{H}$, consists of $c_2=8$ faces, $c_0=6$ vertices of vertex degree $d=4$, and $c_1=12$ edges. 
By $n=3$, all faces are 3-gons. 
The edge diameter is $\vartheta =2$ on $S^2$.
We have to discuss pole distances $\delta$ and Western duals $\mathbf{W}^*$ with $\eta$ faces such that
	\begin{equation}
	1\leq \delta \leq \vartheta=2 \quad \text{and} \quad
	1\leq \eta \leq c_2/2=4\,,
	\label{eq:7.3.1}
	\end{equation}
without loss of generality.
See table~\ref{tbl:7.3.1} below for a list of results.

\begin{figure}[p!]
\centering \includegraphics[width=\textwidth]{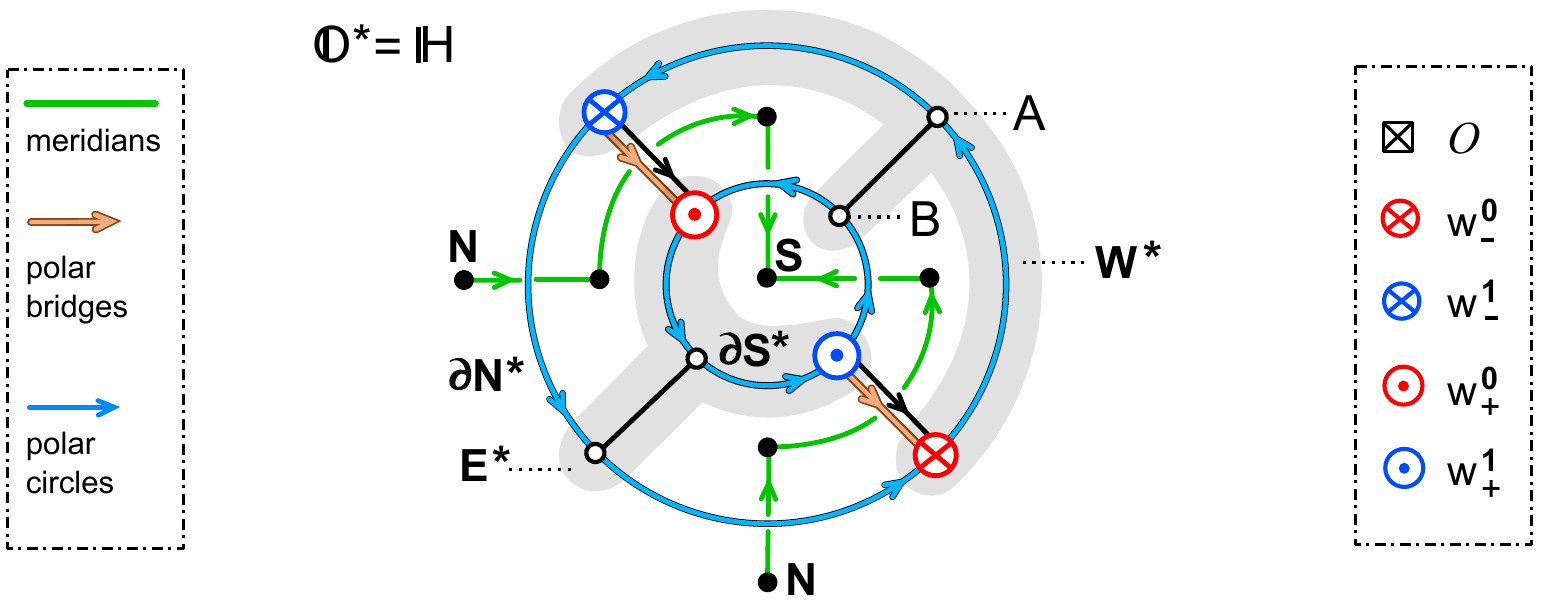}
\caption{\emph{
The impossibility of pole distance $\delta=2$ in the octahedron $\mathbb{O}$ with cube dual $\mathbb{O}^*= \mathbb{H}$.
Note the orientations of the disjoint polar circles $\partial \mathbf{N}^*,\ \partial \mathbf{S}^*$, and the pairs of directed polar bridges $e_*= w_\pm^1w_\mp^0$, dual to meridian edges $e$.
The remaining meridian edges are polar.
The meridian circle separates the (impossible) dual tri-star core $\mathbf{W}^*$ of $A,\ w_-^0,\ w_-^1,\ B$ from the tri-star core $\mathbf{E}^*$ (both shaded gray).
Polar bridges directed from $w_\pm^1$ to $w_\mp^0$ are indicated (orange).
}}
\label{fig:7.3.1}
\end{figure}

\begin{figure}[p!]
\centering \includegraphics[width=\textwidth]{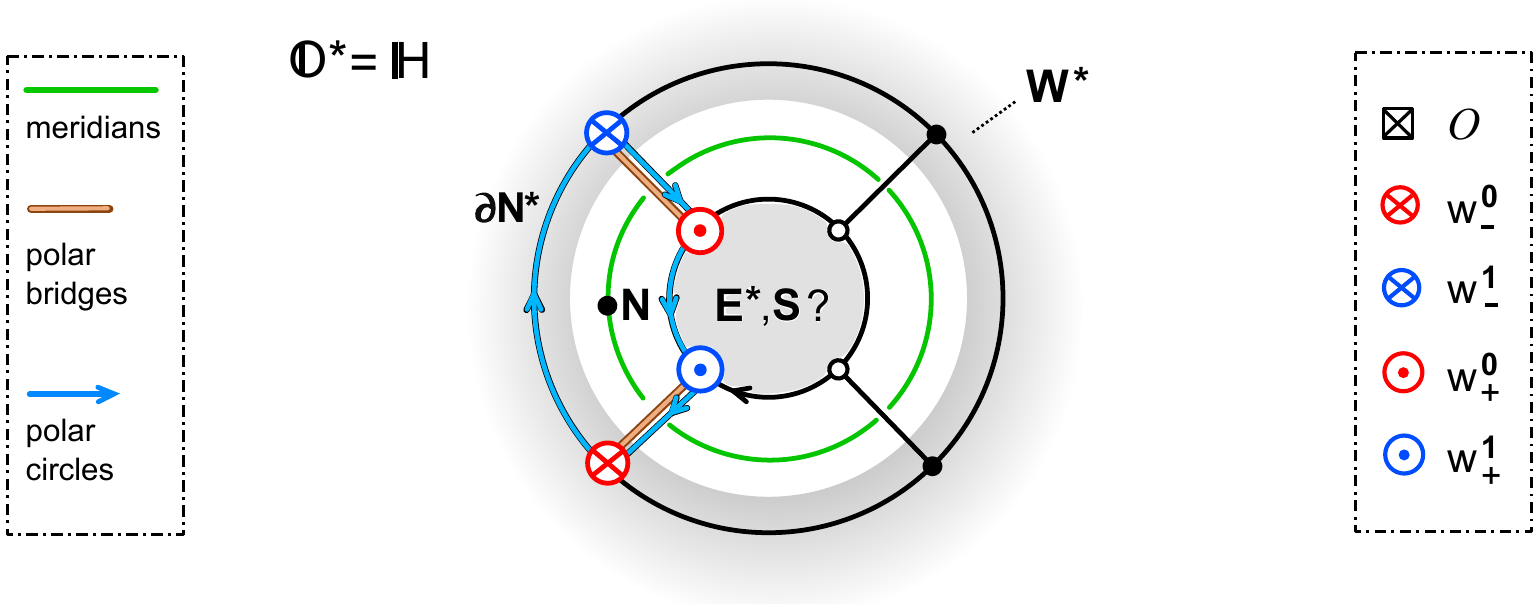}
\caption{\emph{
The impossibility of $\dim \mathbf{W}^*=2$ in the octahedron with cube dual $\mathbb{O}^*= \mathbb{H}$.
Note how the exterior square $\mathbf{W}^*$ (gray) forces the location of the meridian circle (green), with edge adjacent poles $w_-^\iota$ on the polar circle $\partial \mathbf{N}^* \,\cap \,\partial \mathbf{W}^*$.
These force $w_+^{1-\iota} \in \partial \mathbf{N}^*$ to be adjacent on the inner square $\partial \mathbf{S}^*$.
The resulting position of $\mathbf{S}$ in the central square is not on the meridian circle, and hence is impossible.
}}
\label{fig:7.3.2}
\end{figure}

In \cite{firo3d-1} we have already observed that pole distance $\delta=2$ cannot occur in the octahedron; see also \cite{firo14}.
For illustration we give another proof here, based on the dual cube $\mathbb{H}=\mathbb{O}^*$.
For $\delta =2$, the poles $\mathbf{N},\ \mathbf{S}$ are antipodes.
Hence the dual polar circles $\partial\mathbf{N}^*,\ \partial\mathbf{S}^*$ in $\mathbb{O}^*$ are disjoint; see fig.~\ref{fig:7.3.1}.
At least two of the remaining four non-polar edges of the dual $\mathbb{H}$ must connect the dual poles $w_\pm^1$ to $w_\mp^0$ as directed polar bridges, in pairs. See lemma \ref{lem:5.1}(iv).
The dual cores $\mathbf{W}^*,\ \mathbf{E}^*$ cannot both be singletons, since $c_2^*=c_0=6$.
As in corollary \ref{cor:5.2}(v), $w_\pm^1$ cannot be followed by $w_\pm^0$ on either polar circle, after a single directed edge.
Therefore the two polar bridges must be diagonally opposite.
This determines the meridians and the hemisphere attractors $\mathbf{W}^*,\ \mathbf{E}^*$, as in fig.~\ref{fig:7.3.1}.
However, $\mathbf{W}^*$ then consists of a tri-star, with edge spikes to $w_-^0,\ w_-^1$, and $B$ all emanating from the same dual vertex $A$.
Similarly, $\mathbf{E}^*$ is also a tri-star.
This contradicts bipolarity of $\mathbf{W}^*,\ \mathbf{E}^*$.
Therefore we can only encounter pole distance $\delta=1$ in the octahedron $\mathbb{O}$.

We show next that $\mathbf{W}^*$, with $\eta \leq 4$ vertices, must be one-dimensional.
Otherwise, $\mathbf{W}^* \subseteq \mathbb{O}^* = \mathbb{H}$ is a single closed 4-gon face of the cube $\mathbb{H}$.
In fig.~\ref{fig:7.3.2} we draw $\mathbf{W}^*$ as the exterior face.
The polar circle $\partial \mathbf{N}^*$ must be centered around some polar vertex $\mathbf{N}$ on the meridian.
Since $\partial \mathbf{N}^*$ contains the poles $w_-^\iota$ of $\mathbf{W}^*$, the path from $w_-^0$ to $w_-^1$ in the square boundary must therefore consist of a single directed edge.
See fig.~\ref{fig:7.3.2} again for the resulting orientation.
The poles $w_+^\iota$ of $\mathbf{E}^*$ must lie on the remaining centered 4-gon candidate $\mathbf{E}^*$ which is separated from $\mathbf{W}^*$ by the meridian circle.
The polar bridges locate $w_-^{1-\iota}$ opposite $w_+^\iota$, across the meridians, by lemma~\ref{lem:5.1}.
This forces $\mathbf{S}$ to be the barycenter of the inner square $\mathbf{S}^* \subseteq \mathbf{E}^*$.
Since $\mathbf{S}$ must also lie on the meridian circle, this is a contradiction.
This proves $\dim \mathbf{W}^*=1$ is a path of edges in the polar circle $\partial \mathbf{N}^*$.
In particular $\eta \leq 3$ with edge distance $\eta -1\leq 2$ from $w_-^0$ to $w_-^1$ on $\partial \mathbf{N}^*$; see \eqref{eq:5.6b}.

\begin{figure}[t!]
\centering \includegraphics[width=\textwidth]{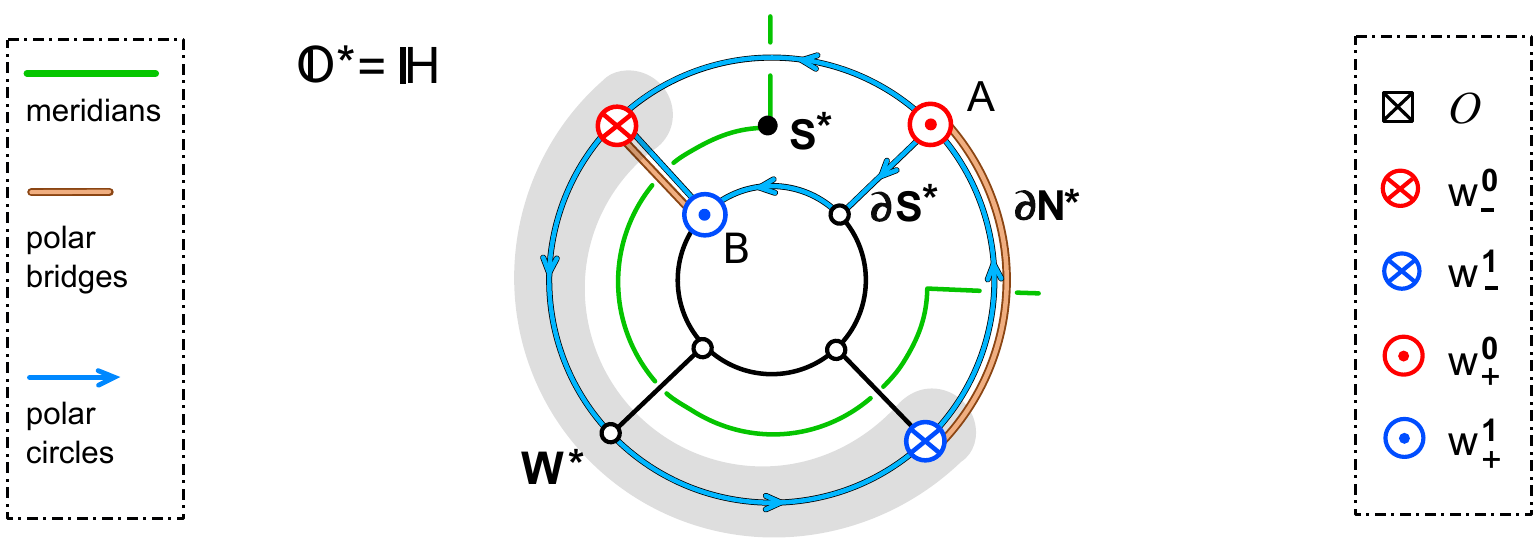}
\caption{\emph{
An orientation conflict arising from the (gray) Western core $\mathbf{W}^*$ with $\eta =3$ vertices.
Note how the location of the $\mathbf{E}^*$-poles $w_+^\iota$ forces the face $\mathbf{S}^*$ to possess an inconsistent orientation of its polar circle, $\partial \mathbf{S}^*$ from $A = w_+^0$ to $w_+^1$.
}}
\label{fig:7.3.3}
\end{figure}

Now let $\partial\mathbf{N}^*$ be the outer square again, with exterior face $\mathbf{N}^*$.
Suppose $\eta =3$.
Then the poles $w_-^\iota$ of $\mathbf{W}^*$ are diagonally opposite on the outer square; see fig.~\ref{fig:7.3.3} for $\mathbf{W}^*$ and the surrounding meridian circle.
Since $\mathbf{E}^*$ is spiked at $A$, but bipolar, one of its poles $w_+^{1-\iota}$ must coincide with $A$.
The other pole $w_+^\iota$ must be on the inner 4-gon, diagonally opposite to $w_-^\iota$ across the meridian (green) via a polar bridge (orange). 
Up to a reflection, we may consider $B = w_+^\iota$ located opposite $w_-^0$, as in fig.~\ref{fig:7.3.3}.
If $B = w_+^0$ is red, then it does not posses any mandatory single-edge bridge with $w_-^1$ (blue).
Therefore $B=w_+^1$ and $A=w_+^0$. 
Because one of the directed paths from $A$ to $B$ in the boundary $\partial \mathbf{E}^*$ must be a segment of the polar circle $\partial \mathbf{S}^*$, this fixes the location of $\mathbf{S}^*$ on the meridian (green), as indicated. 
However, the direction of the segment $\partial \mathbf{S}^* \cap \partial \mathbf{E}^* $ from $A$ to $B$ induces a clockwise orientation of the South polar circle $\partial \mathbf{S}^*$.
This contradicts the counter-clockwise orientation of $\partial \mathbf{S}^*$ required in lemma \ref{lem:5.1},
and eliminates the option $\eta=3$.

\begin{figure}[t!]
\centering \includegraphics[width=\textwidth]{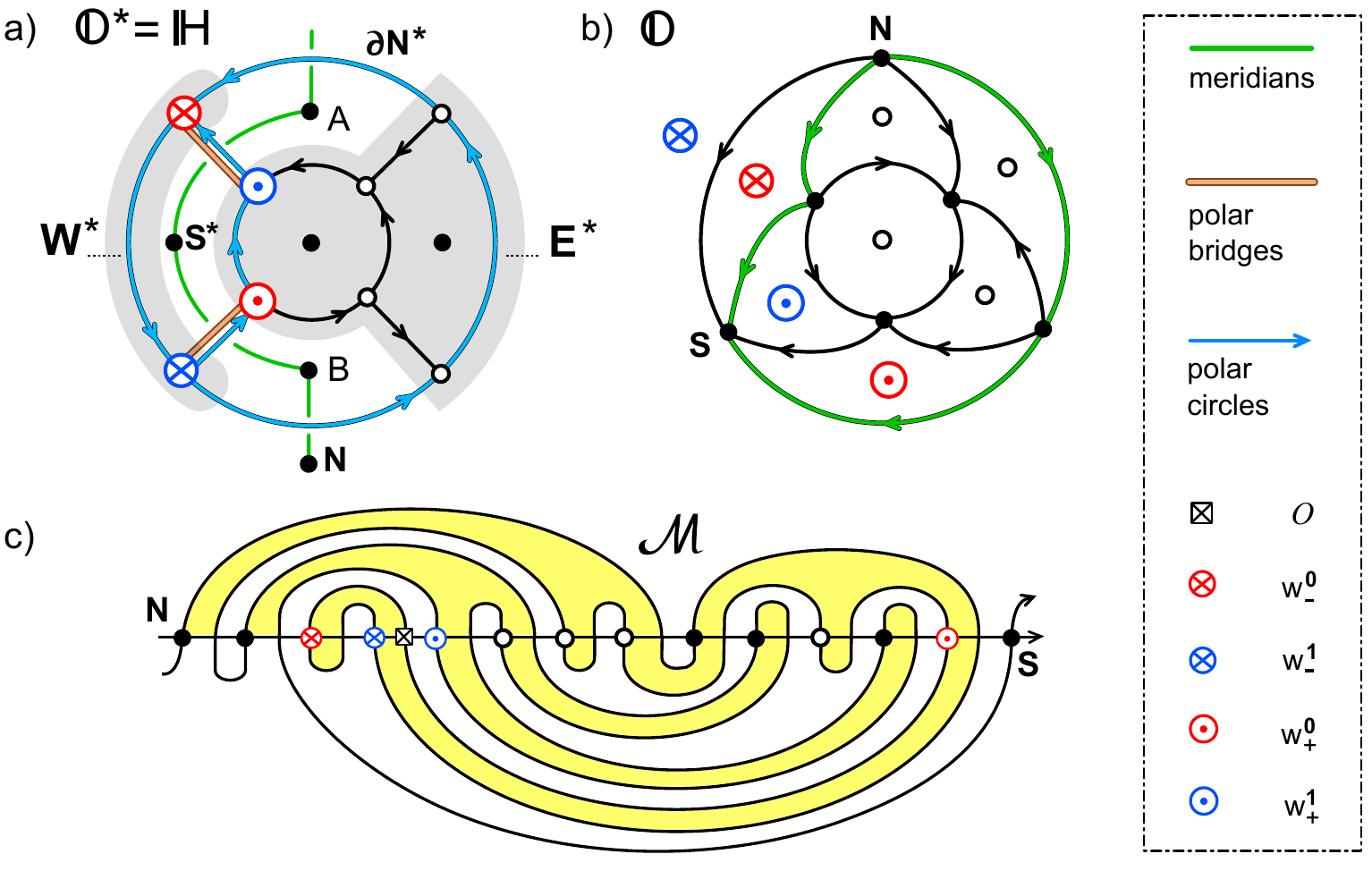}
\caption{\emph{
A viable dual core $\mathbf{W}^*$ (gray) with $\eta=2$ vertices, (a).
The locations $A,\ B$ are not viable for $\mathbf{S}$ because $\partial \mathbf{S}^*$ cannot accommodate $w_+^0$ and $w_+^1$ at edge distance 1 across the meridian from $w_-^1$ and $w_-^0$, respectively.
We draw the only viable location for $\mathbf{S}$.
The bipolar orientation of (gray) $\mathbf{E}^*$, from $w_+^0$ to $w_+^1$ determines all other edge orientations uniquely.
See (b) for the resulting bipolar octahedron complex $\mathbb{O}.2$.
See table~\ref{tbl:7.2.1} for the Sturm permutation $\sigma$.
}}
\label{fig:7.3.4}
\end{figure}

\begin{figure}[p!]
\centering \includegraphics[width=\textwidth]{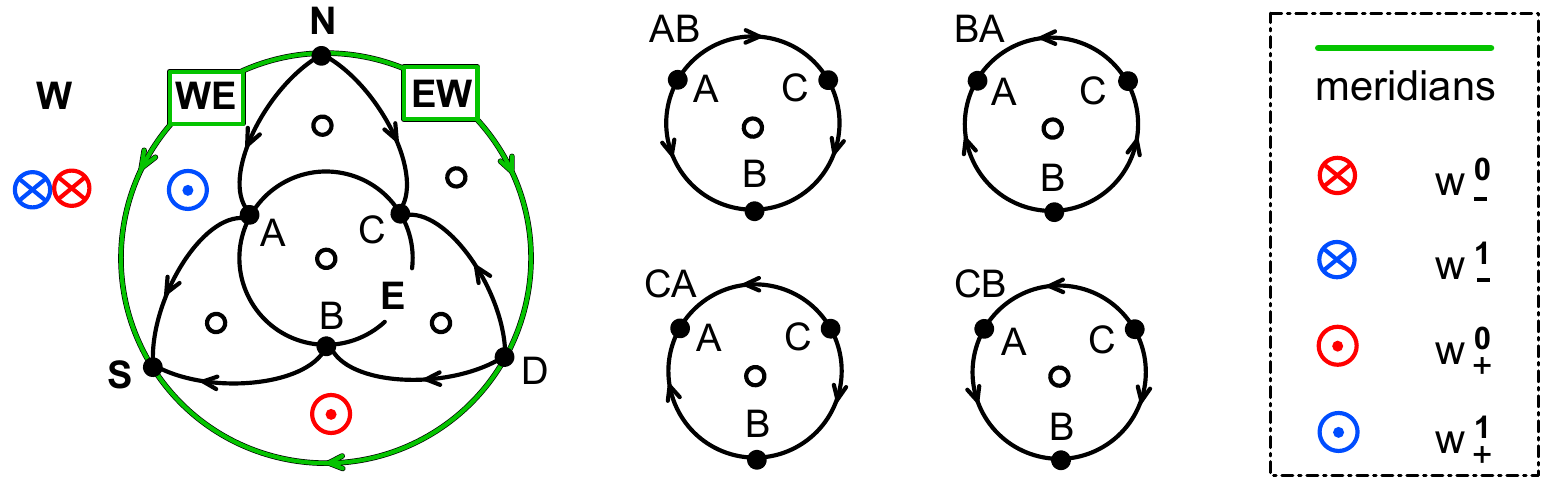}
\caption{\emph{
The four possible bipolar orientations of the Sturm octahedron with a single face Western hemisphere $\mathbf{W}$ (exterior).
The orientations only differ on the acyclic central triangle $ABC$ inside the Eastern hemisphere $\mathbf{E}$.
The four possibilities on the right arise by the selection of a maximal and a minimal vertex among $A,\ B,\ C$.
Bipolarity prevents $C$ from being minimal.
The case $AB$, for example, chooses $A$ as maximal and $B$ as minimal.
}}
\label{fig:7.3.5}
\end{figure}

\begin{figure}[p!]
\centering \includegraphics[width=\textwidth]{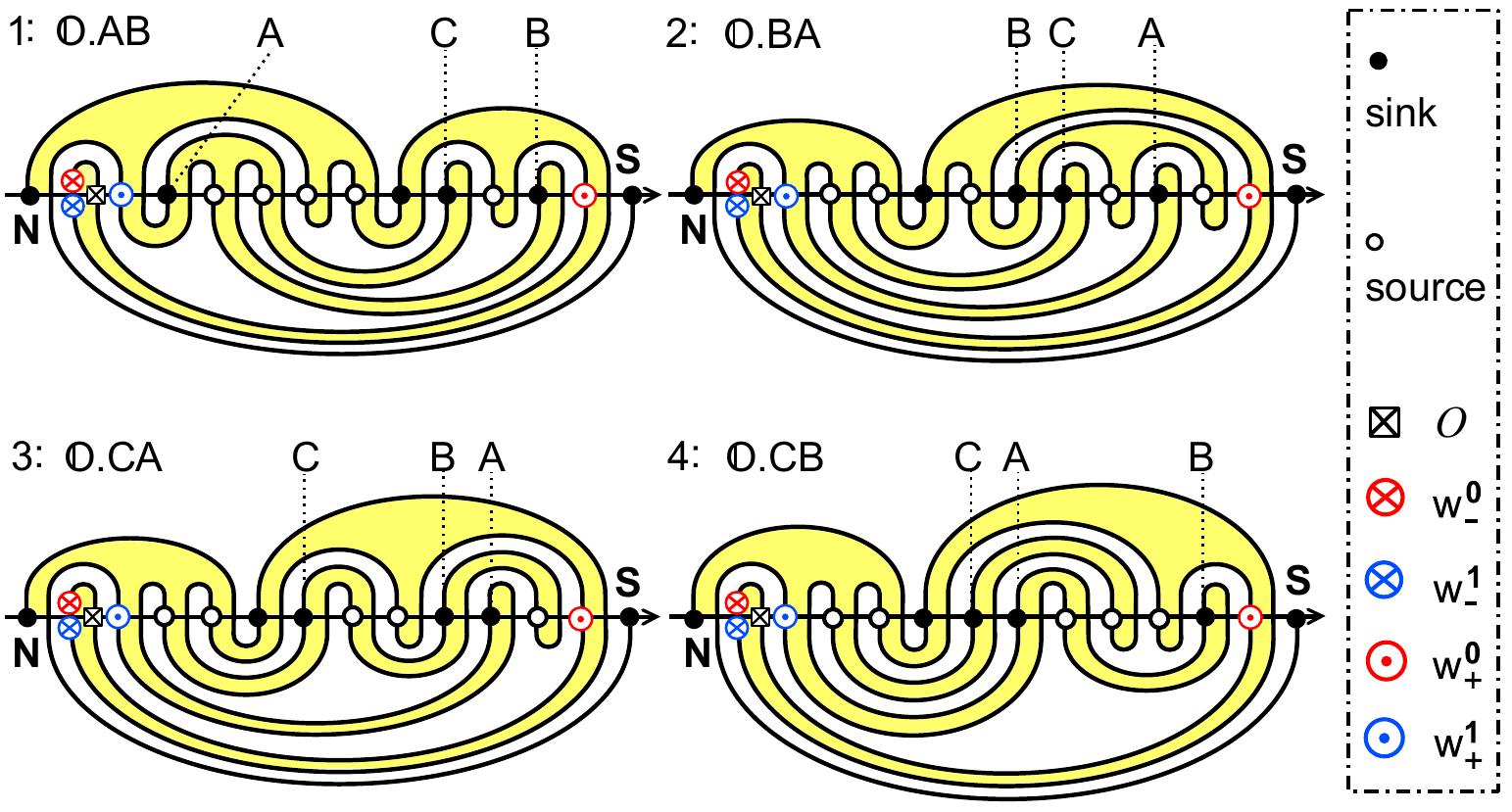}
\caption{\emph{
The four 3-meander templates of single-face lift octahedra, $\eta=1$.
Note the identical locations of the four core poles $w_\pm^\iota$, and the different location configurations of the central triangle vertices $A,\ B,\ C$.
See fig.~\ref{fig:7.3.5} for bipolar orientations, and table~\ref{tbl:7.3.1} for Sturm permutations and case labels.
}}
\label{fig:7.3.6}
\end{figure}

Consider the case of $\eta=2$ sink vertices on the single-edge  dual attractor $\mathbf{W}^*$, next.
The meridian (green) surrounding $\mathbf{W}^*$ (gray) offers three locations for $\mathbf{S}$.
See fig.~\ref{fig:7.3.4}.
The top and bottom choices $A,\ B$ cannot accommodate both poles $w_+^\iota$ of $\mathbf{E}^*$ on the polar circle $\partial\mathbf{S}^*$.
Indeed, at least one of $w_+^\iota \in \mathbf{E}^*$ would not be connected to its counterpart $w_-^{1-\iota}$ by a single-edge polar bridge (orange).
The only remaining option is worked out in fig.~\ref{fig:7.3.4}(a)--(c), from the dual $\mathbb{O}^*$ to the octahedron $\mathbb{O}$.
Since the bipolar orientations are determined uniquely, we obtain the unique permutation for the Sturm octahedron $\mathbb{O}.2$ with $\eta =2$ faces in the Western hemisphere.
See case 5 in table~\ref{tbl:7.3.1}.

\begin{table}[b!]
\centering \includegraphics[width=0.85\textwidth]{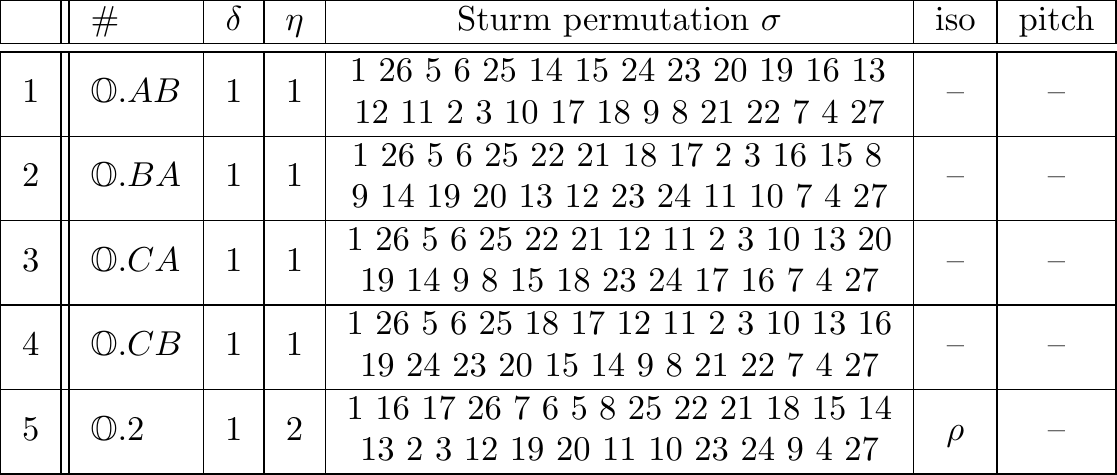}
\caption{\emph{
The five Sturm octahedra $\mathbb{O}$.
Pole distance $\delta=1$.
The number $\eta$ of Western faces is 1 or 2.
The four cases of single face lifts, $\eta=1$, arise from the local orientations of the triangle $ABC$ in fig.~\ref{fig:7.3.5}.
The involutive case $\eta =2$ of two Western faces is the only case with isotropy.
Although the requirements of corollary \ref{cor:3.2} are satisfied, flip-isotropy $\kappa$ does not occur in the octahedron $\mathbb{O}$.
Still, it is neither pitchforkable nor realizable by a pendulum type nonlinearity $f=f(u)$.
It defines a unique Sturm permutation, up to trivial equivalence.
}}
\label{tbl:7.3.1}
\end{table}

It only remains to discuss the case $\eta = 1$ where the Western hemisphere is a single face and $\mathbb{O}$ is the single face lift of an Eastern planar octahedron.
See fig.~\ref{fig:7.3.5} for the four possible orientations which arise.
Indeed non-polar edges in $\mathbf{E}$ have to emanate from meridian sink $D \in \mathbf{EW}$, by definition~\ref{def:1.1}(iii).
This leaves the central triangle $ABC$ up for bipolar orientation.
The vertex $C \neq \mathbf{S}$ cannot be chosen as a local minimum, among $\lbrace A,\ B,\ C \rbrace$, due to bipolarity of $\mathbb{O}$.
Consider $B$ as a local minimum on $ABC$.
This leaves us with the cases $AB$ and $CB$, depending on the choice of $A$ or $C$ for a local maximum on $ABC$.
With $A$ as local minimum, we obtain the remaining cases $BA,\ CA$ with local maxima $B,\ C$, respectively.

See table~\ref{tbl:7.3.1} for the resulting four cases 1--4 of bipolar face lifted octahedra, and fig.~\ref{fig:7.3.6} for their meanders.
The permutations $\sigma = h_0^{-1}\circ h_1$ are defined via the SZS-pair $(h_0,h_1)$ of each orientation.
Case~5 is the unique case with $\eta =2$ Western faces; see \eqref{eq:7.3.1}.
Note the pole distance $\delta =1$ in all cases, because the octahedron with antipodal poles cannot be realized in the Sturm class.
The only isotropy which arises is $\rho$, i.e. $\sigma^{-1}=\sigma$, in case~5.
See table~\ref{tbl:3.1}.
None of the cases~1--5 is pitchforkable or realizable by a pendulum $f=f(u)$.

For lack of scientific understanding it is also possible to arrive at table~\ref{tbl:7.3.1} by brute force.
There are 70,944~Hamiltonian path candidates for $h^\iota$, between antipode vertices, and 62,552 between neighbors.
Sifting through pairs for Sturm permutations, the above five cases can be obtained.
Alas, what would we have understood?

\subsection{The seven Sturm cubes}\label{subsec7.4}

The hexahedral cube $\mathbb{H}$, with dual octahedron $\mathbb{H}^*= \mathbb{O}$, consists of $c_2=6$ faces with $c_0=8$ vertices of vertex degree $d=4$, with $c_1=12$ edges. The edge diameter is $\vartheta =3$.
Therefore we have to discuss pole distances $\delta$ and Western duals $\mathbf{W}^*$ with $\eta$ faces such that
	\begin{equation}
	1 \leq \delta\leq \vartheta = 3 \quad \text{and} \quad
	1\leq \eta \leq c_2/2=3\,.
	\label{eq:7.4.1}
	\end{equation}
See table~\ref{tbl:7.4.1} for a list of results.

\begin{figure}[p!]
\centering \includegraphics[width=\textwidth]{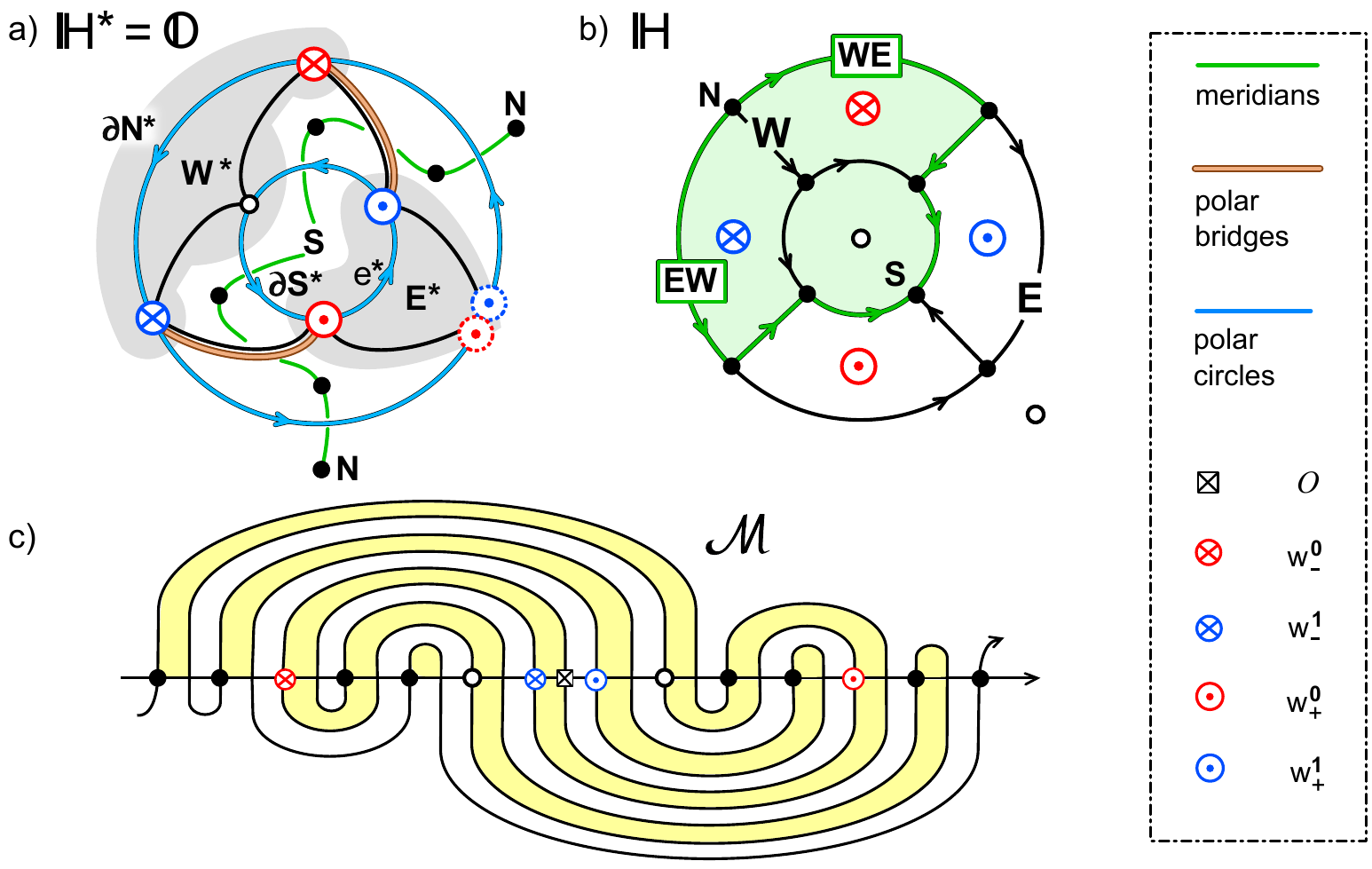}
\caption{\emph{
The unique cube $\mathbb{H}.3.3$ with pole distance $\delta =3$.
Note the resulting 3-gon, $\eta =3$, Western and Eastern cores $\mathbf{W}^*,\ \mathbf{E}^*$ (both gray).
(a) Disjoint polar circle triangles $\partial \mathbf{N}^*,\ \partial \mathbf{S}^*$, with poles $\otimes = w_-^\iota$ of $\mathbf{W}^*$ in $\partial \mathbf{N}^*$.
The bridge options for poles $\odot = w_+^\iota$ of $\mathbf{E}^*$ in $\partial \mathbf{S}^*$, across the meridians, are dotted or solid.
Only the solid option is compatible with the proper left orientation of $\partial \mathbf{S}^*$.
(b) The resulting cube 3-cell template $\mathbb{H}$ with uniquely determined bipolar orientation.
(c) The cube meander $\mathcal{M}$ generated by the SZS-pair $(h_0,h_1)$ of the cube 3-cell template (b).
Note the nested $3^2,\ 3,\ 1$ upper arcs, and $1,\ 3,\ 3^2$ lower arcs of $\mathcal{M}$, from left to right.
The meander exhibits full isotropy under all trivial equivalences $\langle \kappa, \rho \rangle$.
Flip-isotropy $\kappa$ is in compliance with the requirements of corollary \ref{cor:3.2}.
Still, it is neither pitchforkable nor realizable by a pendulum type nonlinearity $f=f(u)$.
See table~\ref{tbl:7.4.1}, case 7, for the Sturm permutation $\sigma$.
}}
\label{fig:7.4.1}
\end{figure}

\begin{figure}[p!]
\centering \includegraphics[width=\textwidth]{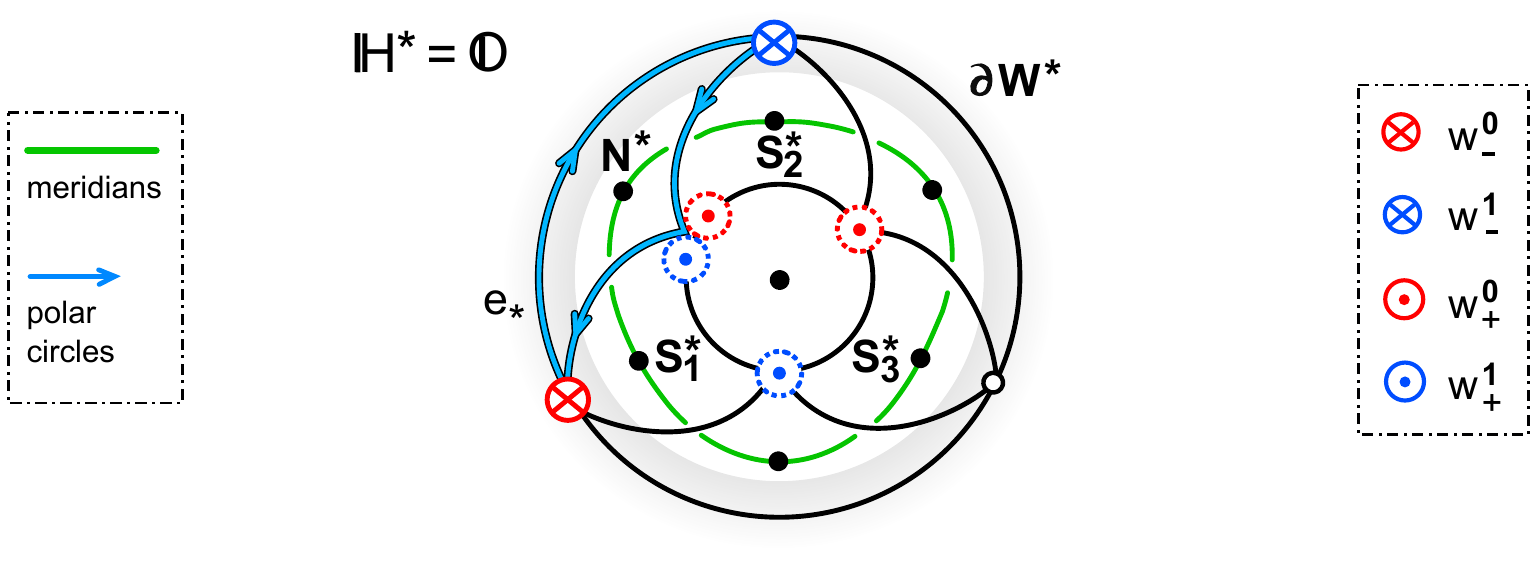}
\caption{\emph{
Dual octahedron with a 3-gon Western core $\mathbf{W}^*$ (exterior, gray) and surrounding meridian.
The direction of the dual edge $e_* = w_-^0w_-^1$ follows because $\mathbf{N}^*$, with barycenter $\mathbf{N}$ on the meridian, cannot coincide with the exterior face of $\mathbf{W}^*$.
The left orientation of the Southern polar circle $\partial \mathbf{S}^*$ contradicts the direction $w_+^0w_+^1$, for $\mathbf{S}=\mathbf{S}_1$ or $\mathbf{S}_2$ and all (dotted) candidates $w_+^{1-\iota}$ paired with $w_-^\iota$.
Therefore $\mathbf{S} = \mathbf{S}_3$ is the antipode of $\mathbf{N}$, as in fig.~\ref{fig:7.4.1}.
}}
\label{fig:7.4.2}
\end{figure}

Again we consider the case of maximal pole distance first:
$\delta =3$, with diagonally opposite poles $\mathbf{N},\ \mathbf{S}$.
In the dual octahedron $\mathbb{O}= \mathbb{H}^*$, this means that the polar circles $\partial \mathbf{N}^*,\ \partial \mathbf{S}^*$ are disjoint.
See fig.~\ref{fig:7.4.1}, where the polar circle $\partial \mathbf{N}^*$ is the outer bounding 3-gon, and $\partial \mathbf{S}^*$ is the disjoint central 3-gon.
Note how both polar circles are left oriented.
We place $w_-^\iota$ on $ \partial \mathbf{N}^*$ as indicated, without loss of generality.
By polar bridge adjacency of $w_-^\iota \in \partial \mathbf{N}^*$ to $w_+^{1-\iota} \in \partial \mathbf{S}^*$, we are restricted to the dotted and solid options for $w_+^{1-\iota}$, in fig.~\ref{fig:7.4.1}(a).
The requirement $ w_+^{1-\iota}\in \partial \mathbf{S}^*$ eliminates the dotted option, and selects the solid option, uniquely.
The meridian edge separations required by lemma~\ref{lem:5.1} define the complete meridian circle, which surrounds, separates, and hence defines, the dual cores $\mathbf{W}^*$ and $\mathbf{E}^*$.
Note how the Western core $\mathbf{W}^*$ is a Sturm 3-gon, $\eta =3$, as is the Eastern core $\mathbf{E}^*$.
This determines the orientations of $\mathbb{H}^*=\mathbb{O}$ and of the 3-cell template for $\mathbb{O}$ uniquely, as in fig.~\ref{fig:7.4.1}(b).
The unique meander $\mathcal{M}$ in (c) follows, as usual, from the SZS-pair $(h_0,h_1)$ of (b) via the Sturm permutation $\sigma = h_0^{-1}\circ h_1$.
See case 7 in table~\ref{tbl:7.4.1}.

\begin{figure}[]
\centering \includegraphics[width=\textwidth]{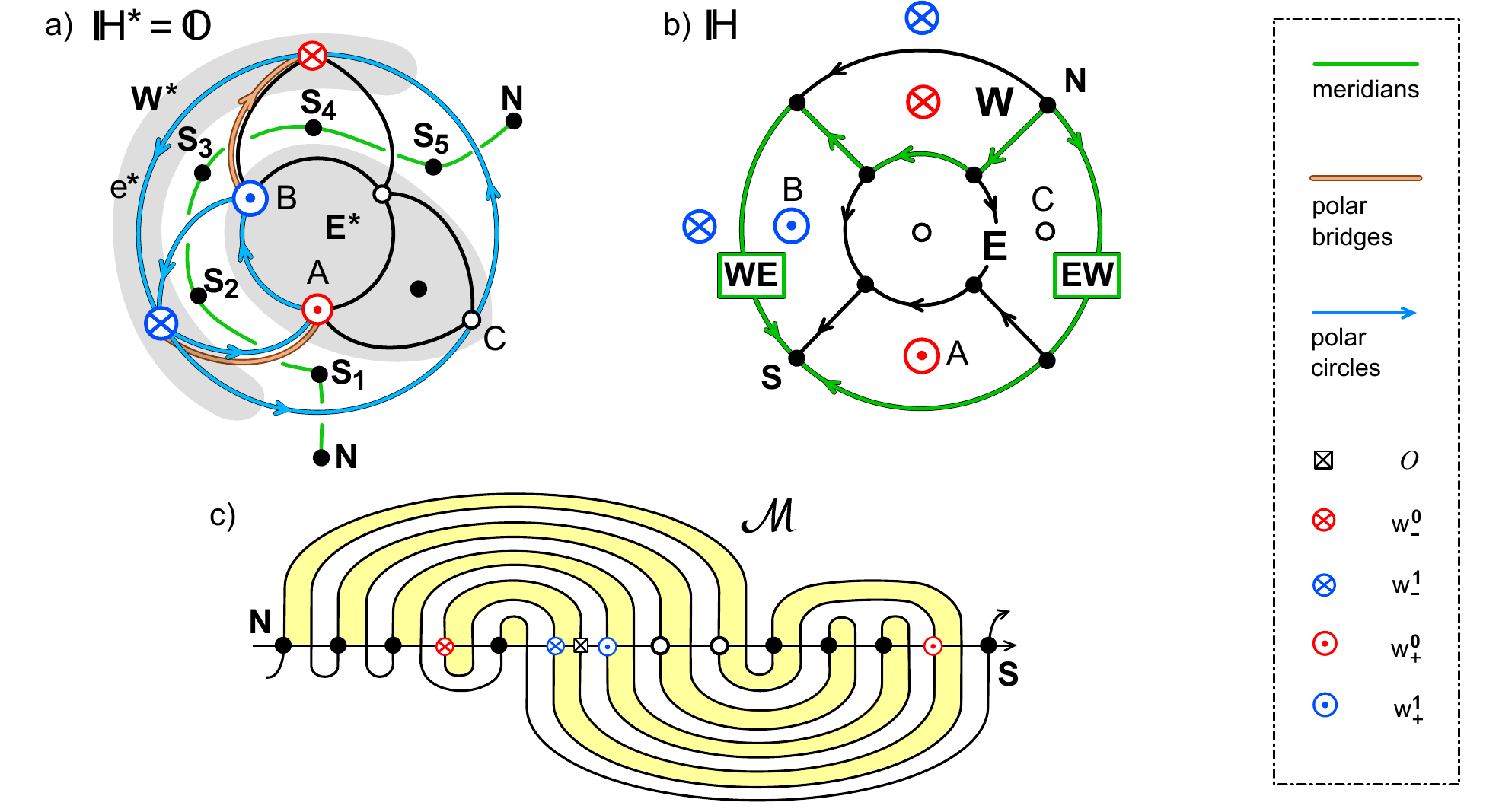}
\caption{\emph{
The unique cube $\mathbb{H}.2.2$ with $\eta=2$ Western faces.
Note the unique edge $e_* = w_-^0 w_-^1$ in the Western core $\mathbf{W}^*$ and the double 3-gon Eastern core $\mathbf{E}^*$ (both gray).
(a) Exterior polar circle $\partial \mathbf{N}^*$ and meridian circle in the octahedral dual $\mathbb{O}= \mathbb{H}^*$.
Only South poles $\mathbf{S} = \mathbf{S}_2$ and $\mathbf{S}_4$ are viable options, trivially equivalent under $\rho$.
Lemma~\ref{lem:5.1} forces the location $w_+^0 = A$, and hence $w_+^1=B$.
All remaining orientations of dual edges follow.
(b) The resulting cube 3-cell template $\mathbb{H}$ with uniquely determined bipolar orientation.
(c) The cube meander $\mathcal{M}$ generated by the SZS-pair $(h_0, h_1)$ of (b), without remaining isotropy.
For the Sturm permutation $\sigma$ see table~\ref{tbl:7.4.1}, case 6.
}}
\label{fig:7.4.3}
\end{figure}

Conversely, suppose the Western core $\mathbf{W}^* \subseteq \mathbb{O}=\mathbb{H}^*$ is 2-dimensional.
Since $\mathbf{W}^*$ contains at most $\eta =3$ vertices, and the octahedron $\mathbb{O}$ consists of triangles, this implies $\mathbf{W}^*$ itself is a Sturm 3-gon with poles $w_-^\iota$.
In fig.~\ref{fig:7.4.2} we choose $\mathbf{W}^*$ to be the exterior face.
We recall that the polar face $\mathbf{N}^*$ is located to the right of the directed edge $e_* = w_-^0w_-^1\in \mathbf{W}^*$.
Swapping our placements of $w_-^\iota = \otimes$ would force the barycenter $\mathbf{N}$ of $\mathbf{N}^*$ to be exterior, off the meridian circle.
The polar bridge options for $w_+^{1-\iota} \in \partial \mathbf{S}^*$ are dotted in fig.~\ref{fig:7.4.2}.
The resulting options for barycenters  $\mathbf{S}$ on the meridian are indicated in the faces $\mathbf{S}_1^*,\ \mathbf{S}_2^*,\ \mathbf{S}_3^*$.
The required left orientation of the polar circle $\partial \mathbf{S}^*$ is not compatible with the direction of the edge $w_+^0w_+^1$, unless $\mathbf{S}^* = \mathbf{S}_3^*$.
But then the polar circles $\partial \mathbf{N}^*,\ \partial\mathbf{S}^*$ are disjoint.
In other words, we are back with the case $\delta =3$ of diagonally opposite poles which we already discussed.
In conclusion, the cases $\delta=3$ and $\dim \mathbf{W}^*=2$ are equivalent.

It remains to study pole distances $\delta=1$ or $\delta =2$ with one-dimensional cores $\dim \mathbf{W}^*=1$.
In particular $\mathbf{W}^*$ is contained in the 3-gon polar circle $\partial \mathbf{N}^*$.
By \eqref{eq:5.6b}, $\mathbf{W}^*$ contains at most one edge $e_+= w_-^0 w_-^1$, i.e. $\eta =2$.
The only alternative for $\mathbf{W}^*$ is the singleton $w_-^0=w_-^1$, where $\eta=1$ defines a single face lift.

We consider the case $\eta=2$ first.
In fig.~\ref{fig:7.4.3}(a), we choose $\mathbf{N}^*$ to be the exterior face with $w_-^\iota$ on the polar boundary circle $\partial \mathbf{N}^*$.
The meridian circle around $\mathbf{W}^* = w_-^0w_-^1$ offers five barycenter locations $\mathbf{S}_1, \ldots , \mathbf{S}_5$ for the South pole $\mathbf{S}$.
The resulting polar bridges $w_-^\iota w_+^{1-\iota}$, which cross the meridian, eliminate choice $\mathbf{S}_3$.
Indeed $w_+^0=w_+^1$, in that case, would force $\eta =5$ because the Eastern core $\mathbf{E}^*= w_+^\iota$ becomes a singleton.
For the pole adjacent choice $\mathbf{S}= \mathbf{S}_1$ the polar bridges force $w_+^1= C,\ w_+^0=A$.
This contradicts the required left orientation of the polar circle $\partial \mathbf{S}_1^*$.
By trivial equivalence, this also eliminates $\mathbf{S} = \mathbf{S}_5$.
It is therefore sufficient to study $\mathbf{S} = \mathbf{S}_2$ with $w_+^0=A,\ w_+^1=B$.
Bipolarity of the Eastern core $\mathbf{E}^*$ fixes the remaining orientations of $\mathbb{O}=\mathbb{H}^*$, and hence of $\mathbb{H}$; see fig.~\ref{fig:7.4.3}(b).
The SZS-pair $(h_0,h_1)$ of (b) defines the meander $\mathcal{M}$ in (c), with the Sturm permutation $\sigma = h_0^{-1}\circ h_1$.
See table~\ref{tbl:7.4.1}, case 6.

\begin{figure}[p!]
\centering \includegraphics[width=\textwidth]{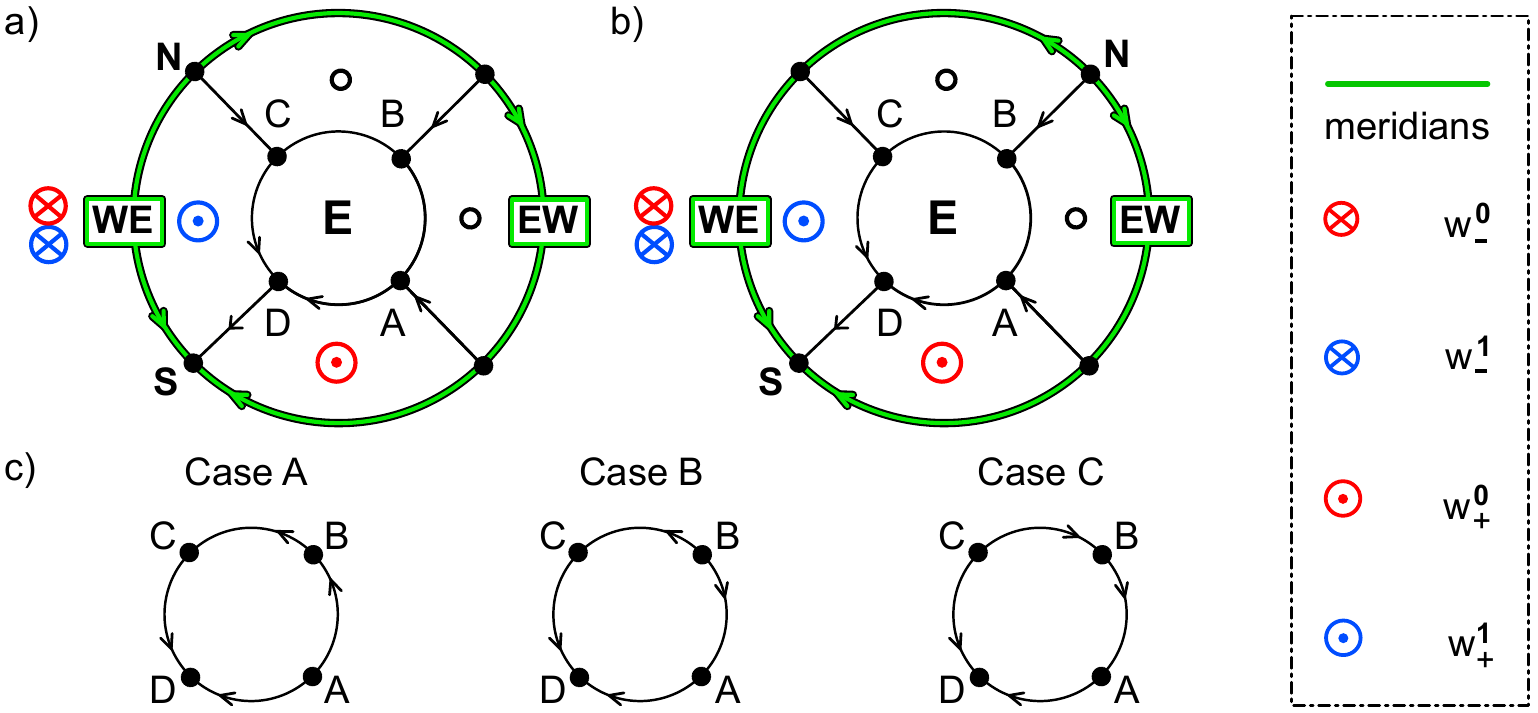}
\caption{\emph{
The three possible bipolar orientations, each, of the Sturm cube with single face Western hemisphere, exterior.
The orientations only differ on the acyclic central square $ABCD$, respectively, for the case (a) of adjacent poles, $\delta =1$, and for the case (b) of diagonally opposite poles, $\delta =2$, on the Western face.
Note how acyclicity of $ABCD$ makes $D$ a local minimum in (c).
The cases $A,\ B,\ C$ refer to the location of the local maximum.
}}
\label{fig:7.4.4}
\end{figure}

\begin{table}[p!]
\centering \includegraphics[width=\textwidth]{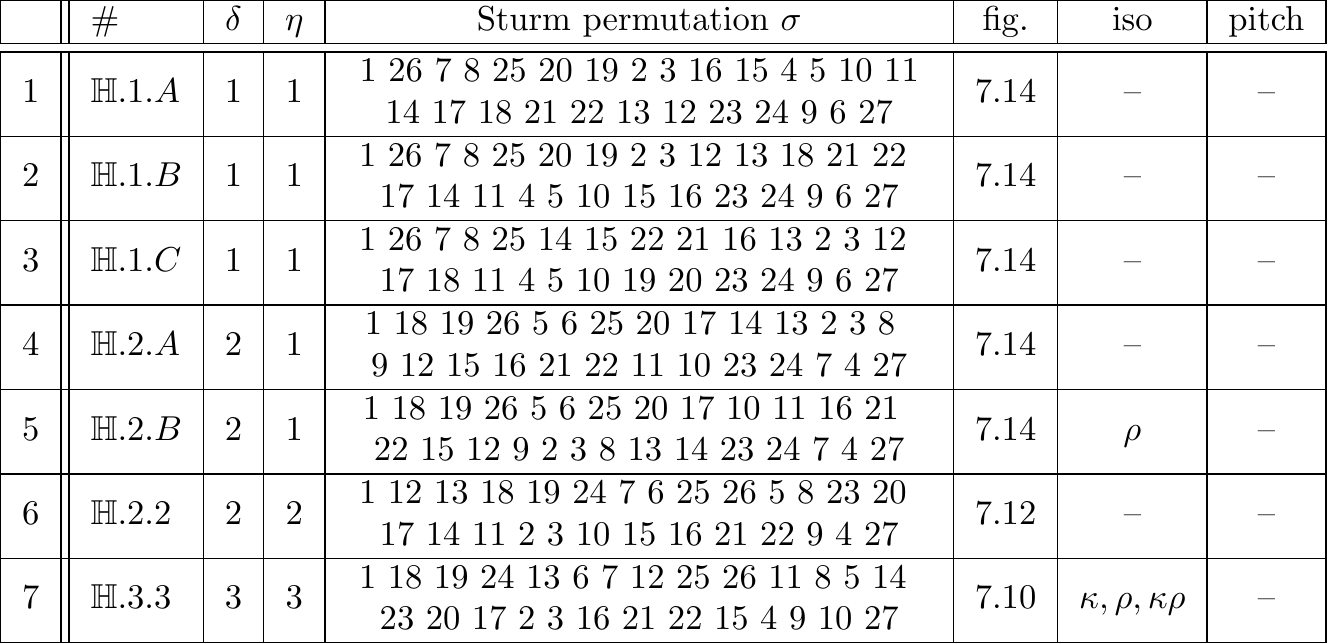}
\caption{\emph{
The seven Sturm hexahedral cubes $\mathbb{H}$.
All pole distances $\delta$ are realized.
The full diameter case $\delta = 3$ is maximally symmetric; see fig.~\ref{fig:7.4.1}. However, it is neither pitchforkable nor realizable by a pendulum type nonlinearity $f=f(u)$.
The nonuniqueness of the single-face lifts $\eta =1$ with pole distances $\delta = 1$ and $\delta =2$, respectively, arises from the choice of the local maximum vertex in the central 4-gon $ABCD$ of $\mathbb{H}$; see fig.~\ref{fig:7.4.4} (c).
The unique double-face lift $\eta =2$ is the third possibility of pole distance $\delta =2$.
}}
\label{tbl:7.4.1}
\end{table}

All remaining cases are single-face lifts, by $\eta =1$.
We argue for the cube $\mathbb{H}$, directly, with $\mathbf{W}$ as the exterior face and the exterior boundary as meridian circle.
We consider the two cases of pole distances $\delta=1$ and $\delta = 2$, separately; see fig.~\ref{fig:7.4.4}(a) and (b).
In either case the boundary meridian orientation follows from the location of the poles $\mathbf{N},\ \mathbf{S}$, with the top edge as the only difference.
The three edges emanating into $\mathbf{E}$ from meridian $i=0$ sinks other than $\mathbf{S}$ are all directed inward, towards $A,\ B,\ C$, respectively.
The fourth edge $D\mathbf{S}$ must be directed towards the South pole $\mathbf{S}$, of course.

Note how $AD$ has to be directed towards $D$.
Indeed, the opposite direction $DA$, and the absence of any other poles besides $\mathbf{N},\ \mathbf{S}$, would force the cyclic orientations $AB,\ BC$, and $CD$, successively.
The directed cycle $ABCD$ contradicts bi-polarity.

\begin{figure}[t!]
\centering \includegraphics[width=\textwidth]{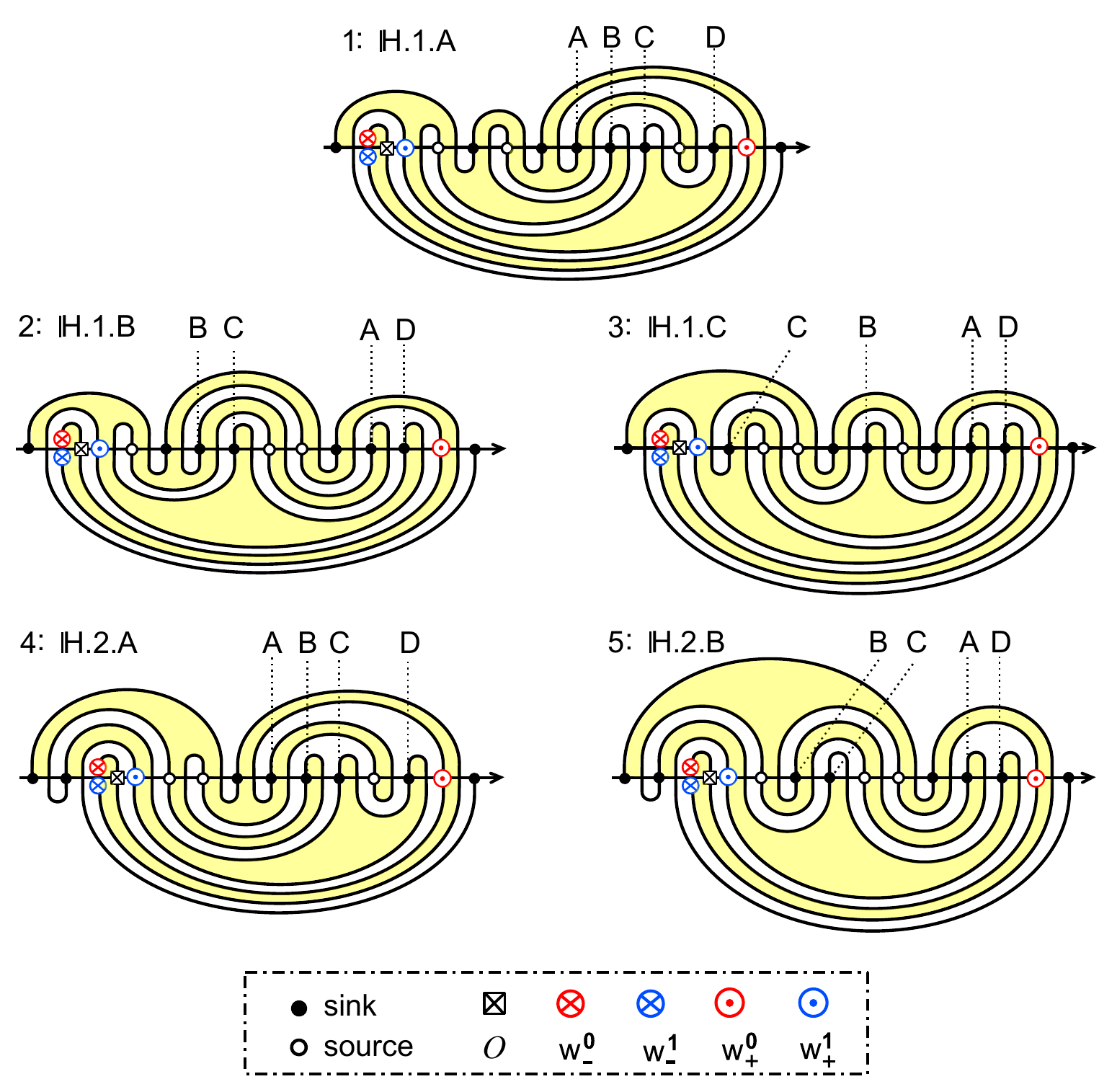}
\caption{\emph{
The five 3-meander templates of single face lift cubes, $\eta =1$.
Note the three identical locations of the four core poles $w_\pm^\iota$, for each pole distance $\delta =1,2$, and the different configurations of the central square $A,\ B,\ C,\ D$.
See fig.~\ref{fig:7.4.4} for bipolar orientations, and table~\ref{tbl:7.4.1} for Sturm permutations and case labels.
}}
\label{fig:7.4.5}
\end{figure}

Likewise, $CD$ has to be directed towards $D$.
This identifies $D$ as a local minimum on the boundary $ABCD$ of the central square.
We distinguish cases
	\begin{equation}
	A,\ B,\ C
	\label{eq:7.4.2}
	\end{equation}
according to the three remaining choices of a local maximum on the boundary $ABCD$ of the central square.
See fig.~\ref{fig:7.4.4}(c).
The only trivial equivalence arises for pole distance $\delta =2$: under $\rho$, cases $A$ and $C$ are then interchangeable.

This determines the bi-polar orientation in all five remaining cases of single-face lifted cube 3-ball complex templates.
See fig.~\ref{fig:7.4.5} for the five resulting Sturm meanders.

In table~\ref{tbl:7.4.1} we summarize the seven inequivalent 3-meander templates for cubes $\mathbb{H}$.
In cases 1--5, the Western hemisphere is a single face, $\eta =1$.
The minimal pole distance $\delta =1$ only occurs for $\eta = 1$; see cases 1--3.
The three cases differ by the choice of the locally maximal $i=0$ sink vertex $A,\ B,\ C$ on the central 4-gon of the Eastern hemisphere; see fig.~\ref{fig:7.4.4}.
Diagonally opposite poles across a face of the cube, $\delta = 2$, may arise for single face and for double face Western hemispheres, i.e. for $\eta =1$ and for $\eta =2$.
The two cases $4,\ 5$ are characterized by $\delta =2,\ \eta =1$, and differ by the locally maximal vertex $A,\ B$ of the central 4-gon, in the bipolar orientation.
Case 6 is characterized uniquely by its double face Western hemisphere, $\eta =2$.
It is the third case of face diagonally opposite poles, $\delta = 2$.
The final case 7, treated first in the present section, is equivalently characterized by the requirement of space diagonally opposite poles, $\delta =3$, or a face count $\eta =3$ in each hemisphere.
In each hemisphere, the three faces share one vertex.
In other words, each dual core is a Sturm 3-gon.

\subsection{Sturm icosahedra and dodecahedra}\label{subsec7.5}

We do not aim for complete case lists, in this section.
Instead, we explore the possible pole distances $\delta$ and Western, i.e. smaller hemisphere, face counts $\eta$ for solid Sturm icosahedra $\mathbb{I}$ and dodecahedra $\mathbb{D}$.
See theorems~\ref{thm:7.5.1} and \ref{thm:7.5.2}.
A priori, 
	\begin{equation}
	\begin{aligned}
	1\leq \delta \leq \vartheta =3 \quad &\text{and} \quad
	1\leq \eta \leq c_2/2 =10 \quad &\text{for}& \quad \mathbb{I}\,,\\
	1\leq \delta \leq \vartheta =5 \quad &\text{and} \quad
	1\leq \eta \leq c_2/2 =6 \quad &\text{for}& \quad \mathbb{D}\,.
	\end{aligned}
	\label{eq:7.5.1a}
	\end{equation}

\begin{figure}[t!]
\centering \includegraphics[width=\textwidth]{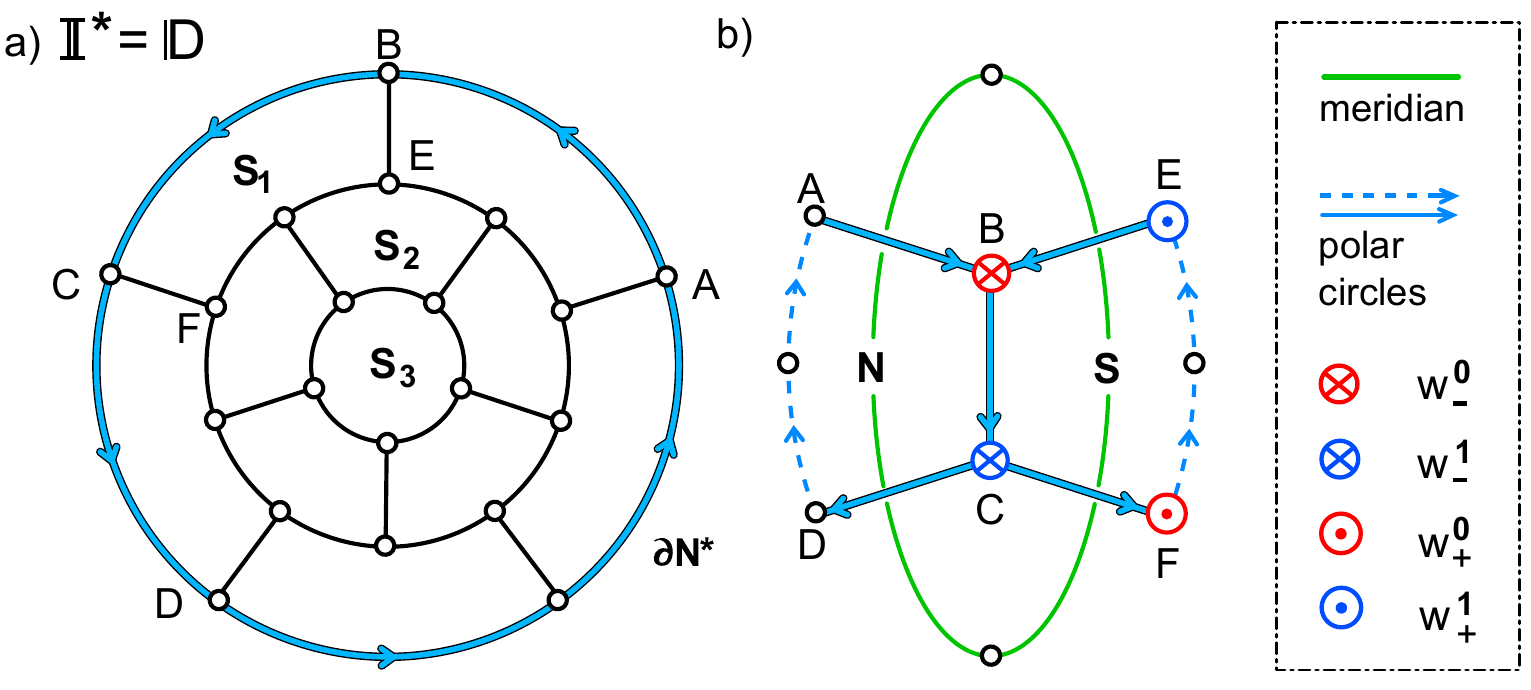}
\caption{\emph{
The dodecahedral dual $\mathbb{I}^*=\mathbb{D}$ of the icosahedron complex $\mathbb{I}$.
(a) Exterior polar face dual $\mathbf{N}^*$ with oriented boundary $\partial \mathbf{N}^*$.
Representative barycenters $\mathbf{S}_\delta$ of candidate face duals $\mathbf{S}_\delta^*$ denote South poles $\mathbf{S}$ at distances $\delta = 1,2,3$ from the exterior North pole $\mathbf{N}$.
Note the single bridge $BE$ between the polar circle $\partial \mathbf{N}^*$ and $\partial \mathbf{S}_2^*$, as well as the absence of bridges between $\partial \mathbf{N}^*$ and $\partial \mathbf{S}_3^*$.
(b) Viable placement of the four pole $w_\pm^\iota$ of the dual cores $\mathbf{W}^*,\ \mathbf{E}^*$ in case $\mathbf{S} = \mathbf{S}_1$.
Only the locations $A, \ldots , F$ allow for single-edge directed bridges $w_\pm ^1w_\mp^0$.
The bridges must lie in the solid boundary parts of the two adjacent pentagonal polar faces.
The orientations of the polar circles $\partial \mathbf{N}^*,\ \partial \mathbf{S}_1^*$ are indicated, and result in the green meridian circle.
}}
\label{fig:7.5.1}
\end{figure}

For the Sturm icosahedron Thom-Smale complex, it turns out that poles $\mathbf{N},\ \mathbf{S}$ must always be neighbors:
$\delta = 1$.
The maximal Western face count is $\eta = 2$.
See figs.~\ref{fig:7.5.3}, \ref{fig:7.5.4} and case 1 in table~\ref{tbl:7.5.2} for an example.
For the Sturm dodecahedron, the maximal pole distance is $\delta =2$.
It arises if, and only if, the Western face count is $\eta =2$.
See figs.~\ref{fig:7.5.5}, \ref{fig:7.5.6} and case 2 in table~\ref{tbl:7.5.2} for an example.

\begin{thm}\label{thm:7.5.1}
Consider any Sturm icosahedron $\mathbb{I}$.

Then the poles $\mathbf{N},\ \mathbf{S}$ are (edge) neighbors, and the (smaller) Western hemisphere face count $\eta$ is at most 2.

For $\eta =2$, the poles are located at the endpoints of the unique shared, non-meridian edge of the two Western triangle faces.
\end{thm}

\begin{proof}[\textbf{Proof.}]
We proceed by decreasing pole distance $\delta \leq \vartheta =3$, via the dual dodecahedron $\mathbb{D} = \mathbb{I}^*$.
See table~\ref{tbl:7.1.1} and figs.~\ref{fig:7.1.1}, \ref{fig:7.5.1}.

By corollary~\ref{cor:5.2}(ii), the polar pentagon circles in $\mathbb{I}^*$ must be joined by at least two polar bridges.
In fig.~\ref{fig:7.5.1}(a), the polar circle $\partial \mathbf{N}^*$ is the boundary of the exterior face $\mathbf{N}^*$.
Up to trivial equivalences, the three barycenter options $\mathbf{S} \in \lbrace \mathbf{S}_1, \mathbf{S}_2, \mathbf{S}_3 \rbrace$ arise.
Note the pole distance is $\delta$, for $\mathbf{S} = \mathbf{S}_\delta$.

If $\mathbf{S} = \mathbf{S}_3$, then there is no polar bridge.
Therefore $\delta \leq 2$.
If $\mathbf{S} = \mathbf{S}_2$ there is a unique polar bridge, instead of the required two bridges.
This proves $\delta =1$, i.e. $\mathbf{S} = \mathbf{S}_1$ and the poles $\mathbf{N},\ \mathbf{S}$ are edge adjacent in the icosahedron $\mathbb{I}$.

We show $\eta \leq 2$ for the Western face count under adjacent poles.
See fig.~\ref{fig:7.5.1}(b) for the polar circles $\partial \mathbf{N}^*,\ \partial\mathbf{S}^*$.
The solid parts show all candidate edges $e^* \in \mathbb{I}^* = \mathbb{D}$ which are potential polar bridges between end points on different polar circles.
Note that all (solid) bridge candidates happen to be contained in the union of the polar circles themselves, here.

Suppose $\eta > 1$.
Then the four pole vertices $w_\pm^\iota$ of the dual cores $\mathbf{W}^*,\ \mathbf{E}^*$ are all disjoint.
They must be placed on the solid part of fig.~\ref{fig:7.5.1}(b), to afford the polar bridges of corollary~\ref{cor:5.2}(ii), directed from $w_\pm^1$ to $w_\mp^0$.
The North polar face $\mathbf{N}^*$ is located to the right of the directed polar circle segment from $w_-^0$ to $w_-^1$ in $\partial \mathbf{N}^* \cap \partial \mathbf{W}^*$.
Similarly, the South polar face $\mathbf{S}^*$ is located to the left of the directed polar circle segment from $w_+^0$ to $w_+^1$ in $\partial \mathbf{S}^* \cap \partial \mathbf{E}^*$. 
This implies
	\begin{equation}
	w_-^\iota \in \lbrace A, B, C, D \rbrace\,, \qquad
	w_+^\iota \in \lbrace E, B, C, F \rbrace\,,
	\label{eq:7.5.1b}
	\end{equation}
We proceed by location of $w_-^1$.

Suppose $w_-^1 =D$.
Then there is no directed single-edge bridge $w_-^1 w_+^0$.
Suppose $w_-^1=B$.
Then the only directed single-edge bridge from $w_-^1 \in \partial \mathbf{N}_*$ to $w_+^0 \in \partial \mathbf{S}_*$ is $BC$.
Hence $w_+^0=C$.
With $B,\ C$ already occupied, however, there does not remain any bridge from $w_+^1 \in \lbrace E, F \rbrace$ to $w_-^0 \in \lbrace A, D \rbrace$.
We illustrate the case $w_-^1 = C$ in fig.~\ref{fig:7.5.1}(b).
The remaining case $w_-^1 = A$ leads to the symmetric case $w_+^0 = B,\ w_+^1 =C,\ w_+^0 = D$, and just provides a trivially equivalent Sturm realization.

The polar bridges $e^*$ are duals to meridian edges.
Likewise, the pentagon edges preceding $w_\pm^0$ and following $w_\pm^1$, on their respective polar circles, are duals to meridian edges.
In $\mathbb{I}^* = \mathbb{D}$, these four meridian edges define the meridian circle which encloses the Western hemisphere with $\eta =2$ faces given by $w_-^\iota$.
This proves that $\eta > 1$ implies $\eta =2$.
The interior Western edge from $\mathbf{N}$ to $\mathbf{S}$ follows because the shared edge $BC$ of the polar circles is not dual to a meridian edge.
This proves the theorem.
\end{proof}

\begin{thm}\label{thm:7.5.2}
Consider any Sturm dodecahedron $\mathbb{D}$.

Then the maximal pole distance $\delta$ and the maximal (smaller) hemisphere face count $\eta$ satisfy
	\begin{equation}
	1\leq \eta \leq \delta \leq 2\,,
	\label{eq:7.5.1c}
	\end{equation}
and all these cases do occur. For $\eta = \delta =2$, the poles $\mathbf{N},\ \mathbf{S}$ are located asymmetrically at edge distance $\delta =2$ via the unique shared edge of the two Western pentagon faces.
Their edge distance along the meridian circle of edge circumference eight is three.
\end{thm}

\begin{figure}[t!]
\centering \includegraphics[width=\textwidth]{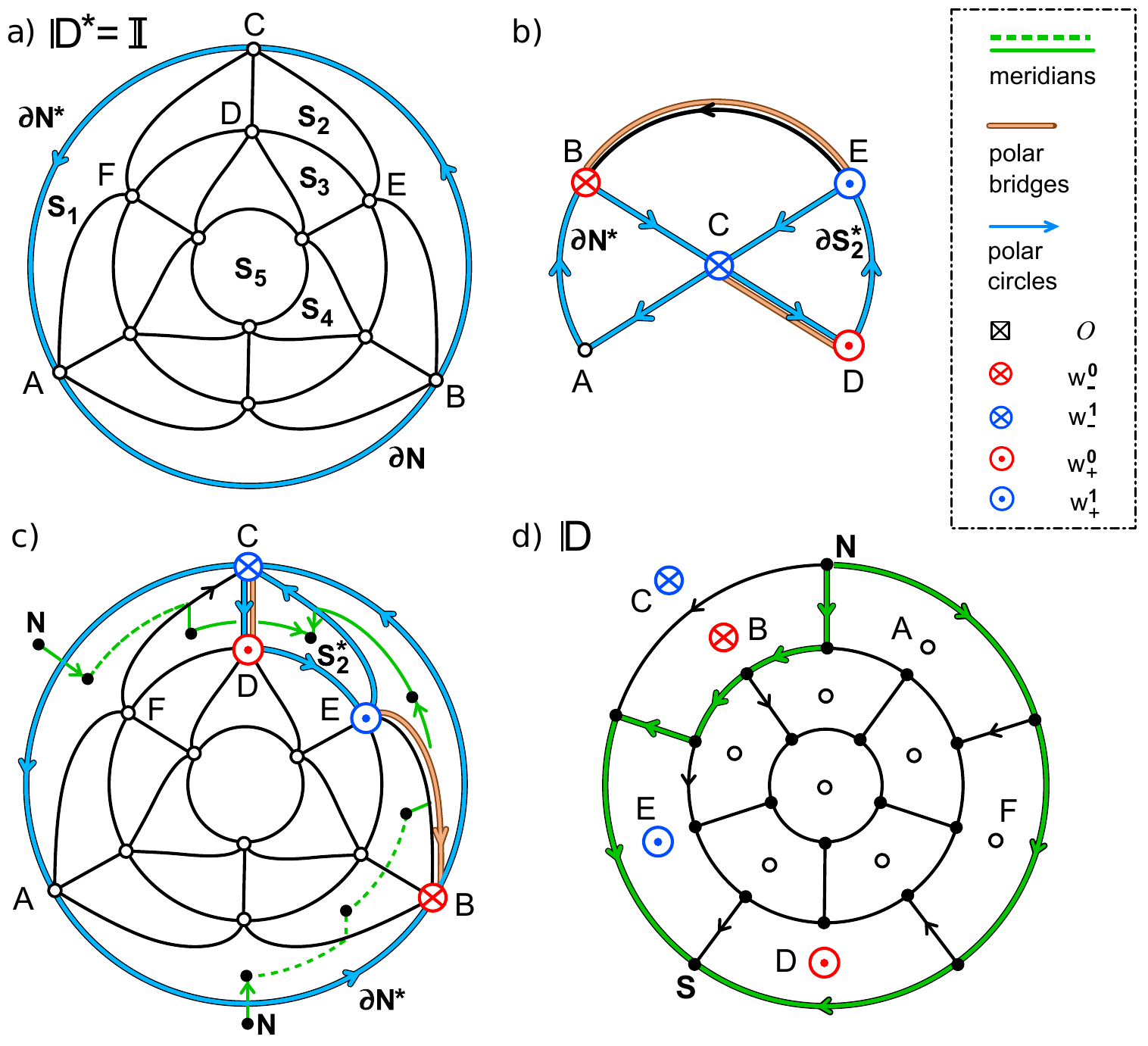}
\caption{\emph{
The icosahedral dual $\mathbb{D}^* = \mathbb{I}$ and the dodecahedral complex $\mathbb{D}$.
(a) Exterior polar face dual $\mathbf{N}^*$ with oriented boundary $\partial \mathbf{N}^*$.
Representative candidate face duals $\mathbf{S}_\delta^*$ indicate South poles $\mathbf{S}$ at distances $\delta =1, \ldots , 5$ from the North pole $\mathbf{N}$.
Bridges between $\partial \mathbf{S}_5^*$ and $\partial \mathbf{N}^*$ are absent. Bridges are unique between $\partial \mathbf{S}_4^*$ and $\partial \mathbf{N}^*$.
(b) Placement of the directed 4-cycle $BCDEB$ of $w_\pm^\iota$ in (a), for the case $\mathbf{S}= \mathbf{S}_2$ of pole distance $\delta = 2$.
See also table~\ref{tbl:7.5.1}.
(c) The resulting meridian segments (green, solid) in $\mathbb{I} = \mathbb{D}^*$ for the configuration (b) of $w_\pm^\iota$.
For the closure of the meridian circle (green, dashed) see text.
(d) The resulting hemisphere decomposition with pole distance $\delta = 2$ and $\eta = 2$ Western faces $w_-^0 = B,\ w_-^1= C$, in the original dodecahedron 3-cell template $\mathbb{D}$.
Only mandatory parts of the bipolar orientation are indicated.
}}
\label{fig:7.5.2}
\end{figure}

\begin{proof}[\textbf{Proof.}]
Similarly to the proof of theorem~\ref{thm:7.5.1}, we proceed by decreasing pole distance $\delta \leq \vartheta=5$, this time via the dual icosahedron $\mathbb{I} = \mathbb{D}^*$.
See table~\ref{tbl:7.1.1} and figs.~\ref{fig:7.1.1}, \ref{fig:7.5.2}.

The candidates for the South pole $\mathbf{S}$ at distance $1 \leq \delta \leq 5$ from the exterior North pole barycenter $\mathbf{N}$ are $\mathbf{S}_\delta$, in fig.~\ref{fig:7.5.2}(a), up to trivial equivalences.
Evidently the polar circles $\partial \mathbf{S}_5^*$ and $\partial \mathbf{S}_4^*$ do not possess two edge-disjoint polar bridges to the boundary polar circle $\partial \mathbf{N}^*$.
Therefore $\delta \leq 3$.

Suppose $\delta =3,\ \mathbf{S} = \mathbf{S}_3$.
Then the three edges $BE$, $CD$, and $CE$ are the only polar bridge candidates.
Two cases arise.

First suppose $\eta=1$. 
Then $\mathbf{W}^*$ is the singleton $w_-^0 = w_-^1 = C$.
In particular, the meridian circle surrounds $C =\mathbf{W}^*$, as the only dual vertex.
Therefore the meridians cannot reach the South pole barycenter $\mathbf{S}$ of the opposite polar circle $\partial \mathbf{S}^* = \partial \mathbf{S}_3^*$.
This contradicts lemma~\ref{lem:5.1}(iii).

Second suppose $\eta>1$.
Then $w_-^0 \neq w_-^1$ occupy both candidate North polar vertices $B$ and $C$ of bridges between the polar circles.
By corollary \ref{cor:5.2}(iv), the directed segment $w_-^0 w_-^1 = BC$ is the intersection of $\mathbf{W}^*$ with the triangular North polar circle $\partial \mathbf{N}^*$.
In particular $w_-^0=B$ and $w_-^1=C$.
The polar bridges then imply $w_+^1=E$ and, because the face count of $\mathbf{E}^*$ is at least $\eta>1$, also $w_+^0=D$.
See corollary \ref{cor:5.2}(v).
However, the resulting clockwise direction $w_+^0w_+^1=DE$ contradicts the  counter-clockwise orientation of  $\partial \mathbf{S}^*$ required by lemma \ref{lem:5.1}(ii).

Consider $\delta =2,\ \mathbf{S} = \mathbf{S}_2$ next, as indicated in fig.~\ref{fig:7.5.2}(a).
All core poles $w_\pm^\iota$ must be placed on the union of polar circles, i.e. at one of the five locations
	\begin{equation}
	w_\pm^\iota \in \lbrace A, B, C, D, E \rbrace
	\label{eq:7.5.2}
	\end{equation}
of fig.~\ref{fig:7.5.2}(b). The mandatory single-edge polar bridges must appear in the same reduced diagram.

Suppose $\eta =1$ first, i.e. the Western core $\mathbf{W}^* = \lbrace w_-^\iota \rbrace$ is a singleton.
Then many options arise, all of which require
	\begin{equation}
	w_-^0 = w_-^1 = C\,,\quad w_+^0=D\,,\quad w_+^1 =E \,.
	\label{eq:7.5.3a}
	\end{equation}
The two poles $\mathbf{N},\ \mathbf{S}$, are therefore located non-adjacently on the boundary of the single Western pentagon face.
We omit further details on this case, which certainly meets the claims of the theorem.

In case $\delta =2,\ \eta \geq 2$, the four core poles $w_\pm^\iota$ are all distinct.
Because all dual faces are 3-gons of boundary length $n=3$, corollary~\ref{cor:5.2}(vi) implies that the segment $w_\pm^0 w_\pm^1$, on the appropriate polar circle, consists of a single directed edge; see \eqref{eq:5.6b}.
Together with the directed single-edge polar bridges $w_\pm^1 w_\mp^0$, this defines a directed 4-cycle
	\begin{equation}
	w_-^0 w_-^1 w_+^0 w_+^1 w_-^0
	\label{eq:7.5.3b}
	\end{equation}
of four mutually disjoint edges	in the reduced diagram of fig.~\ref{fig:7.5.2}(b).
See corollary~\ref{cor:5.2}(iv).
Only the direction of the nonpolar edge $BE$ can still be chosen freely.
Note $w_-^\iota \in \partial \mathbf{N}^*$ implies $w_-^\iota \in \{A,B,C\}$.
Similarly $w_+^\iota \in \partial \mathbf{S}^*$ implies $w_+^\iota \in \{C,D,E\}$.
In table~\ref{tbl:7.5.1} we list all resulting options, left to right, starting from $w_-^0 \in \{ A, B, C\}$.
Evidently, the only possible directed 4-cycles \eqref{eq:7.5.3b} are
	\begin{equation}
	w_-^0 w_-^1 w_+^0 w_+^1 w_-^0 = ABECA\quad \text{or}\quad BCDEB\,.
	\label{eq:7.5.4}
	\end{equation}
The two cycles are trivially equivalent under $\rho$.
The cycle $BCDEB$ is indicated in fig.~\ref{fig:7.5.2}(b),(c).

\begin{table}[b!]
\centering \includegraphics[width=0.4\textwidth]{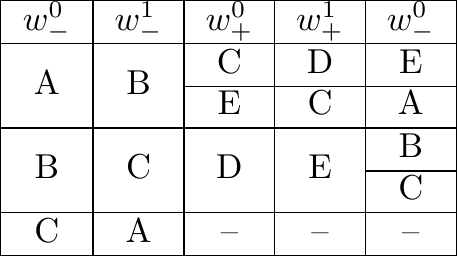}
\caption{\emph{
Realization of the directed 4-cycle \eqref{eq:7.5.3b} in fig.~\ref{fig:7.5.2}(b).
The directed edge $w_-^0w_-^1$ has to follow the oriented polar circle $\partial \mathbf{N}^*$, and $w_+^0w_+^1$ follows $\partial \mathbf{S}_2^*$.
The polar bridges $w_\pm^1w_\mp^0$ encounter two options for $w_\mp^0$, when $w_\pm^1\in \{ B,\ E\}$.
There is no bridge from $w_-^1= A$ to $\partial \mathbf{S}_2^*$.
The two possible directed cycles are therefore $ABECA$ and $BCDEB$, trivially equivalent under the hemisphere exchange $\rho$.
}}
\label{tbl:7.5.1}
\end{table}

For the Western face count $\eta =2$, i.e. for single-edge dual cores $\mathbf{W}^*$, the meridian circle in fig.~\ref{fig:7.5.2}(c) then follows from corollary~\ref{cor:5.2}(iii),(v):
it encloses the dual core $\mathbf{W}^* = w_-^0w_-^1$.
Converting the meridian circle around $\mathbf{W}^*$ back to the original dodecahedron $\mathbb{D}$, we easily identify the Western interior $\mathbf{W}$ as two pentagon faces with barycenters $B= w_-^0$ and $C=w_-^1$, and a single shared edge dual to $BC$.
The meridian therefore possesses circumference length eight.
The relative location of the pentagons with barycenters $A, B, C, D, E$ is easily derived from fig.~\ref{fig:7.5.2}(c); see fig.~\ref{fig:7.5.2}(d).
The locations of the poles $\mathbf{N}$ and $\mathbf{S}$ on the meridian follow just as easily.
The bipolar orientation of $\mathbb{D}$ remains partially undetermined.

We still have to show that $\eta \geq 2$ actually implies the above Western face count $\eta =2$.
This is slightly subtle.
We first show $FC \notin \mathbf{W}^*$, indirectly, as a first step towards closing the dashed meridian gap dual to $FC$ in fig.~\ref{fig:7.5.2}(c).
Indeed suppose $FC \in \mathbf{W}^*$.
By bipolarity of $\mathbf{W}^*$, we can then follow a directed path in $\mathbf{W}^*$, upwards against its orientation all the way, to the North pole $B=w_-^0$ of $\mathbf{W}^*$.
The downward edge $w_-^0w_-^1$ from $B$ to $C$ closes the path to a nonoriented cycle $\Gamma$ in $\mathbf{W}^*$ which does not intersect the dual meridian cycle.
But the meridian circle contains edges on either side of $\Gamma$:
the duals to $CD$ and $AC$, for example.
This contradicts the Jordan curve theorem on $S^2$, and proves $FC \notin \mathbf{W}^*$ is dual to a meridian edge.
An analogous argument closes the meridian circle through the two edges from $B = w_-^0$ which had not been accounted for, so far.
This proves
	\begin{equation}
	\mathbf{W}^* = w_-^0w_-^1
	\label{eq:7.5.5}
	\end{equation}
and hence $\mathbf{W}$ consists of only $\eta =2$ faces with barycenters $w_-^0=B$ and $w_-^1 =C$, as discussed above.

To complete the proof of theorem~\ref{thm:7.5.2}, it only remains to discuss the case $\delta =1$ of edge adjacent poles next, i.e. $\mathbf{S} = \mathbf{S}_1$ in fig.~\ref{fig:7.5.2}(a).
We claim $\eta =1$.
Suppose, indirectly, $\eta \geq 2$.
Then $w_\pm^\iota$ define a directed 4-cycle \eqref{eq:7.5.3b} in the union of polar circles $\partial \mathbf{N}^*,\ \partial \mathbf{S}^*$, again:
	\begin{equation}
	w_\pm^\iota \in \lbrace A, B, C, F\rbrace\,.
	\label{eq:7.5.6}
	\end{equation}
All edges of the 4-cycle must likewise be contained in $\partial \mathbf{N}^* \cup \partial \mathbf{S}^*$, following the given orientation.
Alas, there does not exist any directed 4-cycle in this configuration.
Therefore $\eta =1$, as claimed.
This proves the theorem.
\end{proof}

\begin{table}[t!]
\centering \includegraphics[width=\textwidth]{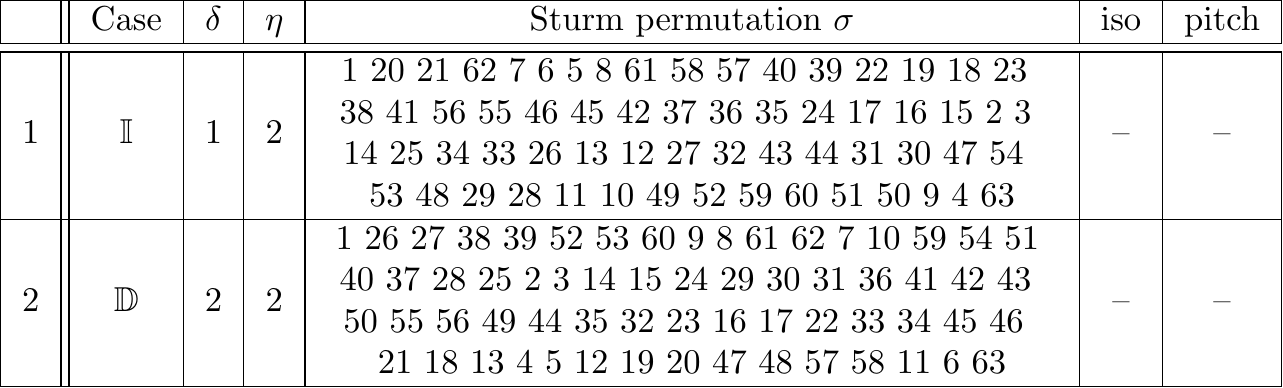}
\caption{\emph{
Two examples of Sturm permutations which lead to one of many icosahedral and dodecahedral 3-cell templates and Sturmian Thom-Smale complexes $\mathbb{I}$ and $\mathbb{D}$, respectively.
The number $\eta = 2$ of faces in the Western hemisphere, and the pole distances $\delta= 1,2$, are maximal in each case.
}}
\label{tbl:7.5.2}
\end{table}

\begin{figure}[p!]
\centering \includegraphics[width=\textwidth]{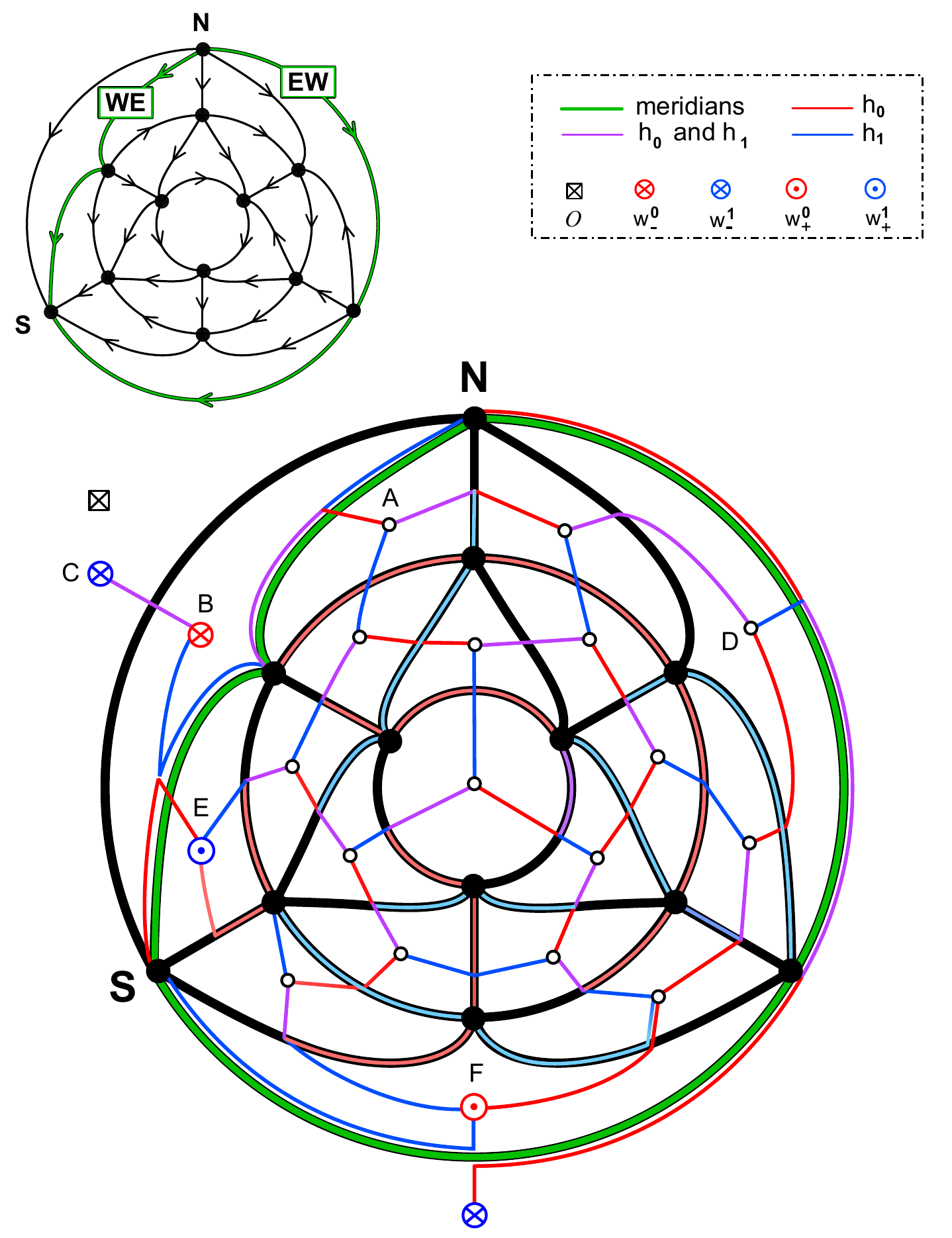}
\caption{\emph{
A sample Sturm icosahedron Thom-Smale complex $\mathbb{I}$ with pole distance $\delta =1$ and with $\eta = 2$ Western faces $\otimes$.
Note the required orientation arrows emanating from the North pole $\mathbf{N}$, directed away from the meridians into the Eastern hemisphere, and terminating at the South pole $\mathbf{S}$.
The SZS-pair $(h_0,h_1)$ results from the bipolar orientation: $h_0$ (red), $h_1$ (blue), $h_0 \& h_1$ (purple).
See fig.~\ref{fig:7.5.4} and table~\ref{tbl:7.5.2} for the Sturm meander $\mathcal{M}$ of the resulting Sturm permutation $\sigma = h_0^{-1}\circ h_1$.
}}
\label{fig:7.5.3}
\end{figure}

\begin{figure}[p!]
\centering \includegraphics[width=\textwidth]{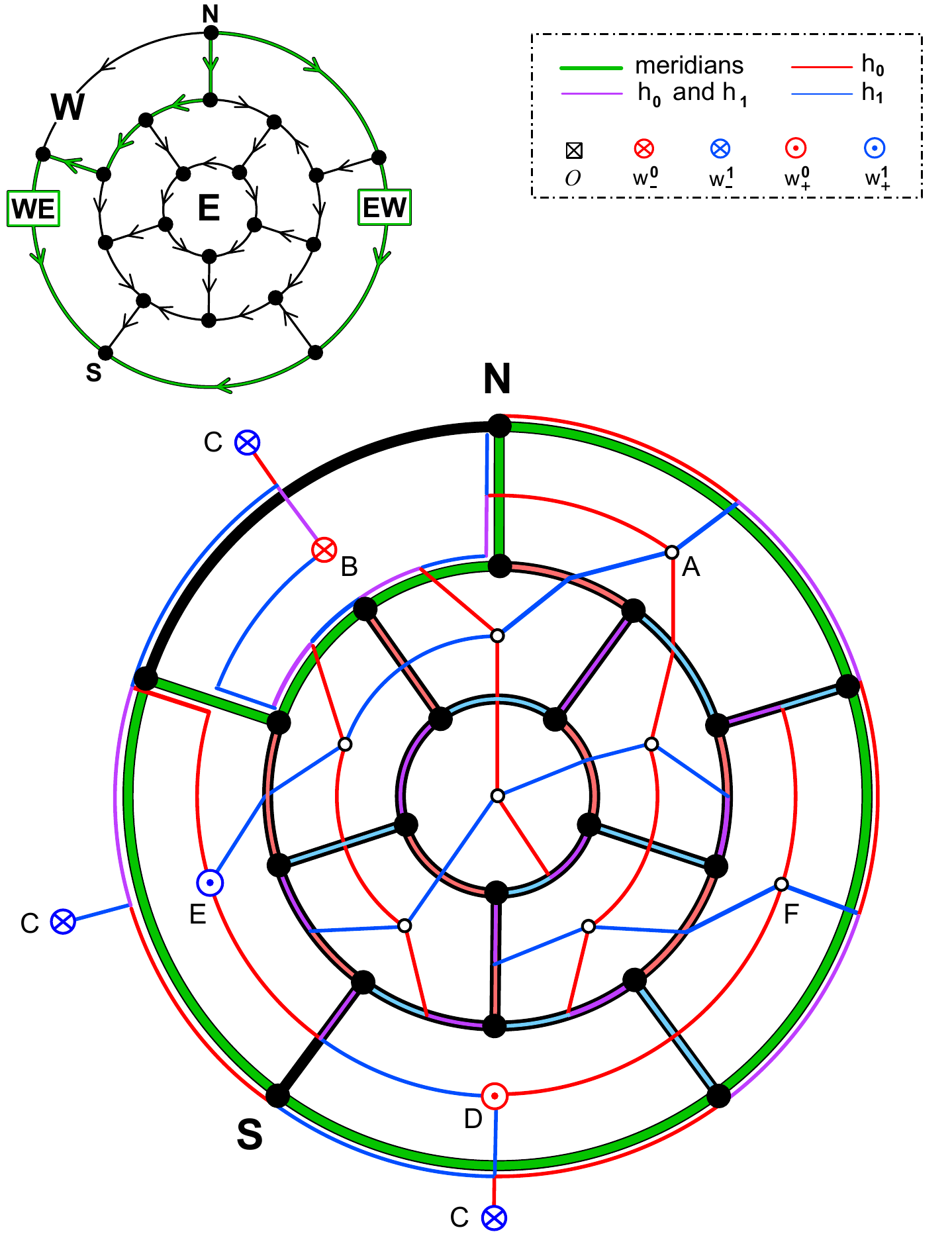}
\caption{\emph{
A sample Sturm dodecahedron Thom-Smale complex $\mathbb{D}$ with maximal pole distance $\delta =2$ and with $\eta = 2$ Western faces.
The SZS pair $(h_0,h_1)$ results from the bipolar orientation; see also fig.~\ref{fig:7.5.2}(d).
See fig.~\ref{fig:7.5.6} and table~\ref{tbl:7.5.2} for the Sturm meander $\mathcal{M}$ of the Sturm permutation $\sigma = h_0^{-1} \circ h_1$.
}}
\label{fig:7.5.5}
\end{figure}

\begin{figure}[p!]
\centering \includegraphics[width=\textwidth]{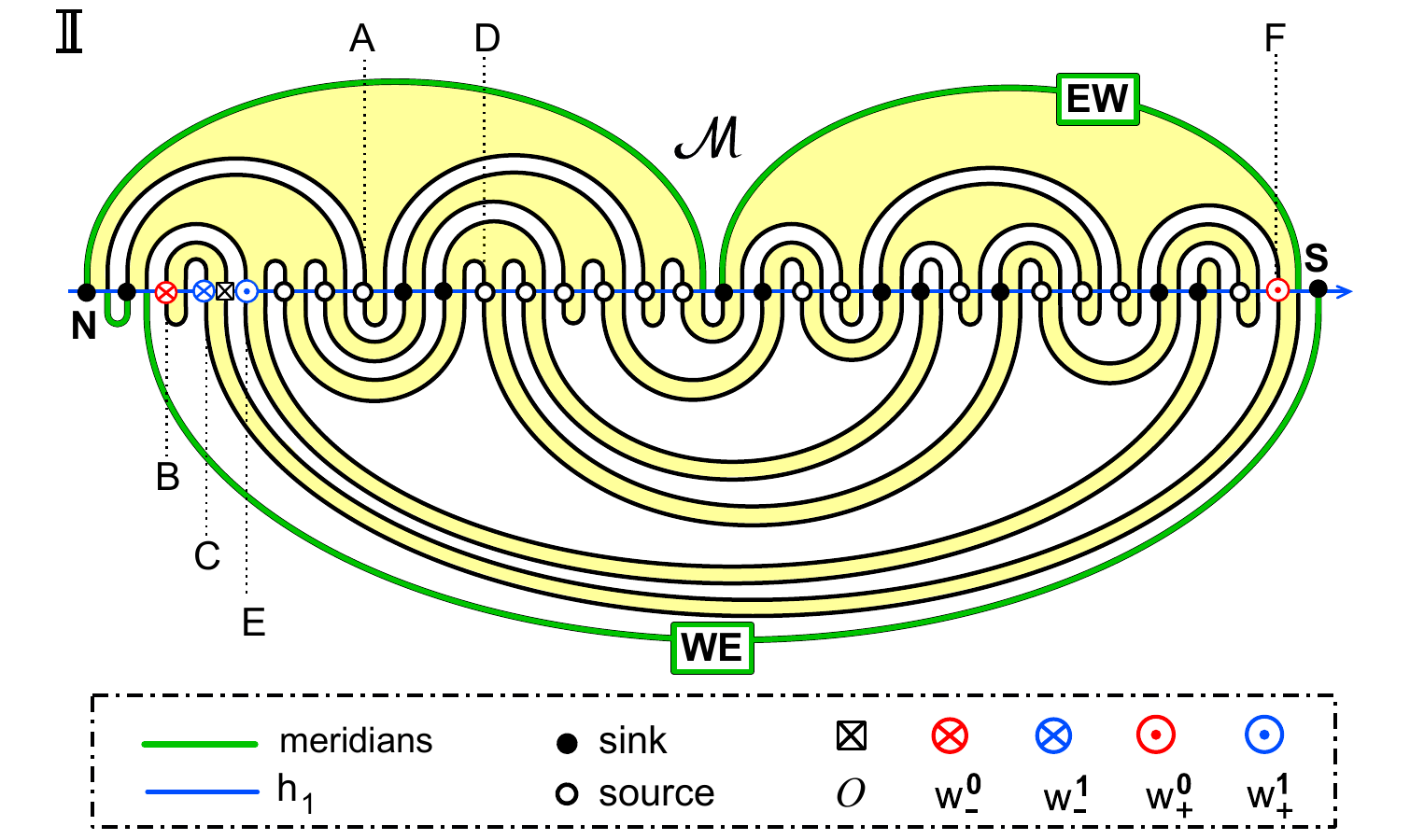}
\caption{\emph{
The Sturm meander $\mathcal{M}$ for the icosahedron $\mathbb{I}$ of fig.~\ref{fig:7.5.3}.
The marked sources $A, \ldots , F$ correspond to fig.~\ref{fig:7.5.1}(b) and to the icosahedral Thom-Smale complex $\mathbb{I}$.
Note the extreme positions of the poles $w_\pm^\iota$ of the dual cores $\mathbf{W}^*,\ \mathbf{E}^*$.
}}
\label{fig:7.5.4}
\end{figure}

\begin{figure}[p!]
\centering \includegraphics[width=\textwidth]{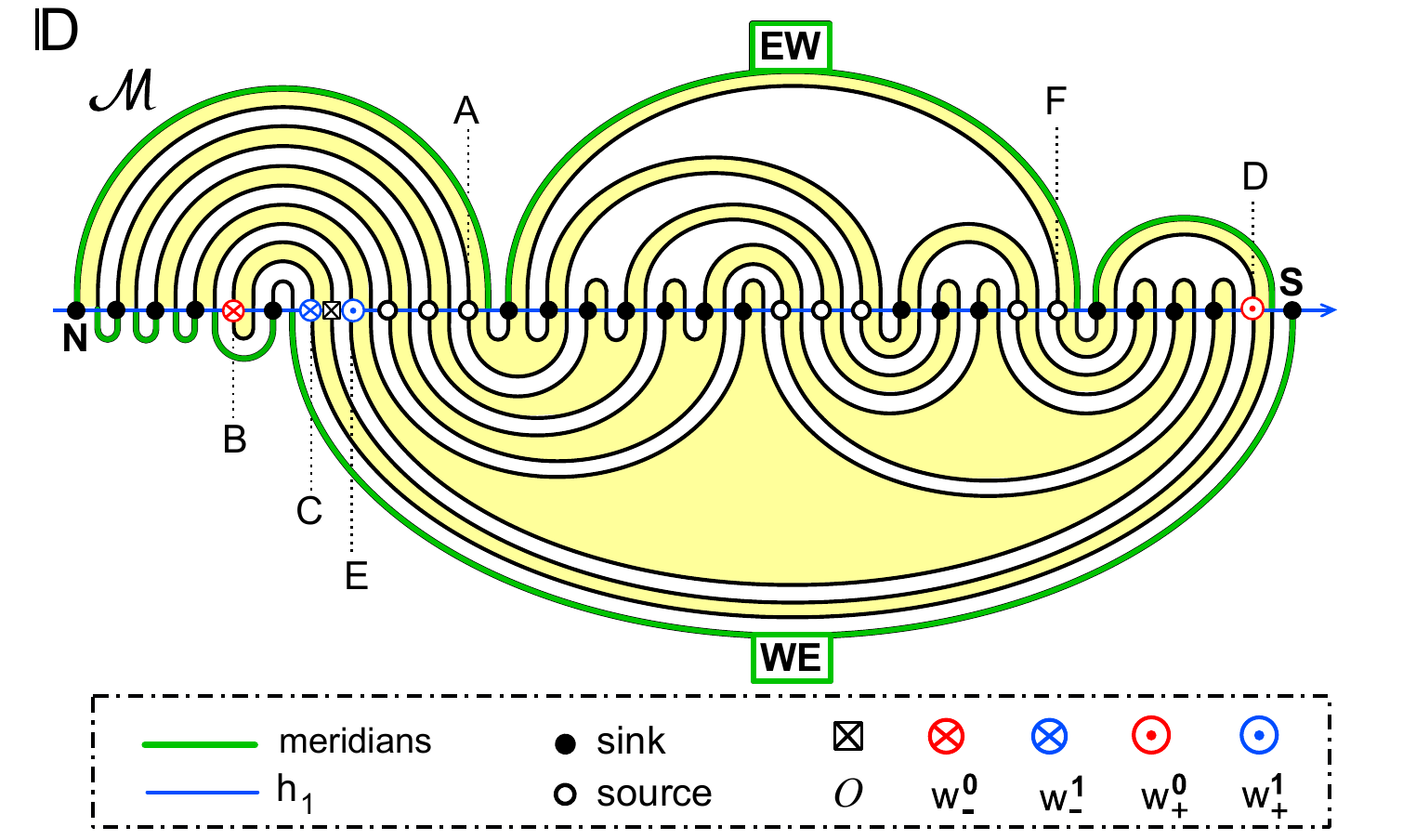}
\caption{\emph{
The Sturm meander $\mathcal{M}$ for the dodecahedron $\mathbb{D}$ of fig.~\ref{fig:7.5.5}.
The marked sources $A, \ldots , F$ correspond to figs.~\ref{fig:7.5.2} and \ref{fig:7.5.5}.
For further comments on $w_\pm^\iota$; see fig.~\ref{fig:7.5.4}.
}}
\label{fig:7.5.6}
\end{figure}

We conclude this section with one example, each, for largest 2-face Western hemispheres and maximal pole distances, $\delta =1$ in the icosahedron, and $\delta =2$ in the dodecahedron.
See figs.~\ref{fig:7.5.3} and \ref{fig:7.5.4} respectively.
The basic configuration of poles and meridians with overlap, satisfying the requirements of definition~\ref{def:1.1} of 3-cell templates, follows the proofs of theorems~\ref{thm:7.5.1} and \ref{thm:7.5.2}.
Orientations of nonpolar edges on the meridian, and away from the meridian in the Eastern hemisphere, follow the requirements of that definition.
We have picked the remaining orientations of the Eastern 1-skeleton, from the many bipolar possibilities, somewhat arbitrarily.
The SZS-pairs $(h_0,h_1)$ then define the Sturm permutations $\sigma$, as in table~\ref{tbl:7.5.2}, and the Sturm 3-meander templates of figs.~\ref{fig:7.5.4} and \ref{fig:7.5.6}.


\section{Conclusion and outlook}
\label{sec8}

We have concluded our trilogy on 3-ball PDE Sturm global attractors $\mathcal{A}_f$.
We have shown how their dynamic signed hemisphere complexes $\mathcal{C}_f$, the 3-cell templates $\mathcal{C}$, the 3-meander templates $\mathcal{M}$, their ODE shooting meanders $\mathcal{M}_f$, and their associated permutations $\sigma$ and $\sigma_f$, are all equivalent descriptions of one and the same geometric object:
not just the ODE critical points of the PDE Lyapunov function, alias the equilibria, but a signed version of the Thom-Smale complex defined by their PDE heteroclinic orbits.
In particular, the definition of unique SZS-pairs $(h_0,h_1)$ in abstract 3-cell templates $\mathcal{C}$ allowed us to design Sturm global attractors such that their signed Thom-Smale dynamic complex $\mathcal{C}_f$ coincides with any prescribed 3-cell template $\mathcal{C}$.
The construction resulted from a nonlinearity $f$ such that its Sturm permutation $\sigma_f$ satisfies
	\begin{equation}
	\sigma_f = \sigma:= h_0^{-1}\circ h_1\,.
	\label{eq:8.1}
	\end{equation}

One remarkable consequence of this result, perhaps, are the low pole distances $\delta$ and face counts $\eta$ of the (smaller) Western hemispheres which we encounter in our examples.
The absence of antipodal poles, $\delta =2$, in Sturm octahedral complexes was a first indication.
Similarly, $\text{max } \delta = 1$ with $\text{max } \eta = 2$ for the 20-faced icosahedron of edge diameter $\vartheta =3$, and $\text{max } \delta =2$ with $\text{max } \eta =2$ for the 12-faced dodecahedron of edge diameter $\vartheta =5$ are surprising.
Trivial isotropies $\kappa$ and $\rho \kappa$ are impossible, automatically, because they swap hemispheres and therefore require equal hemisphere face counts.

One reason for this asymmetric imbalance became apparent in corollary~\ref{cor:5.2}.
For face counts $\eta >1$, the four poles $w_\pm^\iota$ of the dual cores $\mathbf{W}^*$ and $\mathbf{E}^*$ are tightly bound into a short 4-cycle which consists of segments of the polar circles $\partial \mathbf{N}^*,\ \partial \mathbf{S}^*$, and two disjoint single-edge polar bridges between them.
To avoid this difficulty, we chose $\delta = \eta =1$ in \cite{firo14} to obtain some Sturmian signed hemisphere decomposition for any prescribed regular 2-sphere complex $\mathcal{C}^2 = S^2$.

Beyond the closure $\mathcal{C}_f = \overline{c}_\mathcal{O}$ of a single 3-cell, we may aim to describe all 3-dimensional Sturmian Thom-Smale dynamic complexes of maximal cell dimension three.
Even in the presence of a single 3-cell this allows for one-dimensional ``spikes'' or two-dimensional ``balconies''.
Specific examples already arise for $N=9$ equilibria and have been described in \cite{fi94}.

A more interesting example involves $N=15$ equilibria and arises from the ``\emph{Snoopy bun}'' cell complex of fig.~\ref{fig:6.3}, example~19.
See fig.~\ref{fig:8.1}(a), where we have swapped the hemispheres $\mathbf{E},\ \mathbf{W}$ and taken the single face hemisphere $\mathbf{E}$ as the exterior.
We call $\text{clos } \mathbf{W}$ the \emph{Snoopy disk} on top of the bun $c_\mathcal{O}$.
Examples~13, 14 and 24 of fig.~\ref{fig:6.3} are other hemisphere decompositions of the same Snoopy bun cell complex with 13~equilibria.
See also \cite[fig.~\ref{fig:5.2}]{firo3d-1}.

\begin{figure}[t!]
\centering \includegraphics[width=\textwidth]{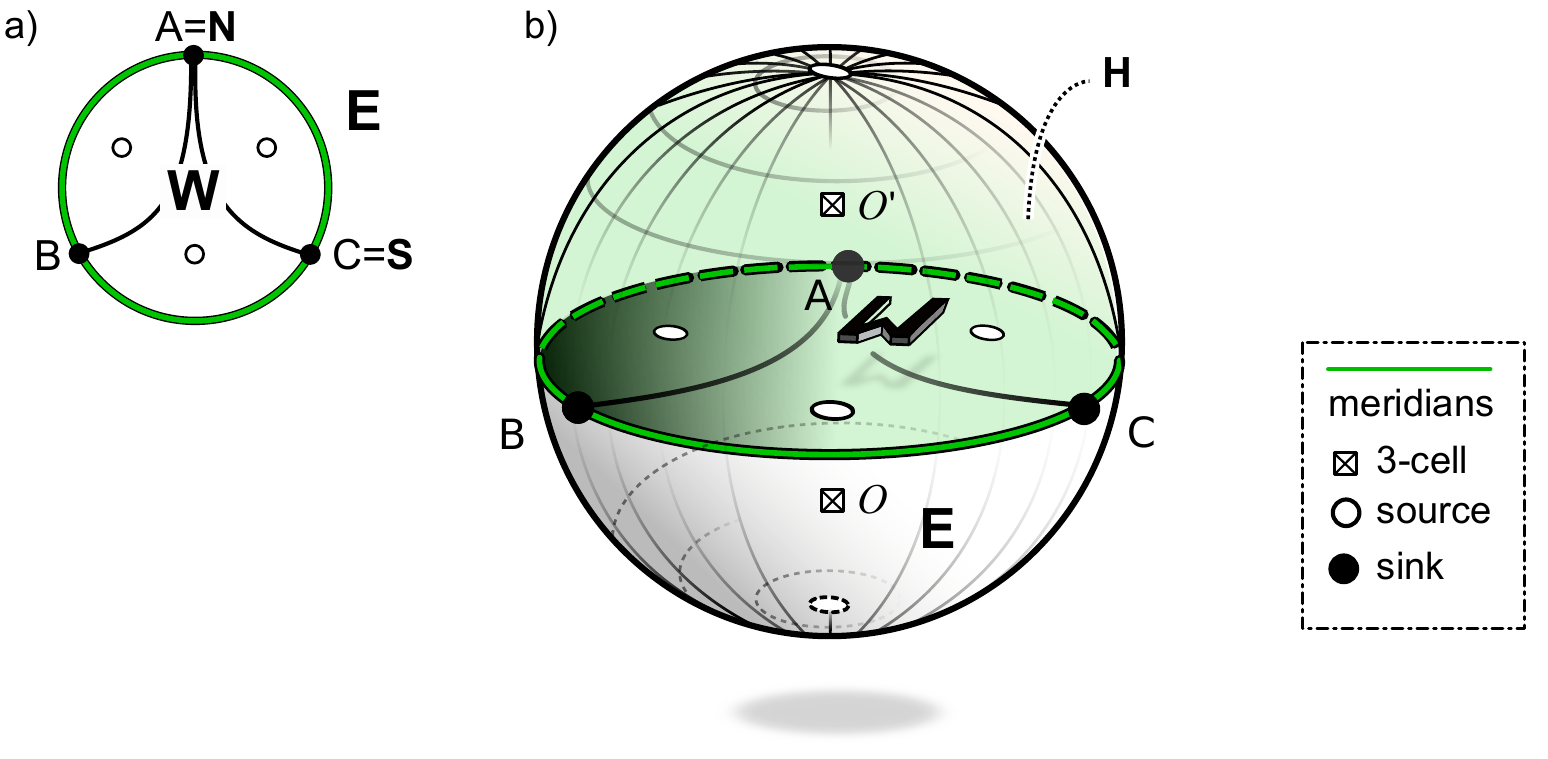}
\caption{\emph{
(a) The Snoopy bun 3-ball Sturm attractor with $N=13$ equilibria.
See fig.~\ref{fig:6.3} and table~\ref{tbl:6.6}, cases 13, 14, 19, 24 for inequivalent realizations.
(b) The Snoopy burger with an additional 3-cell bun $c_{\mathcal{O}}$, and hemisphere $\mathbf{H}$ packed on top.
This regular cell complex of dimension 3, with two adjacent 3-balls sharing 3 faces, is \emph{not} a Sturm dynamic complex.
}}
\label{fig:8.1}
\end{figure}

Let us add two more equilibria, to reach $N=15$.
We simply glue a second 3-cell $c_{\mathcal{O}'}$, to the other side of the equatorial Western disk of three faces, on top, and top off $c_{\mathcal{O}'}$ with a second 2-gon disk $\mathbf{H}$ as an upper hemisphere.
Here $\text{clos } \mathbf{H}$ shares the green meridian circle $\partial \mathbf{H} = \partial \mathbf{W} = \partial \mathbf{E}$ with, both, the lower hemisphere 2-disk $\text{clos } \mathbf{E}$ and the equatorial mid-plane Snoopy disk $\text{clos } \mathbf{W}$.
We call the resulting signed hemisphere complex of two 3-cells a \emph{Snoopy burger}.
See fig.~\ref{fig:8.1}(b).

We claim that the snoopy burger is \emph{not} a Sturm dynamic complex.
Indeed, the faces of the Snoopy disk $\mathbf{W}$ are reached from $\mathcal{O}$ and $\mathcal{O}'$, heteroclinically, from opposite incoming sides, tangent to the third eigenfunction $\pm \varphi_2$.
Therefore the same equatorial 3-face disk $\mathbf{W}$ must play the role of opposite hemispheres in the 3-balls $\text{clos } c_\mathcal{O}$ and $\text{clos } c_{\mathcal{O}'}$, respectively.
Only two of the three $i=0$ sink equilibria $A, \ B,\ C$ on the (green) shared boundary can be poles.
Any interior edge terminating at the third equilibrium thus has to be directed, both, towards the meridian boundary for $\text{clos } c_\mathcal{O}$, and away from that same meridian boundary for $\text{clos } c_{\mathcal{O}'}$.
This conflict prevents any Sturm realization of the Snoopy burger.

So, how about dimensions four and higher?
Already the Snoopy example, say, embedded into the 3-sphere boundary of a 4-cell cautions us to proceed with care.
In principle, at least, the general recipe of \cite{firo3d-2} for the construction of SZS-pairs $(h_0,h_1)$ extends to arbitrary signed hemisphere complexes \cite{firo18}.
A viable and complete geometric description, however, as we have presented for 3-balls here, is not available at this date.
We therefore conclude with a few examples. 

The $m$-dimensional \emph{Chafee-Infante global attractor} $\mathcal{A}_{\text{CI}}^m$ arises from PDE \eqref{eq:1.1} for cubic nonlinearities $f(u) = \lambda u(1-u^2)$.
Consider $\mathcal{O}$:= $0$ and observe $i(\mathcal{O}) = m \geq 1$ for $(m-1)^2 < \lambda / \pi^2 < m^2$.
The $2m$ remaining equilibria $v_\pm^j$ are characterized by $z(v_\pm^j - \mathcal{O}) = j_\pm$, all hyperbolic.
The Thom-Smale dynamic complex of $\mathcal{A}_{\text{CI}}^m = \text{clos } W^u(\mathcal{O})$ consists of the single $m$-cell $W^u (\mathcal{O})$ and the $m$-cell boundary $\partial W^u(\mathcal{O})$.
The hemisphere decomposition is simply the remaining Thom-Smale dynamic decomposition
	\begin{equation}
	\Sigma_\pm^j (\mathcal{O})= W^u (v_\pm^j)\,,
	\label{eq:4.5}
	\end{equation}
$0 \leq j < m = i(\mathcal{O})$, in the Chafee-Infante case.
See also \cite{chin74, he81, he85}.
The Chafee-Infante attractor $\mathcal{A}_{\text{CI}}^m$ is the $m$-dimensional Sturm attractor with the smallest possible number $N = 2m+1$ of equilibria.
Equivalently, among all Sturm attractors with $N = 2m+1$ equilibria, it possesses the largest possible dimension $m$.
Interestingly the dynamics on each closed hemisphere $\text{clos } \Sigma_\pm^j$ is itself $C^0$ orbit equivalent to the Chafee-Infante dynamics on $\mathcal{A}_{\text{CI}}^j$.
In section \ref{sec4}, for example, the Chafee-Infante 3-ball $\mathcal{A}_{\text{CI}}^3$ arose as a face lift, or an equivalent suspension, of $\mathcal{C}_0 =\mathcal{A}_{\text{CI}}^2$, alias the (1,1)-gon, by itself.

The double spiral meander $\mathcal{M}_{\text{CI}}^m$ of $\mathcal{A}_{\text{CI}}^m$ with $N=2m+1$ equilibria looks as follows.
It consists of $m$ nested upper arcs, above the horizontal $h_1$-axis, joining horizontally labeled equilibria $j$ and $2m+1-j$ in pairs, for $j=1, \ldots , m$.
Another $m$ nested lower arcs, joining equilibria $j+1$ and $2m+2-j$ in pairs, complete the meander.
This construction arises by successive meander suspension or, equivalently, by successive pitchfork bifurcation of the most unstable, central equilibrium.

Without proof we state how to obtain an $m$-simplex $\mathbb{S}^m$ with $N=2^{m+1}-1$ equilibria.
Note the 1-edge interval $\mathbb{S}^1= \mathcal{A}_{\text{CI}}^1$, the filled 3-gon $\mathbb{S}^2$, and the solid tetrahedron $\mathbb{S}^3=\mathbb{T}$ of fig.~\ref{fig:7.2.2}(c).
Above the horizontal axis, we keep a single nested sequence of $\tfrac{1}{2} (N-1)=2^m-1$ upper arcs, analogously to the Chafee-Infante case.
Below the axis, we position nests of $1,2,\ldots ,2^{m-1}$ arcs next to each other, starting with the single lower arc from 2 to 3.
This defines a meander $\mathcal{M}_{\mathbb{S}}^m$ for the $m$-simplex.
Of course, the pole distance is $\delta = \vartheta=1$.
The number of $(m-1)$-cells in the hemispheres $\Sigma_\pm^{m-1}$ are $\eta = [(m+1)/2]$ and $[(m+2)/2]$, respectively, each an $(m-1)$-simplex $\mathbb{S}^{m-1}$ itself.
Alas, there are many other Sturm realizations of the $m$-simplex $\mathbb{S}^m$.

A similar construction provides Sturm hypercubes $\mathbb{H}^m= \mathcal{A}_{\text{CI}}^1 \times \ldots \times \mathcal{A}_{\text{CI}}^1$ of any dimension $m$.
Analogously to fig.~\ref{fig:7.4.1} we place nests of $1,\ 3,\ 3^2,\ \ldots , 3^{m-1}$ lower arcs, left to right, below the axis, starting from the second equilibrium.
Above, we start at the first of the $3^m$ equilibria and reverse the nest sizes.
This places nests of $3^{m-1}, \ldots , 3^2,\ 3,\ 1$ upper arcs, left to right, above the horizontal axis.
The pole distance $\delta = \vartheta=m$ and the count $\eta =m$ of $(m-1)$-cells $\mathbb{H}^{m-1}$, identically in each hemisphere, are both maximally possible.
Again there are many other Sturm realizations with lower $\delta,\ \eta$.

\begin{table}[t!]
\centering \includegraphics[width=\textwidth]{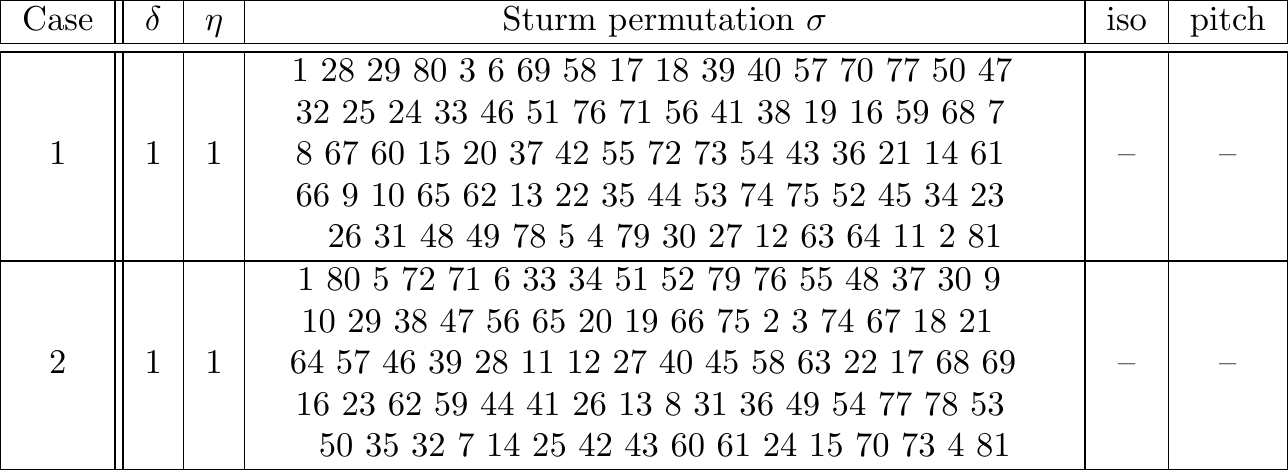}
\caption{\emph{
Two ad-hoc examples of Sturm permutations which lead to four-dimensional solid octahedra $\mathbb{O}^4$ with 81 equilibria and 16 three-dimensional solid tetrahedra $\mathbb{T}$ on the bounding 3-sphere $S^3$.
}}
\label{tbl:8.1}
\end{table}

For $m$-dimensional octahedra $\mathbb{O}^m$, i.e. the hypercube duals, also known as the (solid) $m$-orthoplex or the convex hull of the $2m$ $\pm$unit vectors in $\mathbb{R}^m$, we did not find such a series of Sturm realizations beyond $m=4$, so far.
One reason may be the strong asymmetry induced by small pole distances $\delta$ and the asymmetric counts $\eta$ of $(m-1)$-cells in the two hemispheres $\Sigma_\pm^{m-1}$.
We are only aware of two ad-hoc 4-dimensional Sturm examples of $\mathbb{O}^4$, with $3^m=81$ equilibria, $2^m=16$ tetrahedral 3-cells, and minimal $\delta = \eta =1$.
We conclude with their Sturm permutations, in table~\ref{tbl:8.1}, without further discussion.


\end{document}